\documentclass[reqno]{amsart}
\usepackage{amsmath,  amssymb, amstext, mathrsfs, cite}
\usepackage{bigints,bm}
\usepackage{enumitem}
\makeatletter
\newcommand*\rel@kern[1]{\kern#1\dimexpr\macc@kerna}
\newcommand*\widebar[1]{%
\begingroup
\def\mathaccent##1##2{%
\rel@kern{0.8}%
\overline{\rel@kern{-0.8}\macc@nucleus\rel@kern{0.2}}%
\rel@kern{-0.2}%
}%
\macc@depth\@ne
\let\math@bgroup\@empty \let\math@egroup\macc@set@skewchar
\mathsurround\z@ \frozen@everymath{\mathgroup\macc@group\relax}%
\macc@set@skewchar\relax
\let\mathaccentV\macc@nested@a
\macc@nested@a\relax111{#1}%
\endgroup
}
\makeatother

\usepackage[left=3cm,right=3cm,top=2cm,bottom=2cm,includeheadfoot]{geometry}
\usepackage{hyperref,xcolor}
\hypersetup{
pdfborder={0 0 0},
colorlinks,
}

\usepackage{enumitem}
\usepackage{amsfonts}
\usepackage{amsthm}

\usepackage{amsfonts}
\usepackage{amsthm}

\DeclareMathOperator*{\essinf}{ess\,inf}
\DeclareMathOperator*{\loc}{loc}

\DeclareMathOperator*{\spp}{supp}
\newcommand*\diff{\mathop{}\!\mathrm{d}}

\newcommand{\norm}[1]{|\!|#1|\!|}

\newcommand{\round}[1]{\left(#1\right)}

\newcommand{\scal}[1]{\left\langle#1\right\rangle}

\newcommand{\R}{{\mathbb R}}
\newcommand{\N}{{\mathbb N}}
\newcommand{\RN}{\mathbb{R}^N}
\newcommand{\Lp}[1]{L^{#1}(\Omega)}

\newcommand{\Wp}[1]{W^{1,#1}(\Omega)}

\newcommand{\close}{\overline{\Omega}}

\newcommand{\eps}{\varepsilon}

\newcommand{\into}{\int_{\Omega}}

\newcommand{\Assg}[1]{\textup{(g)}}

\newcommand{\calbf}[1]{\bm{\mathcal{#1}}}

\begin{document}
\title [Critical double phase problems in $\RN$]
{On critical double phase problems in $\mathbb{R}^N$ involving variable exponents}

\author[H.H.Ha]{Hoang Hai Ha}
\address{Hoang Hai Ha \newline
	Department of Mathematics, Faculty of Applied Science , Ho Chi Minh City University of Technology (HCMUT), 268 Ly Thuong Kiet Street, District 10, Ho Chi Minh City, Viet Nam \newline
	Vietnam National University Ho Chi Minh City, Linh Trung Ward, Ho Chi Minh City,  Viet Nam}
\email{hoanghaiha@hcmut.edu.vn }

\author[K. Ho]{Ky Ho}
\address{Ky Ho \newline 
	Institute of Applied Mathematics, University of Economics Ho Chi Minh City, 59C, Nguyen Dinh Chieu Street, District 3, Ho Chi Minh City, Viet Nam}
\email{kyhn@ueh.edu.vn}


\begin{abstract}
	We establish a Lions-type concentration-compactness principle and its variant at infinity for Musielak-Orlicz-Sobolev spaces associated with a double phase operator with variable exponents. Based on these principles, we demonstrate the existence and concentration of solutions for a class of critical double phase equations of Schr\"odinger type  in $\RN$ involving variable exponents with various types of potentials. Our growth condition is more appropriately suited compared to the existing works. 
\end{abstract}

\maketitle
\numberwithin{equation}{section}
\newtheorem{theorem}{Theorem}[section]
\newtheorem{lemma}[theorem]{Lemma}
\newtheorem{proposition}[theorem]{Proposition}
\newtheorem{corollary}[theorem]{Corollary}
\newtheorem{definition}[theorem]{Definition}
\newtheorem{example}[theorem]{Example}
\newtheorem{remark}[theorem]{Remark}
\allowdisplaybreaks


\section{Introduction}\label{Intro}



Given $N\geq 2$, this paper is concerned with problems in $\RN$ associated with a double phase operator with variable exponents given by
\begin{eqnarray}\label{oper}
	-\operatorname{div}\mathcal{A}(x,\nabla u)+ V(x)A(x,u),
\end{eqnarray}
where $\mathcal{A}: \mathbb{R}^N \times  \RN \to \RN  $, $A: \mathbb{R}^N \times  \R \to \R $ are defined as
\begin{equation}\label{A}
	\mathcal{A}(x,\xi): = |\xi|^{p(x)-2}\xi + a(x) |\xi|^{q(x)-2}\xi,\ A(x,t):=|t|^{p(x)-2}t+a(x)|t|^{q(x)-2}t,
\end{equation}
with  $a, p, q \in C^{0,1}(\mathbb{R}^N)$ such that $a(\cdot) \geq 0$, $1<p(\cdot)<q(\cdot)<N$, and $0\leq V(\cdot)\in L^1_{\loc}(\RN)$.

The main characteristic of the operator given in \eqref{oper} is that its behavior switches continuously between two sets $\{x \in \RN: \ a(x)=0\}$ and $\{x \in \RN: \ a(x)>0\}$. This is the reason why it is called double phase. 

In the past decades, there has been a growing interest in studying problems related to the  double phase operator 
\begin{equation}\label{d.p.cons}
	\operatorname{div}\left(|\nabla u|^{p-2}\nabla u+a(x)|\nabla u|^{q-2}\nabla u\right)
\end{equation}
and the associated energy integral
\begin{equation}\label{dp.f}
	I(u)=\int_\Omega \left(\frac{|\nabla u|^p}{p} + a(x)\frac{|\nabla u|^q}{q}\right)  \diff x,
\end{equation} 
where $\Omega$ is a domain in $\RN$. The functional $I$ in \eqref{dp.f} was originally introduced by Zhikov \cite{Zhikov-1995,Zhikov1997,ZKO} in the setting of homogenization of strongly anisotropic materials. The differential operator and the energy functional given in \eqref{d.p.cons} and \eqref{dp.f}, respectively, appear in several physical applications, for instance, in transonic flows \cite{Ba-Ra-Re},  quantum physic \cite{Ben2000}, reaction-diffusion systems \cite{Che2005} and non-Newtonian fluid \cite{LD18}.

\vskip5pt
The functional $I$ belongs to a class of integral functionals with an integrand $f$ satisfying a non-standard growth condition
\begin{equation*}
	C_1|\xi|^{p}\leq f(x,\xi)\leq C_2(1+|\xi|^q)\quad \text{for a.a. } x\in\Omega\ \ \text{and for all }\xi\in\R^N,
\end{equation*} 
introduced by Marcellini in \cite{Mar89b}. The regularity of local minimizers for such energy functionals has garnered significant interest from many authors, see e.g., \cite{Byun-Oh-2020,CM15a,CM15b,ELM,FMM,RT20} and the references therein. 

\vskip5pt
Lately, there have been a number of studies focusing on the existence results of elliptic problems governed by the operator \eqref{d.p.cons} when $\Omega$ is a bounded Lipschitz domain. We refer to 
\cite{CS16,FMPV23,FP20,GP19,GW21,Liu-Dai-2018,PS-CCM18} for some recent results on this topic. Very recently, Crespo-Blanco et al. \cite{crespo2022new} studied problems involving the operator $\operatorname{div}\mathcal{A}(x,\nabla u)$ in a bounded Lipschitz domain $\Omega$. Based on results obtained in \cite{fan2012,Mus}, they explored the fundamental properties of the Musielak-Orlicz spaces $L^{\mathcal{H}}(\Omega)$ and Musielak-Orlicz-Sobolev spaces $W^{1,\mathcal{H}}(\Omega)$, where $$\mathcal{H}(x,t):=t^{p(x)}+a(x) t^{q(x)}\quad\text{for}\ (x,t) \in \Omega \times [0,\infty).$$ 
Some followed up papers studying the double phase problems with variable exponents in a bounded domain can be referred to \cite{HaHo2023,HW2022,HW2022-2,KKOZ22,Pa-Ve-Wi24,Ve-Wi23} and the references therein. On the contrast, existence results for double phase problems in unbounded domains have been less discussed, even for the constant case of exponents. To date, as far as we aware, publications on this topic have been limited to some references, see \cite{Ar-Dw,ge2022quasilinear, Ge-Yu24,Liu-Dai-2020, Li-Wi22, Ste}. It is likely due to the lack of compactness arising in connection with variational methods. Motivated by the gap research on unbounded domains, in this paper we study problems in $\RN$ that involve the operator given in \eqref{oper}. We believe that our work is a significant contribution  to this research direction.

\vskip5pt
It is noteworthy that problems considered in most of existing works involve reaction terms whose growth does not surpass the threshold $p^\ast(\cdot)$ (here and in the sequel, $h^*(\cdot):=\frac{Nh(\cdot)}{N-h(\cdot)}$ for $h(\cdot)<N$). In \cite{Ar-Dw}, the authors dealt with a nonlinearity including an arbitrary growth term of the form $-|u|^{r(x)-2}u$, but this negative term was actually treated as a component of the main operator. With a growth not exceeding $p^\ast(\cdot)$, one easily obtains necessary embeddings in connection with variational methods since $W^{1,\mathcal{H}}(\Omega)\hookrightarrow W^{1,p(\cdot)}(\Omega)\hookrightarrow L^{p^\ast(\cdot)}(\Omega)$ (see Section~\ref{Pre} for the details). Nevertheless, $|t|^{p^*(x)}$ obviously does not capture the behavior for the critical term of Sobolev-type embedding on $\Wp{\mathcal{H}}$, which continuously switches order between $p^\ast(\cdot)$ and $q^\ast(\cdot)$.  In fact, the general definition of the Sobolev conjugate function of $\mathcal{H}$ is given implicitly, see \cite{crespo2022new}, but it is a challenge to set up results on Sobolev spaces with this implicit form. Therefore, Ho-Winkert \cite{HW2022} proposed the function
\begin{equation*}
	\mathcal{G^\ast}(x,t):=t^{p^\ast(x)} +   a(x)^{\frac{q^\ast(x)}{q(x)}}t^{q^\ast(x)}\quad \text{for} \ (x,t) \in \Omega \times [0,\infty),
\end{equation*}
as a critical term that captures the behavior of the Sobolev conjugate function of $\mathcal{H}$, and showed that
\begin{equation}\label{I-CE}
	W^{1,\mathcal{H}}(\Omega)\hookrightarrow L^{\mathcal{G^\ast}}(\Omega),
\end{equation}
when $\Omega$ is a bounded domain.
Furthermore, they also discussed the optimality of $\mathcal{G}^\ast$ among those  $\Psi$ of the form $\Psi(x,t)=|t|^{r(x)}+a(x)^{\frac{s(x)}{q(x)}}|t|^{s(x)}$ such that $W^{1,\mathcal{H}}(\Omega)\hookrightarrow  L^{\Psi}(\Omega)$. With the aid of the critical embeding \eqref{I-CE}, in our recent work \cite{HaHo2023}, we established multiplicity results for a class of double phase problems featuring the critical term $\mathcal{G^\ast}$ in a bounded domain subject to the Dirichlet boundary condition. Independently, a similar critical growth for double phase problems with constant exponents was considered in \cite{Co-Pe23}. These works seem to be the only existence results for double phase problems with growths surpassing $p^\ast(\cdot)$. Once again, we emphasize that results in \cite{Co-Pe23,HaHo2023} were obtained for problems in a bounded domain. If one wants to develop such results for critical double phase problems in unbounded domains, particularly using variational methods, it would be  important to set up the embedding \eqref{I-CE} for an arbitrary domain. For this reason, our first objective in this paper is to prove \eqref{I-CE} for an  open domain in $\RN$ and then explore function spaces associated with the double phase operator given in \eqref{oper}. This setting will   enable us to study double phase problems  in $\RN$ not only with a wider range of nonlinearities but also with a more appropriately suited critical term  to the main operator.

\vskip5pt
Critical problems are originated in geometry and physics, and the most notorious example is Yamabe's problem \cite{Aub76}. The lack of compactness arising in connection with the variational approach due to the unboundedness of $\RN$ and the critical term makes the critical problems in $\RN$ delicate and interesting.  In order to address such problems with constant exponents, the concentration-compactness principle (abbreviated by CCP) by Lions \cite{Lions} and its variant introduced in \cite{Chabrowski} have been used as a very effective tool. Using these CCPs, one can demonstrate the precompactness of Palais-Smale sequences, a crucial step in proving existence results via variational methods. These CCPs have been extended to the several types of Sobolev spaces in order to address problems involving an extension of the $p$-Laplacian and critical growth, see e.g., \cite{Bonder, BSS, fu2010multiple, HK2021, Pa-Pi14}. A Lions-type CCP for $W^{1,\mathcal{H}}(\Omega)$ with a bounded Lipschitz domain $\Omega$, in connection with the critical embedding \eqref{I-CE}, was established by us in \cite{HaHo2023}.  Motivated by this, our second objective in this paper is to establish a Lions-type CCP and its variant at infinity for  Musielak-Orlicz-Sobolev spaces associated with the double phase operator given in \eqref{oper}. This will be instrumental in handling double phase problems involving critical growths in the entire space $\RN$ in the future.

\vskip5pt
Upon acquiring these CCPs, we employ them to investigate the existence and concentration of solutions to equations of Schr\"odinger-type involving the double phase operator \eqref{oper} and a critical growth, which is our last main objective in this paper. Precisely, let us consider critical double phase equations of Schr\"odinger-type as follows:
\begin{equation}\label{eq1}
	-\operatorname{div}\mathcal{A}(x,\nabla u)+ \lambda V(x)A(x,u) = f(x,u)+\theta B(x,u) \quad \text{in}~ \RN,
\end{equation}
where $f(x,u)$ exhibits a subcritical growth while $B(x,u)$ is a generalization of the critical term $c_1|u|^{p^*(x)-2}u+c_2a(x)^{\frac{q^*(x)}{q(x)}}|u|^{q^*(x)-2}u,$ with   $c_1>0$, $c_2\geq 0$,  and $\lambda,\theta$ are positive real parameters.

\vskip5pt
Equations of Schr\"odinger-type have long been an interest research direction to mathematicians due to their profound applications in the quantum physic. The class of equations related to problem \eqref{eq1} can be seen as an extension. Our first question is regarding the existence result. In comparison to aforementioned double phase papers in $\RN$, which only dealt with a potential $V$ satisfying $\inf_{x\in\RN} V(x)>0$, our novelty is that we can handle problems with $V$ holding $\inf_{x\in\RN} V(x)=0$ . Our second question is about a phenomena called solution concentration of problem \eqref{eq1}. Our interest stems from the work \cite{Bar.Wang2001} by Bartsch et al. and the followed up papers it inspired, see e.g.,\cite{sun2014ground, sun.wu.2020, Xie.Ma2015, Yang.Liu 2016}  and the references therein. Specifically, in \cite{sun2014ground}  the authors considered the following model:
\begin{equation} \label{sun1}
	-\left(a\int_{\R^N}|\nabla u|^2\diff x+b\right)\Delta u + \lambda V(x) u = f(x,u) ~\text{in} ~ \RN,
\end{equation}
where $a,b>0$, and $0\leq V(\cdot)\in C(\RN)$ satisfies the following conditions:
\begin{enumerate}
	\item[$(i)$]  there exists $K_0>0$ such that the set $\left\{x \in \mathbb{R}^{N}: V(x)<K_0\right\}$ is nonempty and has finite measure;
	\item[$(ii)$]  $\Omega_V:= \operatorname{int} \{V^{-1}(0)\}$ is a nonempty bounded smooth domain, and $\overline{\Omega}_V= V^{-1}(0)$.
\end{enumerate}
They showed that under some suitable conditions of $f$, when $\lambda\to\infty$, equation \eqref{sun1} has a sequence of weak solutions $\{u_{\lambda_n}\}_{n \in \N}$ such that $u_{\lambda_n} \to \bar{u}$ in $L^r(\R^N)$ for every $r\in [2,2^*)$,  where $\bar{u}$ is a weak solution of  the  limit problem
\begin{equation*}\label{sun2}
	\begin{cases}
		-\left(a\int_{\R^N}|\nabla \bar{u}|^2\diff x+b\right)\Delta \bar{u}=f(x, \bar{u}) & \quad \text { in } \Omega_V,\\
		\bar{u}=0 & \quad \text { on }  \partial \Omega_V.
	\end{cases}
\end{equation*}
This phenomena, known as the concentration of solutions near the bottom of the potential well, has attracted the great attention from both mathematicians and physicists recently. We stress that all of the referenced papers focused on problems governed by the Laplace operator. This raises a natural question whether the concentration phenomena still holds  for classes of equations controlled by another type of operators in general or the double phase operator  in particular. Finally,  it is worth pointing out that we aim to deal with a wider class of potentials compared to the referred papers on the concentration, that is our potential may be a constant.

\vskip5pt
We organize the paper as follows.  In  Section~\ref{Pre}, we explore the  Musielak-Orlicz-Sobolev spaces, which play as the solution space for problem \eqref{eq1}. The critical embedding (Theorem \ref{T.CE1}) is the main result of this section. In Section~\ref{ccp}, we  establish two CCPs  (Theorems \ref{Theo.CCP1} and \ref{Theo.CCP.inf}) that extend the CCPs in \cite{Lions} and \cite{Chabrowski} to the aformentioned solution spaces. Section~\ref{existence} is devoted to results concerning the existence and concentration of solutions to problem \eqref{eq1} (Theorems \ref{Theo.Superlinear} and \ref{Theo.Superlinear2}).

\vskip5pt
\noindent\textbf{Notation.} There are some notations that are used frequently throughout the paper, and for the reader's convenience, we make a list  below:
\begin{itemize}[label={\tiny\raisebox{0.8ex}{\textbullet}}]
\item $\Omega$ represents an open domain in $\RN$, $\Omega^c$ means the complement of $\Omega$ in $\RN$.
\item $m^-:=\inf _{x \in \overline{\Omega}} m(x)$ and $m^+:=\sup _{x \in \overline{\Omega}} m(x)$ for $m \in C\left(\overline{\Omega}\right)$.
\item $C_+(\overline{\Omega}):=\left\{m \in C\left(\overline{\Omega}\right): 1<m^-\leq m^+<\infty\right\}.$
\item $r'(\cdot):=\frac{r(\cdot)}{r(\cdot)-1}$ for $r \in C_+(\overline{\Omega})$.
\item For $f,g\in C(\overline{\Omega})$, we write $f(\cdot) \ll g(\cdot)$ if $\displaystyle \inf_{x \in \overline{\Omega}} (g-f)(x) > 0$. 
\item $M(\Omega)$ denotes the space of all Lebesgue measurable functions $u:\Omega\to\R$.
\item For a measurable subset $E\subset\RN$, $|E|$ stands for the Lebesgue measure of $E$.
\item   $X \hookrightarrow Y$ indicates that the space $X$ is embedded continuously into the space $Y$, while $X \hookrightarrow\hookrightarrow Y$ means that $X$ is embedded compactly into $Y$.
\item $X^\ast$ and $\langle \cdot,\cdot \rangle$  respectively denote the dual space of a normed space $X$ and the duality pairing between $X^\ast$ and $X$.
\item $u_n \to u$ (resp., $u_n \rightharpoonup u$, $u_n \stackrel{\ast}{\rightharpoonup} u$) in $X$ means $u_n$ converges to $u$ strongly (resp., weakly, weakly-$\ast$) in a normed space $X$ as $n \to \infty$.
\item $B_\epsilon(x)$ denotes  a ball in $\RN$ centered at $x$ with radius $\epsilon$; if $x$ is the origin, we write shortly $B_\epsilon$.

\item  The letters $C,\,C_i$ stand for generic positive constants that may be different through lines and depend only on the data.
\end{itemize}



\section{Variable exponent spaces} \label{Pre}
In this section, we provide some properties of the function spaces essential to our approach. Throughout this section, let $\Omega$ be an open domain in $\RN$ with the cone property.

\subsection{Variable exponent Lebesgue spaces}${}$

\vskip5pt

Let $z \in C_+(\overline\Omega) $  and $\mu$ be a $\sigma$-finite and complete measure  in $\overline{\Omega}$. We define the variable exponent Lebesgue space $L_\mu^{z(\cdot)}(\Omega)$ as
$$
L_\mu^{z(\cdot)}(\Omega) := \left \{ u : \Omega\to\mathbb{R}\  \hbox{is}\  \mu\text{-measurable},\ \int_\Omega |u(x)|^{z(x)} \;\diff\mu < \infty \right \},
$$
endowed with the Luxemburg norm
$$
\|u\|_{L_\mu^{z(\cdot)}(\Omega)}:=\inf\left\{\lambda >0:
\int_\Omega 
\Big|\frac{u(x)}{\lambda}\Big|^{z(x)}\;\diff\mu\le1\right\}.
$$
Then, $L_\mu^{z(\cdot)}(\Omega)$ is a separable and uniformly convex Banach space (see \cite{Diening}). When $\mu$ is the Lebesgue measure, we write $L^{z(\cdot)}(\Omega)$ and $\|\cdot\|_{L^{z(\cdot)}(\Omega)}$ in place of $L_\mu^{z(\cdot)}(\Omega)$  and $\|\cdot\|_{L_\mu^{z(\cdot)}(\Omega)}$, respectively.


\begin{proposition}[\cite{Diening}] \label{prop.Holder}
For any $u\in L_{\mu}^{z(\cdot)}(\Omega)$ and $v\in L_{\mu}^{z'(\cdot)}(\Omega)$, the following H\"older-type inequality holds
\begin{equation*}
	\left|\int_\Omega uv\,\diff \mu\right|\leq\ 2 \|u\|_{L_\mu^{z(\cdot)}(\Omega)}\|v\|_{L_\mu^{z'(\cdot)}(\Omega)}.
\end{equation*}
\end{proposition}


\begin{proposition}[\cite{Diening}] \label{prop.norm-modular}
Let $u\in L_{\mu}^{z(\cdot) }(\Omega )$ and denote $\rho (u)=\int_{\Omega }| u(x)| ^{z(x) }\diff \mu$. It holds that
\begin{enumerate}
	\item[(i)] $ \|u\|_{L_{\mu}^{z(\cdot) }(\Omega )}<1$ $(resp. >1,=1)$
	if and only if \  $\rho (u) <1$ $(resp. >1,=1)$;
	
	\item[(ii)] if \  $ \|u\|_{L_{\mu}^{z(\cdot) }(\Omega )}>1,$ then  $ \|u\|^{z^{-}}_{L_{\mu}^{z(\cdot) }(\Omega )}\leq \rho (u) \leq  \|u\|_{L_{\mu}^{z(\cdot) }(\Omega )}^{z^{+}}$;
	\item[(iii)] if \ $ \|u\|_{L_{\mu}^{z(\cdot) }(\Omega )}<1,$ then $ \|u\|_{L_{\mu}^{z(\cdot) }(\Omega )}^{z^{+}}\leq \rho
	(u) \leq  \|u\|_{L_{\mu}^{z(\cdot) }(\Omega )}^{z^{-}}$.
\end{enumerate}
Consequently,
$$ \|u\|_{L_{\mu}^{z(\cdot) }(\Omega )}^{z^{-}}-1\leq \rho (u) \leq  \|u\|_{L_{\mu}^{z(\cdot) }(\Omega )}^{z^{+}}+1,\ \forall u\in L_\mu^{z(\cdot)}(\Omega ).$$
\end{proposition}

In the next sections, we will frequently make use of the following elementary inequalities involving variable exponents: 
\begin{equation}\label{young}
|ab|\leq \frac{1}{r(x)}\eps |a|^{r(x)}+\frac{r(x)-1}{r(x)}\eps^{-\frac{1}{r(x)-1}}|b|^{\frac{r(x)}{r(x)-1}}\leq \eps |a|^{r(x)}+\left(1+\eps^{-\frac{1}{r^--1}}\right)|b|^{\frac{r(x)}{r(x)-1}} 
\end{equation}
and
\begin{equation}\label{Ineq2a}
|a+b|^{r(x)} \leq (1+\eps) |a|^{r(x)}+ \left(1+\frac{1}{(1+\varepsilon)^{\frac{1}{r^-}}-1}\right)^{r^+} |b|^{r(x)} 
\end{equation}
for all $a,b\in\R$,  $\eps>0$, $x\in \overline{\Omega}$,  and $r\in C_+(\close)$.


\subsection{A class of Musielak-Orlicz-Sobolev spaces}${}$ \label{Subsec.M.O.S space}

\vskip5pt

Define $\Psi:\, \overline{\Omega}\times [0,\infty)\to [0,\infty)$ as 
\begin{align}\label{Psi}
\Psi(x,t):=w_1(x)t^{\alpha(x)}+w_2(x)t^{\beta(x)}
\quad\text{for } (x,t)\in \overline{\Omega}\times [0,\infty),
\end{align}
where $\alpha,\,\beta\in C_+(\close)$ with $\alpha(\cdot)<\beta(\cdot)$, $0 <  w_1(\cdot) \in \Lp{1}\cup L^\infty(\Omega)$ and $0 \leq w_2(\cdot) \in \Lp{1}\cup L^\infty(\Omega)$.

\noindent Define the modular $\rho_{\Psi}$ associated with $\Psi$ as
\begin{align*}
\rho_{\Psi}(u): =  \into \Psi (x,|u(x)|)\,\diff x .
\end{align*}
The corresponding Musielak-Orlicz space $\Lp{\Psi}$ is defined as 
\begin{align*}
L^{\Psi}(\Omega):=\left \{u\in M(\Omega):\,\rho_{\Psi}(u) < \infty \right\},
\end{align*}
endowed with the norm
\begin{align*}
\|u\|_{\Psi}:= \inf \left \{ \tau >0 :\, \rho_{\Psi}\left(\frac{u}{\tau}\right) \leq 1  \right \}.
\end{align*}
Then, $\Lp{\Psi}$ is a separable, uniformly convex and reflexive Banach space (see \cite[Theorems 3.3.7, 3.5.2 and 3.6.6]{HH.Book}). Moreover, the following relation between the modular $\rho_{\Psi}$ and the norm $\|\cdot\|_{\Psi}$ can be easily obtained in the same manner as \cite[Proof of Proposition 2.13]{crespo2022new}.
\begin{proposition}\label{prop.nor-mod.D}
Let $u, u_n\in\Lp{\Psi}$ ($n\in\N$). Then, the following assertions hold:
\begin{enumerate}
	\item[(i)]
	if $u\neq 0$, then $\|u\|_{\Psi}=\lambda$ if and only if $ \rho_{\Psi}(\frac{u}{\lambda})=1$;
	\item[(ii)]
	$\|u\|_{\Psi}<1$ (resp.\,$>1$, $=1$) if and only if $ \rho_{\Psi}(u)<1$ (resp.\,$>1$, $=1$);
	\item[(iii)]
	if $\|u\|_{\Psi}<1$, then $\|u\|_{\Psi}^{ \beta^+}\leqslant \rho_{\Psi}(u)\leqslant\|u\|_{\Psi}^{ \alpha^-}$;
	\item[(iv)]
	if $\|u\|_{\Psi}>1$, then $\|u\|_{\Psi}^{ \alpha^-}\leqslant \rho_{\Psi}(u)\leqslant\|u\|_{\Psi}^{ \beta^+}$;
	\item[(v)] $\|u_n\|_{\Psi}\to 0$ as $n\to \infty$ if and only if $\rho_{\Psi}(u_n)\to 0$ as $n\to \infty$.
\end{enumerate}
\end{proposition}
\begin{corollary}\label{Cor-n}
	For $r,s\in C(\overline{\Omega})$ satisfying $0<r(\cdot)\ll \alpha(\cdot)$ and $0<s(\cdot)\ll \beta(\cdot)$, one has
	$$\left\|w_1^{\frac{r}{\alpha}}|u|^{r}\right\|_{L^{\frac{\alpha(\cdot)}{r(\cdot)}}(\Omega)}\leq 1+\|u\|_{\Psi}^{r^+}\quad \text{and}\quad \left\|w_2^{\frac{s}{\beta}}|u|^{s}\right\|_{L^{\frac{\beta(\cdot)}{s(\cdot)}}(\Omega)}\leq 1+\|u\|_{\Psi}^{s^+},\quad \forall u\in L^{\Psi}(\Omega).$$
\end{corollary}
\begin{proof}
	Let $u\in L^{\Psi}(\Omega)$ and set $\lambda:=\left\|w_1^{\frac{r}{\alpha}}|u|^{r}\right\|_{L^{\frac{\alpha(\cdot)}{r(\cdot)}}(\Omega)}$. If  $\lambda>1$, then by Proposition~\ref{prop.nor-mod.D} one has
	$$1=\int_\Omega\left|\frac{w_1(x)^{\frac{r(x)}{\alpha(x)}}|u|^{r(x)}}{\lambda}\right|^{\frac{\alpha(x)}{r(x)}}\diff x\leq \int_\Omega w_1(x)\left|\frac{u}{\lambda^{\frac{1}{r^+}}}\right|^{\alpha(x)}\diff x\leq \rho_{\Psi}\left(\frac{u}{\lambda^{\frac{1}{r^+}}}\right).$$
	By Proposition~\ref{prop.nor-mod.D} again, one has $\lambda^{\frac{1}{r^+}}\leq \|u\|_{\Psi}.$ Thus, we obtain
	$$\left\|w_1^{\frac{r}{\alpha}}|u|^{r}\right\|_{L^{\frac{\alpha(\cdot)}{r(\cdot)}}(\Omega)}\leq 1+\|u\|_{\Psi}^{r^+},\quad \forall u\in L^{\Psi}(\Omega).$$
	The remaining inequality is proved in the same fashion.
\end{proof}
We have the following extension of the Brezis-Lieb Lemma to the Musielak-Orlicz spaces $L^{\Psi}(\Omega)$. 
\begin{lemma}\label{L.brezis-lieb}
Let $\{f_n\}_{n\in\N}$ be a bounded sequence in $L^{\Psi}(\Omega)$
and $f_n(x)\to f(x)$ for a.a. $x\in\Omega$. Then $f\in L^{\Psi}(\Omega)$ and
\begin{equation}\label{LeBL}
	\lim_{n\to\infty}\int_{\Omega} \Big|\Psi(x,|f_n|)
	-\Psi(x,|f_n-f|)-\Psi(x,|f|)\Big|\diff x=0.
\end{equation}
\end{lemma}	
\begin{proof}
	We have $f\in L^{\Psi}(\Omega)$ by Fatou's lemma. To show \eqref{LeBL}, it suffices to prove that  
	$$
	\lim_{n\to\infty}\int_{\Omega} \Big|w_1|f_n|^{\alpha(x)}
	-w_1|f_n-f|^{\alpha(x)}-w_1|f|^{\alpha(x)}\Big|\diff x=0$$
and
	$$
\lim_{n\to\infty}\int_{\Omega} \Big|w_2|f_n|^{\beta(x)}
-w_2|f_n-f|^{\beta(x)}-w_2|f|^{\beta(x)}\Big|\diff x=0.$$ The proof of these limits follows the lines of \cite[Proof of Lemma 1.32]{Willem}, so we omit it.
\end{proof}

The Musielak-Orlicz-Sobolev space  $W^{1,\Psi}(\Omega)$ is defined as
\begin{align*}
W^{1,\Psi}(\Omega):=\left\{u \in L^{\Psi}(\Omega):|\nabla u| \in L^{\Psi}(\Omega)\right\},
\end{align*}
endowed with the norm
\begin{align*}
\|u\|_{W^{1,\Psi}(\Omega)}:=\inf \left\{\tau>0: \hat{\rho}_\Psi \left(\frac{u}{\tau}\right) \leq 1\right\},
\end{align*}
where $\hat{\rho}_\Psi:W^{1,\Psi}(\Omega) \to \mathbb{R}$ is the modular defined as
\begin{align*}
\hat{\rho}_\Psi(u):=  \int_\Omega \big[\Psi(x,|\nabla u|)+\Psi(x,|u|)\big]\diff x.
\end{align*}
The space $W_0^{1,\Psi}(\Omega)$ is defined as the closure of $C_c^\infty(\Omega)$ in $W^{1,\Psi}(\Omega)$. Then, $W^{1,\Psi}(\Omega)$ and $W_0^{1,\Psi}(\Omega)$ are separable, uniformly convex and reflexive Banach spaces (see \cite[Theorem 6.1.4]{HH.Book}). Note that when $\Psi(x,t)=t^{\alpha(x)},$ the spaces $W^{1,\Psi}(\Omega)$ and $W_0^{1,\Psi}(\Omega)$ become the well-studied generalized Sobolev spaces $W^{1,\alpha(\cdot)}(\Omega)$ and $W_0^{1,\alpha(\cdot)}(\Omega)$, respectively (see \cite{Diening,fan2001sobolev,fanzhao2001}). 

The next proposition gives the relation between the modular $\hat{\rho}_\Psi$ and the norm $\|\cdot\|_{W^{1,\Psi}(\Omega)}$. We also omit the proof since it is similar to that of \cite[Proposition 2.14]{crespo2022new}.

\begin{proposition}\label{Prop_m-n_W(1,H)}  
For any $u \in W^{1,\Psi}(\Omega) $, we have
\begin{enumerate}
	\item[(i)]
	if $u\neq 0$, then $\|u\|_{W^{1,\Psi}(\Omega)}=\lambda$ if and only if $ \hat{\rho}_{\Psi}(\frac{u}{\lambda})=1$;
	\item[(ii)]
	$\|u\|_{ W^{1, \Psi}(\Omega)}<1$ (resp.\,$=1$, $>1$)  if and only if	$\hat{\rho}_ \Psi(u)<1$ (resp.\,$=1$, $>1)$;
	\item[(iii)]
	if $\|u\|_{ W^{1, \Psi}(\Omega)}<1 $, then $\|u\|^{\beta^+}_{ W^{1, \Psi}(\Omega)} \leq \hat{\rho}_ \Psi(u) \leq \|u\|^{\alpha^-}_{ W^{1, \Psi}(\Omega)}$;
	\item[(iv)]
	if $\|u\|_{ W^{1, \Psi}(\Omega)}>1 $, then $\|u\|^{\alpha^-}_{ W^{1, \Psi}(\Omega)} \leq \hat{\rho}_ \Psi(u) \leq \|u\|^{\beta^+}_{ W^{1, \Psi}(\Omega)}$.
\end{enumerate}
\end{proposition}


Next, we present a general domain version of the critical embedding for Musielak-Orlicz-Sobolev spaces associated with double phase functions  established in \cite{HW2022}.

For $(x,t)\in \overline{\Omega}\times [0,\infty)$, define
\begin{equation}\label{def_H}
\mathcal{H}(x,t):=t^{p(x)} +a(x)t^{q(x)} 
\end{equation}
and 
\begin{equation} \label{def_G*}
\mathcal{G}^\ast(x,t):= t^{p^\ast(x)} + a(x)^{\frac{q^\ast(x)}{q(x)}}t^{q^\ast(x)}.
\end{equation}
We make the following assumptions.
\begin{enumerate}
\item [$(\calbf{A})$] $N \geq 2$, $p,q \in C_+(\overline{\Omega}) \cap C^{0,1}(\overline{\Omega}),$    $p(x)<q(x)<N \,\text{ for all } x \in \overline{\Omega}$, 
$\displaystyle q(\cdot) \ll p(\cdot)\frac{N+1}{N},$ $0 \leq a(\cdot)\in L^\infty(\Omega) \cap C^{0,1}(\close)$.
\end{enumerate}
Note that the conditions on exponents $p,q$ in $(\calbf{A})$ imply
\begin{equation}\label{A10}
q(\cdot) \ll p^\ast(\cdot)
\end{equation}
and
\begin{equation}\label{A11}
\frac{Nq(\cdot)}{N-q(\cdot)+1} \ll p^\ast(\cdot).
\end{equation}

\begin{theorem} \label{T.CE1}
Let $(\calbf{A})$ be satisfied. Then the following embedding holds
\begin{align*}
	W^{1,\mathcal{H}}(\Omega) \hookrightarrow L^{\mathcal{G}^\ast}(\Omega).
\end{align*}
\end{theorem}	

\begin{proof}
We follow the idea of \cite[Lemma 2.1]{fan2001sobolev}. Throughout this proof, with the data given by $(\calbf{A})$ we denote by $C_i$ (resp. $C_i(t_0)$) a positive constant depending only on the data (resp. the data and $t_0$).

Clearly, the conclusion follows if we can show that
\begin{equation} \label{case3}
	\int_\Omega  \mathcal{G}^\ast(x,|u|)\diff x  \leq C_1 \left[\int_{\Omega}\mathcal{H}(x,|\nabla u|)\diff x +\int_{\Omega}\mathcal{H}(x, |u|)\diff x+1\right]^{{(q^\ast)}^+}, \quad \forall u \in W^{1,\mathcal{H}}(\Omega).
\end{equation}
We will prove \eqref{case3} by showing the following claims.
\medskip

\textbf{Claim  1:} \textit{\eqref{case3} holds for all $u \in W_c^{1,\mathcal{H}}(\Omega) \cap L^\infty(\Omega)$, where}
\[W^{1,\mathcal{H}}_{c}(\Omega):=\Big\{u \in W^{1,\mathcal{H}}(\Omega):\  \operatorname{supp} (u) \text { is compact}\Big\}.\]
To this end, let $u\in\left( W^{1,\mathcal{H}}_{c}(\Omega) \cap L^\infty(\Omega)\right)\setminus\{0\}$, and by using $|u|$ if necessary, we may assume that $u(\cdot) \geq 0$. We will obtain \eqref{case3} by making use of the classical  inequality: 
\begin{equation}\label{C-I}
			\|v\|_{L^{\frac{N}{N-1}}(\Omega)} \leq C_2 \int_{\Omega}\big(|\nabla v|+|v|\big)\diff x,\quad \forall v\in W^{1,1}(\Omega).
\end{equation} 
Set 
$$\beta:=\|u\|_{\mathcal{G}^\ast}\quad \text{and}\quad f(x):=\left[\mathcal{G}^\ast\left(x, \frac{u(x)}{\beta}\right)\right]^{\frac{N-1}{N}} \text { for a.a. } x \in \Omega.$$
Then it follows from Proposition~\ref{prop.nor-mod.D} that
\begin{equation}\label{int_G*}
	\int_{\Omega}\left[f(x)\right]^{\frac{N}{N-1}}\diff x=\int_{\Omega} \mathcal{G}^\ast \left(x, \frac{u(x)}{\beta}\right)\diff x = 1. 
\end{equation} 
Next, we verify that $f\in W^{1,1}(\Omega)$. It is clear that 
\begin{equation}\label{P.Thm2.6.f=0}
	f=|\nabla f|=0\quad \text{a.e. in } \{x \in \Omega:\ u(x)=0\}.
\end{equation} 
Set $\Omega_+:=\{x \in \Omega: u(x)>0\}$. Then, it holds that
\begin{align}\label{Nabla-f}
	\notag
	|\nabla f(x)| &\leq \frac{N-1}{N}\left(\mathcal{G}^\ast\left(x, \frac{u(x)}{\beta}\right)\right)^{-\frac{1}{N}} \left[ \,\left|\frac{\partial \mathcal{G}^\ast\left(x, \frac{u(x)}{\beta}\right)}{\partial t} \right|  \left|\frac{\nabla u}{\beta}\right| +  \sum_{i=1}^N\left|\frac{\partial \mathcal{G}^\ast\left(x, \frac{u(x)}{\beta}\right)}{\partial x_i} \right| \, \right]\\
	&\leq \frac{C_3}{\beta} I_1(x)+C_3I_2(x)\quad \text{for a.a. }x\in\Omega_+,
\end{align}					
where
$$I_1(x):=\left(\mathcal{G}^\ast\left(x, \frac{u(x)}{\beta}\right)\right)^{-\frac{1}{N}} \frac{\partial\mathcal{G}^\ast\left(x, \frac{u(x)}{\beta}\right)}{\partial t} \left|\nabla u\right|$$
and 
$$I_2(x):=\left(\mathcal{G}^\ast\left(x, \frac{u(x)}{\beta}\right)\right)^{-\frac{1}{N}} \sum_{i=1}^N\left|\frac{\partial \mathcal{G}^\ast\left(x, \frac{u(x)}{\beta}\right)}{\partial x_i} \right|. $$
We estimate $I_1(\cdot)$ and $I_2(\cdot)$ as follows. By  direct computations, it holds that
\begin{equation} \label{I1_1}
	\left|\frac{\partial \mathcal{G}^\ast (x,t)}{\partial t} \right|  \leq C_4\frac{\mathcal{G}^\ast (x,t)}{t},
\end{equation}	

\begin{equation} \label{I1_2}
	\ell(x)\frac{N-1}{N}=\frac{\ell(x)}{q'(x)}+\frac{1}{q(x)}\,  \text{ with } \, \ell(x):=\frac{q^\ast(x)}{q(x)},
\end{equation}

\begin{equation}\label{I1_3}
	\left(p^\ast(x)\frac{N-1}{N}-1\right)p'(x) = p^\ast(x),
\end{equation}
and
\begin{equation}\label{I1_4}
	\left(q^\ast(x)\frac{N-1}{N}-1\right)q'(x) = q^\ast(x) 
\end{equation}
for all $x \in \overline{\Omega} $ and $t \in (0,\infty)$.
From \eqref{I1_1} and the elementary inequality: 
$$|c+d|^{\frac{N}{N-1}} \leq C_5 \left(|c|^{\frac{N}{N-1}}+|d|^{\frac{N}{N-1}}\right), \quad \forall\ c,d \in \mathbb{R},$$
we obtain
\begin{align*}
	\notag
	I_1(x) & \leq C_6 \frac{\left(\mathcal{G}^\ast\left(x, \frac{u(x)}{\beta}\right)\right)^{\frac{N-1}{N}}}{\frac{u(x)}{\beta}}|\nabla u(x)| \\
	& \leq C_7 \left[ \left(\frac{u(x)}{\beta}\right)^{p^\ast(x)\frac{N-1}{N}-1} + (a(x))^{\ell(x)\frac{N-1}{N}}\left(\frac{u(x)}{\beta}\right)^{q^\ast(x)\frac{N-1}{N}-1}\right]\,|\nabla u(x)|\quad \text{for a.a. }x\in\Omega_+.
\end{align*}	
Then, by invoking \eqref{young}, we derive from \eqref{I1_2}-\eqref{I1_4} that
\begin{align*}
I_1(x) \leq C_8 &\left[ \left(\frac{u(x)}{\beta}\right)^{(p^\ast(x)\frac{N-1}{N}-1)p'(x)}+a(x)^{\ell(x)}\left(\frac{u(x)}{\beta}\right)^{(q^\ast(x)\frac{N-1}{N}-1)q'(x)}\right]\\&\quad+C_8 \round{ |\nabla u(x)|^{p(x)}+a(x)|\nabla u(x)|^{q(x)}}\quad \text{for a.a. }x\in\Omega_+,
\end{align*}
i.e,
\begin{equation}
	I_1(x) \leq C_8 \mathcal{G}^\ast\left(x,\frac{u(x)}{\beta}\right)  +  C_8\mathcal{H}(x,|\nabla u(x)|)\quad\text{for a.a. }x\in\Omega_+.\label{I1main}
\end{equation}
In order to estimate $I_2(\cdot)$, we make use of the Lipschitz continuity of $p(\cdot)$ and $q(\cdot)$ to obtain
\begin{equation}\label{I_2(x)}
	I_2(x) \leq  C_9I_{2,1}(x)+C_9I_{2,2}(x)\quad\text{for a.a.}~ x \in \Omega_+,  
\end{equation}
where
\begin{align*}
	I_{2,1}(x):&=\left[\mathcal{G}^\ast\left(x, \frac{u(x)}{\beta}\right)\right]^{-\frac{1}{N}}  \round{\left(\frac{u(x)}{\beta}\right)^{p^\ast(x)} + a(x)^{\ell(x)}\left(\frac{u(x)}{\beta}\right)^{q^\ast(x)}}\left|\ln \left(\frac{u(x)}{\beta}\right)\right|\\
	&= \left[\mathcal{G}^\ast\left(x, \frac{u(x)}{\beta}\right)\right]^{\frac{N-1}{N}}\left|\ln \left(\frac{u(x)}{\beta}\right)\right|
\end{align*}
and 
\begin{align*}
	I_{2,2}(x):= \left[\mathcal{G}^\ast\left(x, \frac{u(x)}{\beta}\right)\right]^{-\frac{1}{N}} \sum_{i=1}^N\Big|\frac{\partial (a(x)^{\ell(x)})}{\partial x_i}\Big|\,\left(\frac{u(x)}{\beta}\right)^{q^\ast(x)}.
\end{align*}
By $(\calbf{A})$, we have
\begin{equation*}
	\varepsilon_0:=\frac{1}{2}\inf_{ x \in\overline{\Omega}}\left[p^\ast(x) \frac{N-1}{N}-p(x)\right]>0.
\end{equation*}			
Hence, we obtain
\begin{align}\label{PT.CE1-L1}
	\notag
	\lim\limits_{t \to 0^+} \frac{[\mathcal{G}^\ast(x,t)]^{\frac{N-1}{N}}}{t^{p(x)+\varepsilon_0}}& = \lim\limits_{t \to 0^+} t^{p^\ast(x)\frac{N-1}{N}-p(x)-\varepsilon_0}\left(1+a(x)^{\ell(x)}t^{q^\ast(x)-p^\ast(x)}\right)^{\frac{N-1}{N}} \\
	& =0,\quad \text{uniformly for}\ \  x\in \overline{\Omega}.
\end{align}
From \eqref{PT.CE1-L1} and the fact that $\lim\limits_{t \to 0^+} t^{\varepsilon_0}\ln t=0$ it holds that
$$C_{10}(t_0):=\sup\limits_{0<t \leq t_{0} \atop x \in \overline{\Omega}} \left\{ \left[\mathcal{G}^\ast(x,t)\right]^{\frac{N-1}{N}}t^{-p(x)-\varepsilon_0} t^{\varepsilon_0}\ln t\right\} \in (0,\infty)$$
for each $t_0>0$. Consequently, we have
\begin{align} \label{I2.1}
	\notag I_{2,1}(x) &= \left[\mathcal{G}^\ast\left(x, \frac{u(x)}{\beta}\right)\right]^{\frac{N-1}{N}}\left(\frac{u(x)}{\beta}\right)^{-p(x)-\eps_0}\left(\frac{u(x)}{\beta}\right)^{\eps_0}\left|\ln \left(\frac{u(x)}{\beta}\right)\right|\left(\frac{u(x)}{\beta}\right)^{p(x)}\\
	&\leq C_{10}(t_0) \mathcal{H}\left(x,\frac{u(x)}{\beta}\right)\quad \text{for a.a. }x\in\Omega_+ \ \ \text{with}\ \  \frac{u(x)}{\beta}\leq t_0. 
\end{align}
Let us put $\delta:=\frac{1}{6C_2}$. By the fact that $\lim\limits_{t \to \infty} t^{-\frac{(p^*)^-}{N}}\ln t=0$ we can fix $t_0>1$ such that $t^{-\frac{p^\ast(x)}{N}}\ln t\leq t^{-\frac{(p^\ast)^-}{N}}\ln t \leq \frac{\delta}{C_9}$ for all $(x,t)  \in \overline{\Omega}\times (t_0,\infty).$ Thus, 
\begin{equation} \label{PT.CE1-L2}
	\mathcal{G}^\ast(x,t)^{\frac{N-1}{N}}|\ln t| \leq \mathcal{G}^\ast(x,t) t^{-\frac{p^\ast(x)}{N}}\ln t\leq \frac{\delta}{C_9} \mathcal{G}^\ast(x,t)\quad \text{for all}~ (x,t)  \in \overline{\Omega}\times (t_0,\infty).
\end{equation}
Consequently, it holds that
\begin{align} \label{I2.2}
	I_{2,1}(x) \leq \frac{\delta}{C_9} \mathcal{G}^\ast\left(x,\frac{u(x)}{\beta}\right)\quad \text{for a.a. }x\in\Omega_+ \ \ \text{with}\ \  \frac{u(x)}{\beta}>t_0. 
\end{align}
Combining \eqref{I2.1} with \eqref{I2.2} gives 
\begin{align} \label{Tm2.6-I2}
	I_{2,1}(x)\leq C_{10}(t_0) \mathcal{H}\left(x,\frac{u(x)}{\beta}\right)+\frac{\delta}{C_9} \mathcal{G}^\ast\left(x,\frac{u(x)}{\beta}\right)\quad \text{for a.a. }x\in\Omega_+. 
\end{align}
For estimating $I_{2,2}(\cdot)$, we denote
\begin{align*}
	\Omega_0:=\left\{x \in \Omega_+: \ a(x)=0\right\}\, \text{and}~ \Omega_1:=\left\{x \in \Omega_+: \ a(x)>0\right\}.
\end{align*}
Since $\frac{\partial \left(a^{\ell}\right)}{\partial x_i}=0$  a.e. in $\Omega_0$ for $i=1,\ldots, N$ and 
$$\Big[\mathcal{G}^\ast(x,t)\Big]^{-\frac{1}{N}} \leq a(x)^{-\frac{\ell(x)}{N}}t^{-\frac{q^\ast(x)}{N}} \quad \text{for a.a. } x \in \Omega_1 ~ \text{and all} ~ t \in (0,\infty),$$
it follows that
\begin{align*}
	I_{2,2}(x)=0, \quad \text{for a.a. }x\in\Omega_0,
\end{align*}
and 
\begin{align*}
	\notag I_{2,2}(x)& \leq a(x)^{-\frac{\ell(x)}{N}} a(x)^{\ell(x)}\,  \sum_{i=1}^N\left( \Big|\frac{\partial \ell}{\partial x_i} \Big||\ln a|+\frac{\ell}{a} \, \Big|\frac{\partial a}{\partial x_i} \Big| \right) \, \left(\frac{u(x)}{\beta}\right)^{q^\ast(x)\frac{N-1}{N}} \\
	&\leq C_{11} a(x)^{\ell(x)\frac{N-1}{N}-1}\left( \frac{u(x)}{\beta}\right)^{q^\ast(x)\frac{N-1}{N}}\quad \text{for a.a. }x\in\Omega_1,
\end{align*}
where 
$$C_{11}:=\sup_{x\in\overline{\Omega}}\, \sum_{i=1}^N\left( \Big|\frac{\partial \ell}{\partial x_i} \Big||\ln a|a+\ell \, \Big|\frac{\partial a}{\partial x_i} \Big| \right) \in (0,\infty)$$
 due to $(\calbf{A})$ and the fact that $\lim\limits_{t \to 0^+} t\ln t=0$. Thus, we obtain
\begin{align}\label{I2,2-1}
	I_{2,2}(x)\leq C_{11} a(x)^{\ell(x)\frac{N-1}{N}-1}\left( \frac{u(x)}{\beta}\right)^{q^\ast(x)\frac{N-1}{N}}\quad \text{for a.a. }x\in\Omega_+.
\end{align}
On the other hand, since $1\ll\ell(\cdot)\frac{N-1}{N}$ and $p(\cdot)\ll {q^\ast(\cdot)\frac{N-1}{N}}$ it holds that
\begin{align}\label{I2.Omega1.D1}
	a(x)^{\ell(x)\frac{N-1}{N}-1}t^{q^\ast(x)\frac{N-1}{N}} \leq C_{12}(t_0) \mathcal{H}\left(x,t\right)\quad \text{for all}~ (x,t)  \in \overline{\Omega}\times (0,t_0],
\end{align}
where 
\begin{equation*}
	C_{12}(t_0):=\sup\limits_{0<t \leq t_{0} \atop x \in \overline{\Omega}}a(x)^{\ell(x)\frac{N-1}{N}-1}t^{{q^\ast(x)\frac{N-1}{N}}-p(x)} \in (0,\infty).
\end{equation*}
Applying \eqref{young} for $r(x)=\frac{N}{q(x)-1}$, we obtain
\begin{align}\label{a.est}
	\notag
	a(x)^{\ell(x) \frac{N-1}{N}-1}t^{q^\ast(x)\frac{N-1}{N}} &=	a(x)^{\ell(x)\frac{q(x)-1}{N}} t^{q^\ast(x)\frac{q(x)-1}{N}}t^{q^\ast(x)\left(\frac{N-1}{N}-\frac{q(x)-1}{N}\right)}\\
	&	\leq \frac{\delta}{C_{11}} a(x)^{\ell(x)}t^{q^\ast(x)}+C_{13} t^{\frac{Nq(x)}{N-q(x)+1}} \quad \text{for all}~ (x,t)  \in \overline{\Omega}\times [0,\infty).
\end{align} 
From \eqref{A11} and taking $t_0$ larger if necessary, it holds that
\begin{equation}\label{J4.2}
	t^{\frac{Nq(x)}{N-q(x)+1}-p^\ast(x)}<\frac{\delta}{C_9C_{11}C_{13}} \quad \text{for all}~ (x,t)  \in \overline{\Omega}\times (t_0,\infty).
\end{equation}
Combining \eqref{a.est} and \eqref{J4.2} gives
\begin{align} \label{I2.D2}			
	\notag
	a(x)^{\ell(x) \frac{N-1}{N}-1}t^{q^\ast(x)\frac{N-1}{N}} &  \leq \frac{\delta}{C_9C_{11}} a(x)^{\ell(x)}t^{q^\ast(x)}+ C_{13}  t^{\frac{Nq(x)}{N-q(x)+1}-p^\ast(x)}t^{p^\ast(x)}\, \notag\\
	&\leq \frac{\delta}{C_9C_{11}} \mathcal{G}^\ast\left(x,t\right)\quad \text{for all}~ (x,t)  \in \overline{\Omega}\times (t_0,\infty).
\end{align}	
Utilizing \eqref{I2.Omega1.D1} and  \eqref{I2.D2}, we derive from \eqref{I2,2-1} that
\begin{equation} \label{I2.Omega+.2}
	I_{2,2}(x)   \leq C_{14}(t_0) \mathcal{H}\left(x,\frac{u(x)}{\beta}\right)  + \frac{\delta}{C_9} \mathcal{G}^\ast\left(x,\frac{u(x)}{\beta}\right)\quad \text{for a.a. }x\in\Omega_+.
\end{equation}
By using \eqref{Tm2.6-I2} and \eqref{I2.Omega+.2}, we deduce from \eqref{I_2(x)} that
\begin{equation}\label{I2main}
	I_2(x)\leq C_{15}(t_0) \mathcal{H}\left(x,\frac{u(x)}{\beta}\right)  + 2\delta\mathcal{G}^\ast\left(x,\frac{u(x)}{\beta}\right)\quad \text{for a.a. }x\in\Omega_+.
\end{equation}
Collecting  \eqref{P.Thm2.6.f=0}, \eqref{Nabla-f}, \eqref{I1main} and \eqref{I2main} altogether we obtain 
\begin{equation}\label{Grad-f}
	|\nabla f(x)|\leq \frac{C_{16}}{\beta}\left[\mathcal{G}^\ast\left(x,\frac{u(x)}{\beta}\right)+\mathcal{H}(x,|\nabla u(x)|)\right]+C_{15}(t_0)\mathcal{H}\left(x,\frac{u(x)}{\beta}\right)  + 2\delta\mathcal{G}^\ast\left(x,\frac{u(x)}{\beta}\right)
\end{equation}
for a.a. $x\in\Omega$. 

\vskip5pt
In order to estimate $|f|$, we argue as those leading to \eqref{I2.1} and \eqref{PT.CE1-L2} to obtain for a $t_0$ larger if necessary that
\begin{align*}
	\mathcal{G}^\ast\left(x, t\right)^{\frac{N-1}{N}}=  \mathcal{G}^\ast\left(x, t\right)^{\frac{N-1}{N}} t^{-p(x)}t^{p(x)}\leq C_{17}(t_0)\mathcal{H}\left(x,t\right)\quad \text{for all}~ (x,t)  \in \overline{\Omega}\times (0,t_0],
\end{align*} 
where $$C_{17}(t_0):= \sup\limits_{t \in (0,t_0] \atop x \in \overline{\Omega}} \mathcal{G}^\ast\left(x, t\right)^{\frac{N-1}{N}} t^{-p(x)} \in (0,\infty)$$
and
\begin{align*} 
	\notag
	\mathcal{G}^\ast\left(x, t\right)^{\frac{N-1}{N}}& \leq \mathcal{G}^\ast\left(x, t\right)t^{-\frac{p^\ast(x)}{N}}\leq \delta \mathcal{G}^\ast\left(x, t\right)\quad \text{for all}~ (x,t)  \in \overline{\Omega}\times (t_0,\infty).
\end{align*}
From these facts and \eqref{P.Thm2.6.f=0}, we obtain
\begin{align} \label{f}
	|f(x)|=\left[\mathcal{G}^\ast\left(x, \frac{u(x)}{\beta}\right)\right]^{\frac{N-1}{N}}\leq C_{17}(t_0) \mathcal{H}\left(x,\frac{u(x)}{\beta}\right)  + \delta\mathcal{G}^\ast\left(x, \frac{u(x)}{\beta}\right) \ \ \text{for a.a. }x\in\Omega.
\end{align} 
By \eqref{Grad-f} and \eqref{f}, we infer $f\in W^{1,1}(\Omega)$. Then, applying \eqref{C-I} for $v=f$ and then using \eqref{int_G*}, \eqref{Grad-f}, \eqref{f}, we arrive at
\begin{align*}
	1 \leq \frac{C_{18}}{\beta} \left[1+\int_{\Omega} \mathcal{H}(x,|\nabla u(x)|)\diff x \right]+ C_{19}(t_0)\int_\Omega \mathcal{H}\left(x,\frac{u(x)}{\beta}\right)\diff x+ 3C_2\delta.
\end{align*}
Note that $\delta=\frac{1}{6C_2}$, and we could choose $C_{19}(t_0)$ to depend only on the data. Hence, for $\beta \geq 1$, the last inequality implies that
\begin{align*}
	\frac{1}{2} \leq \frac{C_{20}}{\beta} \left[1+\int_{\Omega} \mathcal{H}(x,|\nabla u(x)|)\diff x \right]+ \frac{C_{20}}{\beta^{p^-}}\int_\Omega \mathcal{H}\left(x,\frac{u(x)}{\beta}\right)\diff x.
\end{align*}
So, we obtain
\begin{align} \label{lambda.ast.2}
	\beta=\|u\|_{\mathcal{G}^\ast} \leq (1+2C_{20})\left[ 1 +  \int_\Omega \mathcal{H}(x,|\nabla u|)\diff x+\int_\Omega \mathcal{H}(x,u)\diff x\right],\quad \forall\beta\geq 1,
\end{align}
Obviously, \eqref{lambda.ast.2} holds with $\|u\|_{\mathcal{G}^\ast}<1$. Thus, we obtain \eqref{lambda.ast.2} for all $u \in W_c^{1,\mathcal{H}}(\Omega) \cap L^\infty(\Omega)$. This fact and Proposition~\ref{prop.nor-mod.D} yield Claim 1.

\medskip

\textbf{Claim  2:} \textit{\eqref{case3} holds for all $u \in W^{1,\mathcal{H}}(\Omega) \cap L^\infty(\Omega)$.}

\medskip
\noindent Indeed, let $u \in W^{1,\mathcal{H}}(\Omega) \cap L^\infty(\Omega)$ be given and arbitrary. For each $n\in\N$, let $\psi_n \in C_c^\infty(\RN)$ be such that
$$\begin{cases}
	\psi_n(x)=1 & \text{ if } |x| \leq n,\\
	\psi_n(x)=0 & \text{ if } |x|>3n,\\
	\psi_n(x)\in [0,1] & \text{ for all } x \in \RN,\\
	|\nabla \psi_n(x)|\leq 1 & \text{ for all } x \in \RN
\end{cases}$$
and define $u_n(x):=\psi_n(x)u(x)$ for $x\in\Omega$. Clearly, $u_n$ has compact support, and furthermore, $|u_n| \leq |u|$ and $|\nabla u_n| \leq |u|+|\nabla u|$ a.e. in $\Omega$. Thus, $u_n \in W^{1,\mathcal{H}}_{c}(\Omega)\cap L^{\infty}(\Omega)$; hence, by Claim 1 we have
\begin{equation*}
	\int_\Omega \mathcal{G}^\ast(x,|u_n|)\diff x \leq C_{1}\left[\int_{\Omega}\mathcal{H}(x,|\nabla u|)\diff x +\int_{\Omega}\mathcal{H}(x,|u|)\diff x+1\right]^{{(q^\ast)}^+}.
\end{equation*}		
Since $u_n \to u$ a.e. in $\Omega$ as $n \to \infty$, by passing to the limit as $n\to\infty$ in the last inequality we obtain
\begin{align}\label{case2}
	\int_\Omega\mathcal{G}^\ast(x,|u|)\diff x \leq C_{1}\left[\int_{\Omega}\mathcal{H}(x,|\nabla u|)\diff x +\int_{\Omega}\mathcal{H}(x, |u|)\diff x+1\right]^{{(q^\ast)}^+}
\end{align}			
in view of Fatou's lemma. That is, Claim 2 has been proved.

\medskip
\textbf{Claim  3:} \textit{\eqref{case3} holds for all $u \in W^{1,\mathcal{H}}(\Omega)$.}

\medskip
\noindent Indeed, let $u \in W^{1,\mathcal{H}}(\Omega)$ be given and arbitrary. For each $n\in\N$, define
\begin{align*}
	v_{n}(x)= \begin{cases}u(x) & \text { if }|u(x)| \leq n, \\ n \operatorname{sgn} u(x) & \text { if }|u(x)|>n.\end{cases}
\end{align*} 			
It is clear that $|v_n|=\min\left\{|u|, n\right\}$ and $|\nabla v_n|\leq |\nabla u|$ a.e. in $\Omega$; hence, $v_n \in W^{1,\mathcal{H}}(\Omega) \cap L^\infty(\Omega)$.  
Thus, by means of Claim 2 it holds that
\begin{align*} 
	\int_\Omega \mathcal{G}^\ast(x,|v_n|)\diff x  & \leq C_1\left[\int_{\Omega}\mathcal{H}(x,|\nabla v_n|)\diff x  + \int_{\Omega}\mathcal{H}(x, |v_n|)\diff x+1\right]^{{(q^\ast)}^+}\\
	&\leq C_1\left[\int_{\Omega}\mathcal{H}(x,|\nabla u|)\diff x  + \int_{\Omega}\mathcal{H}(x, |u|)\diff x+1\right]^{{(q^\ast)}^+}.
\end{align*}
By passing to the limit as $n\to\infty$ in the last estimate, noticing $v_n \to u$ a.e. in  $\Omega$, Claim 3 follows in view of Fatou's lemma again. The proof is complete.
\end{proof}

\begin{remark}\rm
When $a(\cdot) \equiv 0$, Theorem \ref{T.CE1} becomes the critical embedding for $W^{1,p(\cdot)}(\Omega)$, which was proved in \cite[Lemma 2.1]{fan2001sobolev}.
\end{remark}

In applications, we will employ Theorem~\ref{T.CE1} in a more general form as follows. Define
\begin{equation}\label{B}
	\mathcal{B}(x,t) :=c_1(x) t^{r(x)} + c_2(x) a(x)^{\frac{s(x)}{q(x)}}t^{s(x)} \text{ for a.a. } x \in \Omega \text{ and all }t\in [0,\infty),
\end{equation}
where  $r,s\in C_+(\overline{\Omega})$ and $c_1,c_2\in M(\Omega) $ satisfying $c_1(x)>0$ and $c_2(x) \geq 0$ for a.a. $x \in \Omega$.

\begin{theorem} \label{T.CE2}
Let $(\calbf{A})$ hold, and let $\mathcal{B}$ be defined as in \eqref{B} with $c_1,c_2\in L^\infty(\Omega)$.  Then, the following assertions hold:
\begin{enumerate}
	\item[(i)] if  $ p(x)\leq  r(x) \leq p^\ast(x)$ and $q(x) \leq s(x) \leq q^\ast(x)$ for all $x\in \overline{\Omega}$, then	$W^{1,\mathcal{H}}(\Omega) \hookrightarrow L^{\mathcal{B}}(\Omega)$;
	\item[(ii)] if $\Omega$ is a bounded Lipschitz domain,  $r(x)<p^\ast(x),$ and  $s(x)<q^\ast(x)$ for all $x \in \overline{\Omega}$, then 	$W^{1,\mathcal{H}}(\Omega) \hookrightarrow \hookrightarrow L^{\mathcal{B}}(\Omega)$.
\end{enumerate} 
\end{theorem}
\begin{proof}
Suppose that $ p(x)\leq  r(x) \leq p^\ast(x)$ and $q(x) \leq s(x) \leq q^\ast(x)$ for all $x\in \overline{\Omega}$. Then, it holds 
\begin{align*}
	\mathcal{B}(x,t)& = c_1(x)t^{r(x)} + c_2(x) \left(a(x)^{\frac{1}{q(x)}}t\right)^{s(x)} \\
	&\leq c_1(x) \left(t^{p(x)}+t^{p^\ast(x)}\right) + c_2(x) \left[\left(a(x)^{\frac{1}{q(x)}}t\right)^{q(x)}+ \left(a(x)^{\frac{1}{q(x)}}t\right)^{q^\ast(x)}\right] \\
	&\leq C\left[ \mathcal{H}(x,t) + \mathcal{G}^\ast(x,t)\right],\quad \forall (x,t) \in \overline{\Omega} \times [0,\infty).
\end{align*}
From this and Theorem \ref{T.CE1}, we easily  obtain $(i)$. The assertion  $(ii)$ is from \cite[Proposition 3.7]{HW2022}.
\end{proof}


\subsection{The Musielak-Orilcz-Sobolev spaces $W^{1,\mathcal{H}}_V(\RN)$ and $X_V$}${}$ \label{subsec.sol.space}

\vskip5pt

In this subsection, we define and explore Musielak-Orlicz-Sobolev spaces associated with the double phase operator given in \eqref{oper}. In the sequel, let $\mathcal{H}$ be given by \eqref{def_H} and satisfy $(\calbf{A})$ with $\Omega=\R^N$, and let $V\in L^1_{\loc}\left(\R^N\right)$ be such that $V(\cdot) \geq 0$ and $V\ne 0$. We define the space $W^{1,\mathcal{H}}_V(\RN)$ as
\begin{align*}
W^{1,\mathcal{H}}_V(\RN):=\left\{u\in W_{\loc}^{1,1}(\RN):~ \rho_V(u)<\infty\right\},
\end{align*}
where the modular $\rho_V$ is defined as
\begin{align*}
\rho_V(u):=\int_{\RN} \mathcal{H}\left(x,|\nabla u|\right) \diff x + \int_{\RN} V(x)\mathcal{H}\left(x,|u|\right) \diff x\quad\text{for}\ u\in W_{\loc}^{1,1}(\RN).
\end{align*}
Then, $W^{1,\mathcal{H}}_V(\RN)$ is a normed space with the norm
\begin{align}\label{norm}
\|u\|:=\inf\left\{ \tau >0: \rho_V\left(\frac{u}{\tau}\right) \leq 1\right\},
\end{align}
see, e.g. \cite[Theorem 2.1.7]{Diening}. As Proposition~\ref{Prop_m-n_W(1,H)}, on this space we have
\begin{equation}\label{modular-norm.XV}
	 \rho_V\left(\frac{u}{\|u\|}\right)=1, \quad \forall u \in W_V^{1,\mathcal{H}}(\RN) \setminus \{0\},
\end{equation}
and
\begin{equation}\label{m-n}
\min \left\{\|u\|^{p^-},\|u\|^{q^+}\right\}\leq \rho_V(u)\leq \max \left\{\|u\|^{p^-},\|u\|^{q^+}\right\}, \quad \forall u \in W_V^{1,\mathcal{H}}(\RN).
\end{equation}

Clearly, if $\essinf_{x \in \R^N}V(x)>0$, then $W^{1,\mathcal{H}}_V(\RN)$ is a separable reflexive Banach space (a proof is similar to that of Theorem~\ref{T.Xv-P} below). Furthermore, it holds that
\begin{equation}\label{Em}
W^{1,\mathcal{H}}_V(\RN)\hookrightarrow W^{1,\mathcal{H}}(\RN),
\end{equation}
i.e., there exists a constant $C>0$ such that
\begin{equation}\label{infV>0}
\|u\|_{W^{1,\mathcal{H}}(\RN)}\leq C\|u\|,\quad \forall u\in W^{1,\mathcal{H}}_V(\RN). 
\end{equation}
In order to obtain \eqref{Em} with a larger class of potentials $V$, we make the following assumption:
\begin{enumerate} 
\item  [$(\calbf{V}_1)$]   $V\in L^1_{\loc}\left(\R^N\right)$ satisfies $V(\cdot) \geq 0$ and $V\ne 0$, and one of the following two conditions holds:
\begin{enumerate}
	\item[$(a)$]  there exists $K_0>0$ such that the set $E_V:=\left\{x \in \mathbb{R}^{N}: V(x)<K_0\right\} \neq \emptyset$  and $|E_V|<\infty$;
	\item[$(b)$]  there exists $R_0>0$ such that $\essinf_{x \in B^c_{R_0}}V(x)=V_0>0$.
\end{enumerate}
\end{enumerate}
\begin{remark}\label{Rmk.WH}\rm
The condition $(\calbf{V}_1)(a)$  was initially introduced in \cite{BW.95} for the case $\mathcal{H}(x,t)=t^2$. Obviously, this condition does not cover the case of constant potentials, and the alternative condition  $(\calbf{V}_1)(b)$ complements this deficiency.
\end{remark} 
The condition $V\in L^1_{\loc}\left(\R^N\right)$ guarantees that $C_c^\infty(\R^N)\subset W_V^{1,\mathcal{H}}(\RN)$. Moreover, we have the following.
\begin{lemma}\label{Le-IC}
Let $(\calbf{A})$ and $(\calbf{V}_1)$ hold. Furthermore, for the case of $\essinf_{x \in \R^N}V(x)=0$, we assume in addition that 
\begin{enumerate}
	\item[$(\calbf{P})$] The function $p$ satisfies the $\log-$H\"oder decay condition, i.e., there exists $p_\infty \in (1,N)$ such that 
	\begin{align*}
		\sup_{x \in \RN} |p(x)-p_\infty|\log(e+|x|)<\infty.
	\end{align*}
\end{enumerate}Then, there exists a constant $C>0$ such that
\begin{equation}\label{Xemb1}
	\|u\|_{W^{1,\mathcal{H}}(\RN)}\leq C\|u\|,\quad \forall u\in C_c^\infty(\R^N). 
\end{equation}
\end{lemma}
\begin{proof}
It is clear that the conclusion holds for the case of $\essinf_{x \in \R^N}V(x)>0$ in view of \eqref{infV>0}. For the case of $\essinf_{x \in \R^N}V(x)=0$ with $(\calbf{A})$ and $(\calbf{P})$ being assumed, we note that
\begin{equation}\label{S.bar}
	\bar{S}:=\inf_{u \in C_c^\infty(\RN)\setminus \{0\}}\frac{\norm{|\nabla u|}_{L^{p(\cdot)}(\RN)}}{\|u\|_{L^{p^\ast(\cdot)}(\RN)}}>0,
\end{equation}
see \cite[Theorem 8.3.1]{Diening}. First, we consider the case of   $(\calbf{V}_1)(a)$. Let $u \in C_c^\infty(\RN)$, we have
\begin{align}
	\notag
	\int_{\RN} \mathcal{H}(x,|u|) \diff x & = \int_{E_V} \mathcal{H}(x,|u|)\diff x + \int_{E_V^c}  \mathcal{H}(x,|u|)\diff x \\
	& \leq \int_{E_V} \mathcal{H}(x,|u|)\diff x + \frac{1}{K_0} \int_{E_V^c}  V(x) \mathcal{H}(x,|u|)\diff x \notag \\
	& \leq \int_{E_V} \mathcal{H}(x,|u|)\diff x + \frac{1}{K_0} \rho_V(u). \label{Xemb2}
\end{align}
Invoking Propositions \ref{prop.Holder}-\ref{prop.norm-modular} and Corollary~\ref{Cor-n}, it follows from \eqref{A10} and \eqref{S.bar} that
\begin{align}
	\notag
	\int_{E_V} \mathcal{H}(x,|u|) \diff x & \leq  2\big\||u|^{p}\big\|_{L^{\frac{p^\ast(\cdot)}{p(\cdot)}}(E_V)}\|1\|_{L^{\frac{p^\ast(\cdot)}{p^\ast(\cdot)-p(\cdot)}}(E_V)} \\  & \hspace{1cm} +2\big\|a\big\|_{L^\infty(\RN)}\big\||u|^{q}\big\|_{L^{\frac{p^\ast(\cdot)}{q(\cdot)}}(E_V)}\|1\|_{L^{\frac{p^\ast(\cdot)}{p^\ast(\cdot)-q(\cdot)}}(E_V)} \notag \\
	& \leq C_1 \left(1+\big\|u\big\|^{p^+}_{L^{p^\ast(\cdot)}(E_V)}+\big\|u\big\|^{q^+}_{L^{p^\ast(\cdot)}(E_V)}\right) \notag  \\
	& \leq C_1\left(1+{\bar{S}}^{-p^+}\norm{|\nabla u|}^{p^+}_{L^{p(\cdot)}(\RN)} + {\bar{S}}^{-q^+}\norm{|\nabla u|}^{q^+}_{L^{p(\cdot)}(\RN)}\right) \notag \\
	& \leq C_2 \left[1+ \left( \int_{\RN}|\nabla u|^{p(x)} \diff x \right)^{\frac{p^+}{p^-}} + \left( \int_{\RN}|\nabla u|^{p(x)} \diff x \right)^{\frac{q^+}{p^-}}\right] \notag \\
	& \leq C_3 \left[1  + \rho_V(u)^{\frac{q^+}{p^-}} \right]. \label{Xemb3}
\end{align}
Using \eqref{Xemb2} and \eqref{Xemb3} one can find $C_4>1$ such that
\begin{equation}\label{u.Cc}
	\int_{\RN} \mathcal{H}\left(x,|\nabla u|\right) \diff x+\int_{\RN} \mathcal{H}(x,|u|)\diff x \leq  C_4\left[1  + \rho_V(u)^{\frac{q^+}{p^-}} \right], \quad \forall u \in C_c^\infty(\RN).
\end{equation}
Now, for each $u \in C_c^\infty(\RN)\setminus\{0\}$, we apply \eqref{u.Cc} for $v=\frac{u}{\|u\|}$ and use the relation  \eqref{modular-norm.XV} to obtain 
\begin{align*}
	\int_{\RN} \Big[\mathcal{H}(x,|\nabla v|)+\mathcal{H}(x,|v|)\Big] \diff x \leq C_4 \left[1 + \rho_V(v)^{\frac{q^+}{p^-}}\right]= 2C_4.
\end{align*}
It follows that
\begin{align*}
	\int_{\RN} \left[ \mathcal{H}\left(x,\left|\frac{\nabla v}{(2C_4)^{\frac{1}{p^-}}}\right|\right) +  \mathcal{H}\left(x,\left|\frac{v}{(2C_4)^{\frac{1}{p^-}}}\right|\right)  \right] \diff x \leq 1,
\end{align*}
which implies \eqref{Xemb1}. 

The proof for case of $(\calbf{V}_1)(b)$ is the same as above, except that $E_V$ and $K_0$ are replaced by $B_{R_0}$ and $V_0$, respectively. 
\end{proof}

In light of Lemma~\ref{Le-IC}, henceforth in this paper, unless explicitly stated otherwise, we consistently adopt the following assumption on the main operator:
\begin{enumerate}
\item [$(\calbf{O})$] The functions $p,q,a$ satisfy $(\calbf{A})$ while $V$ satisfies $(\calbf{V}_1)$. Furthermore, $p$ additionally fulfills $(\calbf{P})$ when $\essinf_{x \in \R^N}V(x)=0$.
\end{enumerate}
Let $X_V$ denote the closure of $C_c^\infty(\RN)$ in $W_V^{1,\mathcal{H}}(\RN)$. 
We have the following result.

\begin{theorem}\label{T.Xv-P}
$X_V$ is a separable reflexive Banach space. Furthermore, one has
\begin{align}\label{Xim}
	X_V \hookrightarrow W^{1,\mathcal{H}}(\RN).
\end{align}
\end{theorem}
\begin{proof}
We begin the proof by showing 
\begin{equation}\label{Xemb1''}
	\|v\|_{W^{1,\mathcal{H}}(\RN)}\leq C\|v\|,\quad \forall v\in X_V. 
\end{equation} 
Let $u \in X_V$, by the definition of $X_V$, there exists $\{u_n\}_{n \in \N}\subset C_c^\infty(\RN)$ such that $u_n\to u$ in $X_V$.  Up to a subsequence, we have
\begin{equation}\label{T2.11-1}
	u_n \to u \quad\text{a.e. in} ~ \{V\ne 0\}
\end{equation}
and	
\begin{equation}\label{T2.11-2}
	\nabla u_n \to \nabla u \quad\text{a.e. in}~ \RN.
\end{equation} 
Obviously, $\{u_n\}_{n \in \N}$ is a Cauchy sequence in $X_V$, and thus, in $W^{1,\mathcal{H}}(\RN)$ thanks to Lemma~\ref{Le-IC}. By the completeness of $W^{1,\mathcal{H}}(\RN)$, there exists $\bar{u}\in W^{1,\mathcal{H}}(\RN)$ such that $u_n \to \bar{u}$ in  $W^{1,\mathcal{H}}(\RN)$. Hence, along a subsequence we have
\begin{equation}\label{T2.11-4}
	u_n \to \bar{u} \quad\text{a.e. in} ~ \RN,
\end{equation}
and	\begin{equation}\label{T2.11-5}
	\nabla u_n \to \nabla \bar{u} \quad\text{a.e. in}~ \RN.
\end{equation} 
From \eqref{T2.11-2} and \eqref{T2.11-5}, we obtain $\bar{u}=u+c$ for some constant $c$. Combining this with \eqref{T2.11-1} and \eqref{T2.11-4}, noticing $|\{V\ne 0\}|>0$, we infer $c=0$, i.e. $\bar{u}=u$. Thus, $u\in W^{1,\mathcal{H}}(\RN)$, and by applying \eqref{Xemb1} for $u=u_n$ and passing to the limit as $n\to\infty$ we obtain \eqref{Xemb1''} for $v=u$. Thus, \eqref{Xim} has been proved. 

Next, we aim to show the completeness of $X_V$. Let $\{u_n\}_{n \in \N}$ be a Cauchy sequence in $X_V$. Then, by invoking \eqref{m-n}, for a given $\varepsilon>0$ we find $N_\varepsilon\in\N$ such that 
\begin{align}\label{Xv-1}
	\int_{\RN} \mathcal{H}\left(x,|\nabla u_m -\nabla u_n|\right)\diff x+\int_{\RN} V(x)\mathcal{H}\left(x,|u_m -u_n|\right)\diff x < \eps,\quad\forall m,n\geq N_\varepsilon.
\end{align}
On the other hand, in view of \eqref{Xemb1''}, $\{u_n\}_{n \in \N}$ is also a Cauchy sequence in $W^{1,\mathcal{H}}(\RN)$. As before, there exists $u\in W^{1,\mathcal{H}}(\RN)$ such that, up to a subsequence, 
\begin{equation*}
	u_n \to u \quad\text{and}\quad \nabla u_n \to \nabla u\quad\text{a.e. in} ~ \RN.
\end{equation*}
Then, invoking  Fatou's lemma we derive from \eqref{Xv-1} that
\begin{align*}
	\int_{\RN} \mathcal{H}\left(x,|\nabla u_m -\nabla u|\right)\diff x+\int_{\RN} V(x)\mathcal{H}\left(x,|u_m -u|\right)\diff x \leq \eps,\quad\forall m\geq N_\varepsilon.
\end{align*}
It follows that $u\in W_V^{1,\mathcal{H}}(\RN)$ and $u_n\to u$ in $W_V^{1,\mathcal{H}}(\RN)$; hence, $X_V$ is complete. 

Finally, we show the separability and reflexivity of $X_V$. Define
$$Y:=L^{\varphi}(\RN)\times \left(L^\mathcal{H}(\RN)\right)^N=L^{\varphi}(\RN)\times L^\mathcal{H}(\RN)\times \cdots \times L^\mathcal{H}(\RN),$$ endowed with an equivalent norm
$$\|(u_0,u_1,\cdots,u_N)\|_{Y}=\|u_0\|_{\varphi}+\left\|\left(\sum_{i=1}^{N}u_i^2\right)^{1/2}\right\|_{\mathcal{H}},$$
where $\varphi(x,t):=\left(1+V(x)\right)\mathcal{H}(x,t)$. By \eqref{Xemb1''}, we can show  that
$$\|u\|_1:=\|u\|_{\varphi}+\big\||\nabla u|\big\|_{\mathcal{H}}$$
is an equivalent norm on $X_V$. Hence, 
$$ \Phi: \left(X_V,\|\cdot\|_1\right) \to (Y,\|\cdot\|_Y)$$ $$\Phi (u)=(u,u_{x_1},\cdots ,u_{x_N})$$
is a linear isometric operator. By a standard argument, we can show that  $\Phi\left(X_V\right)$ is closed in $Y$ (see e.g., \cite[Proposition 2.4]{HS17}).  Clearly, $L^{\varphi}(\RN)$ and $L^{\mathcal{H}}(\RN)$ are separable reflexive Banach spaces (see \cite[Theorems 3.3.7, 3.5.2 and 3.6.6]{HH.Book}), and so is $Y$. Thus, $\Phi\left(X_V\right)$ is also separable reflexive, and so is $X_V$ by the linear isometry of $\Phi$. The proof is complete.
\end{proof}
The next result is a direct consequence of Theorems~\ref{T.CE2} and \ref{T.Xv-P}.
\begin{theorem} \label{T.mainE}
Let $\mathcal{B}$ be defined as in \eqref{B} with $c_1,c_2\in L^\infty(\R^N)$. It holds that
\begin{enumerate}
	\item[(i)] if $p(x)\leq  r(x) \leq p^\ast(x)$ and $q(x) \leq s(x) \leq q^\ast(x)$ for all $x\in \R^N$, then one has
	\begin{align} \label{Xl.emb.B}
		X_V \hookrightarrow L^{\mathcal{B}}(\RN);
	\end{align}
\item[(ii)] if $ r(x) < p^\ast(x)$ and $s(x) < q^\ast(x)$ for all $x\in \R^N$, then one has
\begin{align*}
	X_V \hookrightarrow \hookrightarrow L_{\loc}^{\mathcal{B}}(\RN).
	\end{align*}
\end{enumerate} 
\end{theorem}
For the subcritical case, we also have the following compact embedding result by employing the idea of \cite[Lemma 4.1]{hokimsim2019}.
\begin{theorem} \label{Lem.com.sub}
Let $\mathcal{B}$ be defined as in \eqref{B} with $r(\cdot)\ll p^*(\cdot)$, $s(\cdot)\ll q^*(\cdot)$, $c_1 \in L^{\frac{p^\ast(\cdot)}{p^\ast(\cdot)-r(\cdot)}}(\RN)$ and  $c_2 \in L^{\frac{q^\ast(\cdot)}{q^\ast(\cdot)-s(\cdot)}}(\RN)$. Then, one has 
\begin{align} \label{Xl.C.B}
	X_V \hookrightarrow \hookrightarrow L^{\mathcal{B}}(\RN).
\end{align}
\end{theorem}
\begin{proof} By Theorem~\ref{T.Xv-P}, it suffices to prove that
	\begin{equation}\label{compact.emd.sub.1}
		W^{1,\mathcal{H}}(\RN) \hookrightarrow\hookrightarrow L^{\mathcal{B}}(\RN).
	\end{equation}
 To this end, let $\mathcal{G}^\ast$ be defined as in \eqref{def_G*}. By Theorem~\ref{T.CE1}, we find $C_e>0$ such that
\begin{equation}\label{CE-I}
\|u\|_{\mathcal{G}^\ast}\leq C_e\|u\|,\quad \forall u\in W^{1,\mathcal{H}}(\RN).
\end{equation} By invoking  Proposition~\ref{prop.Holder}, Corollary~\ref{Cor-n} and \eqref{CE-I}, for any $ u \in W^{1,\mathcal{H}}(\RN)$, we have
\begin{align*}
	\int_{ \mathbb{R}^N} \mathcal{B}(x,|u|) \diff x & \leq 2 \|c_1\|_{L^{\frac{p^\ast(\cdot)}{p^\ast(\cdot)-r(\cdot)}}(\mathbb{R}^N)}\left\||u|^{r}\right\|_{L^{\frac{p^\ast(\cdot)}{r(\cdot)}}(\mathbb{R}^N)} + 2 \|c_2\|_{L^{\frac{q^\ast(\cdot)}{q^\ast(\cdot)-s(\cdot)}}\left(\mathbb{R}^N\right)} \left\|a^{\frac{s}{q}}|u|^{s}\right\|_{L^{\frac{q^\ast(\cdot)}{s(\cdot)}}\left(\mathbb{R}^N\right)} \\
	&\leq C_1\left(1+ \|u\|^{r^+}_{\mathcal{G}^\ast} + \|u\|^{s^+}_{\mathcal{G}^\ast}\right)  \\
	& \leq 2C_1 \left(1+ \|u\|^{(q^\ast)^+}_{\mathcal{G}^\ast}\right)\\
	& \leq C_2  \left(1+ C_e^{(q^\ast)^+}\|u\|^{(q^\ast)^+}_{W^{1,\mathcal{H}}(\mathbb{R}^N)}\right).
\end{align*}
Consequently, in view of Proposition \ref{prop.nor-mod.D} it holds
\begin{align*}
	\|u\|_{\mathcal{B}} \leq 1+ \left(\int_{ \mathbb{R}^N} \mathcal{B}(x,|u|) \diff x\right)^{\frac{1}{r^-}} \leq 1+ (2C_1)^{\frac{1}{r^-}}  \left(1+ C_e^{(q^\ast)^+}\|u\|^{(q^\ast)^+}_{W^{1,\mathcal{H}}(\mathbb{R}^N)}\right)^{\frac{1}{r^-}}, \quad \forall u \in W^{1,\mathcal{H}}(\mathbb{R}^N).
\end{align*}
From this, we derive $W^{1,\mathcal{H}}(\mathbb{R}^N) \hookrightarrow L^{\mathcal{B}}(\mathbb{R}^N)$.

Finally, we will show \eqref{compact.emd.sub.1}. To this end, let $u_n \rightharpoonup 0$ in  $W^{1,\mathcal{H}}(\mathbb{R}^N)$, and we aim to show that $u_n \to 0$ in $L^{\mathcal{B}}(\mathbb{R}^N)$, which is equivalent to 
\begin{equation}\label{PL-SCE}
\displaystyle \lim\limits_{n\to \infty}\int_{ \mathbb{R}^N} \mathcal{B}(x,|u_n|) \diff x =0 
\end{equation}
in view of  Proposition~\ref{prop.nor-mod.D}. By Theorem~\ref{T.CE2}, we get that $u_n\to 0$ a.e. in $\R^N$. Thus, by the Vitali convergence theorem  (see e.g., \cite[Theorem 2.24]{Fonseca}), \eqref{PL-SCE} follows if we can show that for any $\varepsilon>0$, there exist $\delta=\delta(\eps)>0$ and $R=R(\eps)>0$ such that
\begin{enumerate}
	\item[($1^0$)] for any measurable subset $E\subset\RN$ with $|E|< \delta$, it holds $\int_{E}  \mathcal{B}(x,|u_n|) \diff x<\varepsilon$ for all $n\in\N$;
	\item[($2^0$)] $\int_{B_R^c}  \mathcal{B}(x,|u_n|) \diff x<\varepsilon$ for all $n\in\N$.
\end{enumerate}
Indeed, let $\eps \in (0,1)$ be arbitrary and given. Since $c_1 \in L^{\frac{p^\ast(\cdot)}{p^\ast(\cdot)-r(\cdot)}}(\mathbb{R}^N)$ and  $c_2 \in L^{\frac{q^\ast(\cdot)}{q^\ast(\cdot)-s(\cdot)}}(\mathbb{R}^N)$,  there exist $\delta=\delta(\eps)>0$ and $R=R(\eps)>0$ such that
for any measurable subset $E \subset \RN$ with $|E|< \delta$, we have
\begin{equation}\label{compact.emb.E}
	\max\left\{ \int_{E} |c_1(x)|^{\frac{p^\ast(x)}{p^\ast(x)-r(x)}}\diff x, \int_{E} |c_2(x)|^{\frac{q^\ast(x)}{q^\ast(x)-s(x)}}\diff x \right\} \leq \left(\frac{\eps}{2(2+C^{r^+}_{e}+C^{s^+}_{e})}\right)^{\bar{t}}
\end{equation}
and
\begin{equation}\label{compact.emb.E.1}
	\max\left\{\int_{B_R^c} |c_1(x)|^{\frac{p^\ast(x)}{p^\ast(x)-r(x)}}\diff x , \int_{B_R^c} |c_2(x)|^{\frac{q^\ast(x)}{q^\ast(x)-s(x)}}\diff x \right\} \leq \left(\frac{\eps}{2(2+C^{r^+}_{e}+C^{s^+}_{e})}\right)^{\bar{t}},
\end{equation}
where
$\bar{t}:=\max \left\{\left(\frac{p^\ast}{p^\ast-r}\right)^{+},\left(\frac{q^\ast}{q^\ast-s}\right)^{+} \right\}$. Since $\{u_n\}_{n \in \mathbb{N}}$ is bounded in $W^{1,\mathcal{H}}(\mathbb{R}^N)$, we may assume that
\begin{equation*} \label{un.bounded}
	\|u_n\|_{W^{1,\mathcal{H}}(\mathbb{R}^N)} \leq 1,\quad \forall n \in \mathbb{N}.
\end{equation*}
Then, by invoking  Propositions~\ref{prop.Holder}-\ref{prop.norm-modular}, Corollary~\ref{Cor-n}, \eqref{CE-I} and \eqref{compact.emb.E}, we obtain
\begin{align*}
	\int_{E}  \mathcal{B}(x,|u_n|) \diff x & \leq 2 \|c_1\|_{L^{\frac{p^\ast(\cdot)}{p^\ast(\cdot)-r(\cdot)}}(E)}\left\||u_n|^{r}\right\|_{L^{\frac{p^\ast(\cdot)}{r(\cdot)}}(E)}\\
	& \hspace{2cm} + 2 \|c_2\|_{L^{\frac{q^\ast(\cdot)}{q^\ast(\cdot)-s(\cdot)}}(E)}\left\|a^{\frac{s}{q}}|u_n|^{s}\right\|_{L^{\frac{q^\ast(\cdot)}{s(\cdot)}}(E)} \notag \\
	& \leq 2 \left(\int_{E} |c_1(x)|^{\frac{p^\ast(x)}{p^\ast(x)-r(x)}}\diff x\right) ^{1/\bar{t}}\left(1+\|u_n\|_{\mathcal{G}^\ast}^{r^+}\right) \notag \\
	& \hspace{2cm} + 2 \left(\int_{E} |c_2(x)|^{\frac{q^\ast(x)}{q^\ast(x)-s(x)}}\diff x\right) ^{1/\bar{t}}\left(1+\|u_n\|_{\mathcal{G}^\ast}^{s^+}\right)  \notag\\
	& \leq \frac{\eps}{(2+C^{r^+}_{e}+C^{s^+}_{e})}\left(2+\|u_n\|_{\mathcal{G}^\ast}^{r^+}+\|u_n\|_{\mathcal{G}^\ast}^{s^+}\right)\\
	& \leq  \frac{\eps}{(2+C^{r^+}_{e}+C^{s^+}_{e})}\left(2+C_e^{r^+}\|u_n\|_{W^{1,\mathcal{H}}(\mathbb{R}^N)}^{r^+}+C_e^{s^+}\|u_n\|_{W^{1,\mathcal{H}}(\mathbb{R}^N)}^{s^+}\right) \notag \\
	&  \leq \eps,\quad \forall n\in\N.
\end{align*}
Thus, we have shown $(1^0)$. By replacing $E$ with $B_R^c$ and using \eqref{compact.emb.E.1} in the preceding estimate, we easily show $(2^0)$. The proof is complete.
\end{proof}


The next lemma is the so-called $\textup{(S)}_+$ property of the Fr\'echet derivative of the energy funtional associated with the main operator of problem~\eqref{eq1}, which is essential for obtaining the compactness in the next sections. By using a similar argument to \cite[Theorem 4.1]{Le} (and see also \cite[Theorem 3.3]{crespo2022new}), we have the following.
\begin{lemma} \label{Lem.S+}
If $u_{n}\rightharpoonup u$ in $X_V$ and
$$\limsup_{n\to  \infty } \int_{ \RN}\left[\mathcal{A}(x,\nabla u) \cdot\nabla (u_n-u) + V(x)A(x,u)(u_n-u)\right] \diff x \leq 0,$$
then  $u_n \to  u$ in $X_V$. 
\end{lemma}

We close this section with the following remark for the space $W_V^{1,\mathcal{H}}(\mathbb{R}^N)$.
\begin{remark}\rm
Let $(\calbf{A})$ hold, and let $V\in L^1_{\loc}\left(\R^N\right)$ satisfy $\essinf_{x \in \R^N}V(x)>0$. Then, in view of Theorem~\ref{T.CE2} and \eqref{Em}, it is evident that Theorem~\ref{T.mainE}, Theorem~\ref{Lem.com.sub}, and Lemma~\ref{Lem.S+} hold true when $X_V$ is substituted with $W_V^{1,\mathcal{H}}(\mathbb{R}^N)$.
\end{remark}  

\section{The concentration-compactness principles} \label{ccp}

In this section, we extend the concentration-compactness principle by Lions \cite{Lions} and its variant at infinity in \cite{Chabrowski} to $W^{1,\mathcal{H}}_V(\RN)$ and $X_V$. Our results offer a tool for studying critical double problems driven by the operator given in \eqref{oper} via variational methods.

\subsection{Statements of the concentration-compactness principles}${}$

\vskip5pt

Let $C_c(\mathbb{R}^N)$ be the
set of all continuous functions $u : \mathbb{R}^N \to\mathbb{R}$ whose support is compact, and let $C_0(\mathbb{R}^N)$ be the completion of $C_c(\mathbb{R}^N)$ relative to the supremum norm $\|\cdot\|_\infty.$ We denote by $\mathcal{R}(\mathbb{R}^N)$ the space of all signed finite Radon measures on $\mathbb{R}^N$ endowed with the total variation norm. By the Riesz representation theorem (see e.g., \cite[Section 1.3.3]{Fonseca}), we can identify $\mathcal{R}(\mathbb{R}^N)$ with the dual of $C_0(\mathbb{R}^N)$, that is, for each $\mu \in [C_0(\RN)]^\ast$, there exists a unique element in $\mathcal{R}(\mathbb{R}^N)$, still denoted by $\mu$, such that
\begin{equation*}
\langle \mu, f \rangle = \int_{\mathbb{R}^N} f\diff \mu, \,\, \forall f \in C_0(\mathbb{R}^N).
\end{equation*}
We can also identify $L^1({\mathbb{R}^N})$ with a subspace of   $\mathcal{R}(\mathbb{R}^N)$ through the linear mapping $T:L^1(\mathbb{R}^N) \to  \mathcal{R}(\mathbb{R}^N)$ defined as
$$\langle Tu,f \rangle =\int_{ \mathbb{R}^N} uf \diff x, ~\forall u\in L^1(\mathbb{R}^N),\, \forall f \in C_0(\mathbb{R}^N).$$

\vskip5pt

Let $p,q,a,V$ verify $(\calbf{O})$, and let $\mathcal{H}$ and $\mathcal{B}$
be defined as in \eqref{def_H} and \eqref{B} with $\Omega=\RN$, respectively. For obtaining our main results, we make the assumption on $\mathcal{B}$ as follows:
\begin{enumerate}
\item  [$(\calbf{C})$] $0<c_1(\cdot)\in L^\infty(\RN)$, $0\leq c_2(\cdot) \in L^\infty(\RN)$, and $r,s \in C_+(\mathbb{R}^N)$  such that  $q(\cdot) \ll r(\cdot) \leq p^\ast(\cdot)$,  $q^\ast(x)-s(x)=p^\ast(x)-r(x)$ for  all $x \in \mathbb{R}^N$, and 
$$\mathscr{C}:=\left\{x \in \mathbb{R}^N:\, r(x)=p^\ast(x),\,s(x)=q^\ast(x)\right\}\neq \emptyset.$$
\end{enumerate}
Let $X_V$ be defined as in Subsection~\ref{subsec.sol.space}. Then, in view of Theorem~\ref{T.mainE} we infer
\begin{equation} \label{def_S}
S:=\inf_{\phi \in X_V  \setminus\{0\}}
\frac{\|\phi\|}{\|\phi\|_{\mathcal{B}}}>0.
\end{equation}


The next theorem is a Lions-type concentration-compactness principle.

\begin{theorem} \label{Theo.CCP1}
Let $(\calbf{O})$ and $(\calbf{C})$ hold. Let $\{u_n\}_{n\in\mathbb{N}}$ be a bounded sequence in $X_V$ such that
\begin{eqnarray}
	u_n &\rightharpoonup& u \quad \text{in}\quad  X_V, \label{un_weak.conv}\\
	\mathcal{H}(\cdot,|\nabla u_n|)+  V	\mathcal{H}(\cdot,| u_n|) &\overset{\ast }{\rightharpoonup }&\mu\quad \text{in}\quad \mathcal{R}(\mathbb{R}^N),\label{mu_n to mu*}\\
	\mathcal{B}(\cdot,|u_n|)&\overset{\ast }{\rightharpoonup }&\nu\quad \text{in}\quad \mathcal{R}(\mathbb{R}^N). \label{nu_n-to-nu*}
\end{eqnarray}
Then, there exist $\{x_i\}_{i\in I}\subset \mathscr{C}$ of distinct points and $\{\nu_i\}_{i\in I}, \{\mu_i\}_{i\in I}\subset (0,\infty),$ where $I$ is at most countable, such that
\begin{gather}
	\nu=\mathcal{B}(\cdot,|u|) + \sum_{i\in I}\nu_i\delta_{x_i},\label{T.ccp.form.nu}\\
	\mu \geq 	\mathcal{H}(\cdot,|\nabla u|) + V \mathcal{H}(\cdot,|u|) +   \sum_{i\in I} \mu_i \delta_{x_i},\label{T.ccp.form.mu}\\
	S \min \left\{\nu_i^{\frac{1}{p^\ast(x_i)}},\nu_i^{\frac{1}{q^\ast(x_i)}}  \right\}\leq \max \left\{\mu_i^{\frac{1}{p(x_i)}},\mu_i^{\frac{1}{q(x_i)}}\right\}, \quad \forall i\in I,\label{T.ccp.nu_mu}
\end{gather}
where $\delta_{x_i}$ is the Dirac mass at $x_i$.
\end{theorem}


The next result elucidates the possible loss of mass at infinity. 

\begin{theorem}\label{Theo.CCP.inf} 
Let	$(\calbf{O})$ and $(\calbf{C})$ hold, and let $\{u_n\}_{n\in \mathbb{N}}$ be the same sequence as that in Theorem~\ref{Theo.CCP1}. Set
\begin{align}
	\mu_\infty:=\lim _{R \rightarrow \infty} \limsup _{n \rightarrow \infty} \int_{B_{R}^{c}}\Big[ \mathcal{H}(x,|\nabla u_n|) +  V(x)\mathcal{H}(x,|u_n|)\Big]\diff x \label{muinf-def}
\end{align}
and
\begin{align}
	\nu_{\infty}:=\lim _{R \rightarrow \infty} \limsup_{n \rightarrow \infty} \int_{B_{R}^{c}}\mathcal{B}(x,|u_n|)\diff x . \label{nuinf-def}
\end{align}
Then 
\begin{align}
	\limsup_{n \to \infty}\int_{\mathbb{R}^N}\Big[\mathcal{H}(x,|\nabla u_n|) +  V(x)\mathcal{H}(x,|u_n|)\Big]\diff x
	=\mu(\mathbb{R}^N)+\mu_\infty\label{T.ccp.infinity.mu}
\end{align}
and
\begin{align}
	\limsup_{n \to \infty}\int_{\mathbb{R}^N}\mathcal{B}(x,|u_n|)\diff x=\nu(\mathbb{R}^N)+\nu_\infty.
	\label{T.ccp.infinity.nu}
\end{align}
Moreover, assume in addition that 
\begin{enumerate}
	\item[$(\calbf{E}_\infty)$] 
	For each $h \in \{p,q,r,s\}$, there exists $h_\infty \in (0,\infty)$ such that 
	$$\lim\limits_{|x|\to \infty}h(x)=h_\infty.$$
\end{enumerate}
Then it holds
\begin{equation}\label{T.ccp.inf.nu_mu}
	S \min \left\{\nu_\infty^{\frac{1}{r_{\infty}}}, \nu_\infty^{\frac{1}{s_{\infty}}} \right\} \leq \max \left\{\mu_\infty^{\frac{1}{p_{\infty}}},\,\mu_\infty^{\frac{1}{q_{\infty}}} \right\}.
\end{equation} 
\end{theorem}
When $\essinf_{x \in \R^N}V(x)>0$, we can replace $X_V$ by $W_V^{1,\mathcal{H}}(\mathbb{R}^N)$ in Theorems~\ref{Theo.CCP1} and \ref{Theo.CCP.inf} as follows.
\begin{theorem} \label{Theo.CCP2}
	Let $(\calbf{A})$, $(\calbf{C})$ and $(\calbf{E}_\infty)$ hold, and let $V\in L^1_{\loc}\left(\R^N\right)$ satisfy $\essinf_{x \in \R^N}V(x)>0$. Then, the conclusions of Theorems \ref{Theo.CCP1} and \ref{Theo.CCP.inf} remain valid with $W_V^{1,\mathcal{H}}(\mathbb{R}^N)$ in place of $X_V$.
\end{theorem}

\subsection{Proofs of the concentration-compactness principles }${}$

\vskip5pt

In order to prove Theorems \ref{Theo.CCP1}-\ref{Theo.CCP2}, we will need the following auxiliary results.


\begin{lemma}[\cite{hokimsim2019}] \label{L.convergence}
Let $\nu,\{\nu_n\}_{n\in\mathbb{N}}$ be nonnegative and finite Radon measures on $\mathbb{R}^N$ such that $\nu_n\overset{\ast }{\rightharpoonup } \nu$ in $\mathcal{R}(\mathbb{R}^N)$. Then, for any $m\in C_{+}(\mathbb{R}^N)$ it holds
\begin{equation*}
	\lim_{ n \to \infty} 	\|\phi\|_{L^{m(\cdot)}_{\nu_n}(\mathbb{R}^N)}= 	\|\phi\|_{L^{m(\cdot)}_{\nu}(\mathbb{R}^N)}, \quad \forall \phi\in C_c(\mathbb{R}^N).
\end{equation*}
\end{lemma}	

\begin{lemma}[\cite{hokimsim2019}]\label{L.reserveHolder}
Let $\mu,\nu$ be two finite and nonnegative Radon measures on $\mathbb{R}^N$, such that there exists a positive  constant $C$ holding
$$
\|\phi\|_{L_\nu^{t(\cdot)}(\mathbb{R}^N)}\leq C\|\phi\|_{L_\mu^{s(\cdot)}(\mathbb{R}^N)},\ \ \forall \phi\in C_c^\infty(\mathbb{R}^N),
$$
for some $s,t\in C_+(\mathbb{R}^N)$ satisfying $s(\cdot) \ll t(\cdot)$. Then, there exist an at most countable set $\{x_i\}_{i\in I}$ of distinct points in $\mathbb{R}^N$ and
$\{\nu_i\}_{i\in I}\subset (0,\infty)$, such that
$$
\nu=\sum_{i\in I}\nu_i\delta_{x_i}.
$$	
\end{lemma}	

\begin{lemma}\label{L.un.grad.p}
Let $(\calbf{O})$ hold, and let $u_n \rightharpoonup u$ in $X_V$. Let $\phi \in C_c^\infty(\RN)$ be such that $0 \leq \phi \leq 1$, $|\nabla \phi| \leq 2$ in $\RN$, $\phi \equiv 1$ on $B_{1/2}$ and $\spp (\phi) \subset B_1$. For each $x_i\in\RN$ and $\delta>0$, define $\phi_{i,\delta}(x):=\phi\left(\frac{x-x_i}{\delta}\right)$ for $x \in \RN$. Then, we have
\begin{align}\label{limH(un.nabla phi.)}
	\lim_{\delta \to 0^+}\lim_{n \to \infty} \int_{\RN} \mathcal{H}\left(x,|u_n \nabla \phi_{i,\delta}|\right) \diff x=0.
\end{align}
\end{lemma}
\begin{proof}
By Theorem~\ref{T.mainE} (ii), it holds
\begin{equation}\label{PL3.6-1}
	\lim_{n \to \infty} \int_{\RN} \mathcal{H}(x,|u_n \nabla \phi_{i,\delta}|) \diff x = \lim_{n \to \infty} \int_{B_{\delta}(x_i)} \mathcal{H}(x,|u_n \nabla \phi_{i,\delta}|) \diff x = \int_{B_{\delta}(x_i)} \mathcal{H}(x,|u \nabla \phi_{i,\delta}|) \diff x.
\end{equation}
Let $\mathcal{G}^\ast$ be given in \eqref{def_G*}. Then by Theorem \ref{T.mainE} (i), $u\Big|_{B_\delta(x_i)} \in L^{\mathcal{G}^\ast}(B_\delta(x_i))$. Using this fact and Proposition \ref{prop.Holder} we can estimate
\begin{align}
	\notag
	\int_{B_\delta(x_i)}  \mathcal{H}\left(x,|u\nabla \phi_{i,\delta}|\right)  \diff x & =\int_{B_\delta(x_i)} \left[ |u\nabla \phi_{i,\delta}|^{p(x)}+a(x)|u\nabla \phi_{i,\delta}|^{q(x)} \right]\diff x  \\
	&\leq
	2\left\||u|^{p}\right\|_{L^{\frac{p^\ast(\cdot)}{p(\cdot)}}(B_\delta(x_i))}
	\left\||\nabla \phi_{i,\delta}|^{p} \right\|_{L^{\frac{N}{p(\cdot)}}
		(B_\delta(x_i))} \notag \\
	&  \hspace{2cm} +2\left\|a|u|^{q}\right\|_{L^{\frac{q^\ast(\cdot)}{q(\cdot)}}(B_\delta(x_i))} \left\||\nabla \phi_{i,\delta}|^{q} \right\|_{L^{\frac{N}{q(\cdot)}}(B_\delta(x_i))}.
	\label{u.grad.phi}
\end{align}
Meanwhile, for $m \in \{p,q\}$ it holds
\begin{align*}
	\left\||\nabla \phi_{i,\delta}|^{m} \right\|_{L^{\frac{N}{m(\cdot)}}
		(B_\delta(x_i))}  \leq 1+\left(\int_{B_\delta(x_i)}|\nabla \phi_{i,\delta}(x)|^N \diff x\right)^{\frac{m^+}{N}} = 1+\left(\int_{B_1}|\nabla \phi(y) |^N \diff y \right)^{\frac{m^+}{N}}
\end{align*}
in view of Proposition~\ref{prop.norm-modular}. From this and \eqref{u.grad.phi}, we deduce
\begin{equation*} \label{H(uphi)}
	\lim_{\delta \to 0^+} \int_{B_\delta(x_i)} \mathcal{H}\left(x,|u\nabla \phi_{i,\delta}|\right) \diff x =0,
\end{equation*}
i.e., \eqref{limH(un.nabla phi.)} holds.
\end{proof}

\begin{lemma}\label{L.un.grad.phiR} 
Let $(\calbf{O})$ hold, and let $u_n \rightharpoonup u$ in $X_V$. Let $\psi \in C^\infty(\RN)$ be such that $0 \leq \psi \leq 1$,  $|\nabla \psi| \leq 2$ in $ \RN$, $\psi \equiv 0$ on $B_{1}$ and $\psi \equiv 1 $ on $B_2^c$. For each $R>0$, define $\phi_R(x):=\psi\left(\frac{x}{R}\right)$ for $x \in \RN$. Then, we have
\begin{align}\label{un.grad.phi.R}
	\lim_{R \to \infty}\lim_{ n \to \infty} \int_{B^c_R} \mathcal{H}\left(x,|u_n\nabla \phi_R|\right)\, \diff x =0.
\end{align}
\end{lemma}
\begin{proof}
By Theorem~\ref{T.mainE} again, it holds
\begin{align*}
	\lim_{n \to \infty } \int_{B_R^c}  \mathcal{H}\left(x,|u_n \nabla \phi_R|\right)\diff x = \int_{B_{2R}\setminus \overline{B}_R} \mathcal{H}(x,\big|u\nabla \phi_R\big|)\diff x. 
\end{align*}
As in the proof of Lemma~\ref{L.un.grad.p}, we have that $u\in L^{\mathcal{G}^\ast}(\RN)$ and 
\begin{multline}\label{u.grad.phi.R}
	\int_{B_{2R}\setminus \overline{B}_R} \mathcal{H}(x,\big|u\nabla \phi_R\big|)\diff x \leq 2\big\||u|^{p}\big\|_{L^{\frac{p^\ast(\cdot)}{p(\cdot)}}(B_{2R}\setminus \overline{B}_{R})}\big\||\nabla\phi_R|^{p}\big\|_{L^{\frac{N}{p(\cdot)}}(B_{2R}\setminus \overline{B}_{R})}\\ + 2\big\|a|u|^{q}\big\|_{L^{\frac{q^\ast(\cdot)}{q(\cdot)}}(B_{2R}\setminus\overline{B}_{R})}\big\||\nabla\phi_R|^{q}\big\|_{L^{\frac{N}{q(\cdot)}}(B_{2R}\setminus \overline{B}_{R})}.
\end{multline}
Meanwhile, for $m\in\{p,q\}$ by Proposition~\ref{prop.norm-modular} we obtain
\begin{align*} 
	\left\||\nabla\phi_R|^{m}\right\|_{L^{\frac{N}{m(\cdot)}}(B_{2R}\setminus \overline{B}_{R})} \leq 1+ \left(\int_{B_{2R}\setminus \overline{B}_{R}}|\nabla\phi_R|^N \diff x \right)^{\frac{m^+}{N}} &\leq 1+\left(\frac{2}{R}|B_{2R}\setminus \overline{B}_R|^{1/N}\right)^{m^+} \leq C ,\ \forall R>0,
\end{align*}
where $C$ depends only on $N$, $p^+$ and $q^+$. Combining this with  \eqref{u.grad.phi.R} we get 
\begin{equation*}
	\lim_{R \rightarrow \infty}\int_{B_{2R}\setminus \overline{B}_R} \mathcal{H}(x,\big|u\nabla \phi_R\big|)\diff x=0,
\end{equation*}
i.e, \eqref{un.grad.phi.R} holds.
\end{proof}

We are now in a position to prove main results of this section.
\begin{proof}[\rm \textbf{Proof of Theorem~\ref{Theo.CCP1}}]
Set $v_n:=u_n-u$ for $n \in \N$. Then,
\begin{equation}\label{T.conv.of.v_n-w}
	v_n \rightharpoonup \quad 0 \quad \text{in}\quad  X_V,
\end{equation}
and hence, by Theorem~\ref{T.mainE} (ii), up to a subsequence we have
\begin{equation}\label{T.conv.of.v_n}
			v_n(x) \to  \quad 0 \quad \text{a.a.}\quad  x\in \mathbb{R}^N.
\end{equation}
Using this and Lemma~\ref{L.brezis-lieb}, we obtain
$$\lim_{n\to\infty}\int_{\mathbb{R}^N} \Big|\mathcal{B}(x,|u_n|)
-\mathcal{B}(x,|v_n|)-\mathcal{B}(x,|u|)\Big|\diff x=0.$$
From this and \eqref{T.conv.of.v_n} we deduce 
$$\lim_{n\to\infty}\int_{\mathbb{R}^N} \Big[\phi\mathcal{B}(x,|u_n|)
-\phi\mathcal{B}(x,|v_n|)\Big]\diff x=\int_{\mathbb{R}^N}\phi\mathcal{B}(x,|u|)\diff x,\quad \forall \phi \in C_0(\mathbb{R}^N).$$
Hence
\begin{equation}\label{T.w*-bar.nu}
	\bar{\nu}_n:=\mathcal{B}(\cdot,|v_n|)\overset{\ast }{\rightharpoonup }\bar{\nu}:=\nu-\mathcal{B}(\cdot,|u|) \quad \text{in} \ \ \mathcal{R}(\mathbb{R}^N).
\end{equation}	
On the other hand, it is clear that the sequence of $\bar{\mu}_n:= \mathcal{H}(\cdot,| \nabla v_n|) +  V \mathcal{H}(\cdot,|v_n|)$ ($n \in \mathbb{N}$) is bounded in $L^{1}(\mathbb{R}^N)$. Thus, up to a subsequence, we have
\begin{equation}\label{T.w*-bar.mu}
	\bar{\mu}_n\overset{\ast }{\rightharpoonup } \bar{\mu} \quad \text{in} \ \ \mathcal{R}(\mathbb{R}^N),
\end{equation}
for some nonnegative finite Radon measure $\bar{\mu}$ on $\mathbb{R}^N$. 
We claim that there is $C>0$ such that
\begin{equation} \label{CCP1.St1}
	\|\phi\|_{L_{\bar{\nu}}^{r(\cdot)}(\mathbb{R}^N)} \leq C 	\|\phi\|_{L_{\bar{\mu}}^{q(\cdot)}(\mathbb{R}^N)}, \, \forall \phi \in C_c^\infty(\mathbb{R}^N).
\end{equation}
To this end, let $\phi \in C_c^\infty (\mathbb{R}^N)$ be arbitrary but fixed. If $\|\phi\|_{L^{r(\cdot)}_{\bar{\nu}}(\mathbb{R}^N)}=0$, then \eqref{CCP1.St1} holds automatically. If $\|\phi\|_{L^{r(\cdot)}_{\bar{\nu}}(\mathbb{R}^N)}>0$, then by
\begin{equation}\label{rho_n}
	\lim\limits_{n \to \infty} \|\phi\|_{L^{r(\cdot)}_{\bar{\nu}_n}(\mathbb{R}^N)} =\|\phi\|_{L^{r(\cdot)}_{\bar{\nu}}(\mathbb{R}^N)}
\end{equation}
(see Lemma~\ref{L.convergence}), we may assume that $\psi_n:=\|\phi\|_{L^{r(\cdot)}_{\bar{\nu}_n}(\mathbb{R}^N)}>0$ for all $n \in \mathbb{N}$. Clearly, $\phi v_n \in X_V$ for all $n \in \mathbb{N}$, and thus, by \eqref{def_S}, we have
\begin{equation}\label{S.ineq.1}
	S\|\phi v_n\|_{\mathcal{B}} \leq \|\phi v_n\|, \,\, \forall n \in \mathbb{N}.
\end{equation}
From the boundedness of $\{v_n\}_{n \in \N}$ in $X_V$ and Theorem \ref{T.mainE},  it holds that
\begin{equation}\label{vn_L}
	L:=1+\max \bigg\{\sup_{n\in \mathbb{N}}\int_{\mathbb{R}^N} \mathcal{B}(x,|v_n|) \diff x,\, \sup_{n\in \mathbb{N}}\int_{\mathbb{R}^N} \left[\mathcal{H}(x,| \nabla v_n|) +  V(x)\mathcal{H}(x,| v_n|)\right] \diff x \bigg \}\in [1,\infty).
\end{equation} 
Thus, by using Proposition~\ref{prop.norm-modular}, \eqref{young} and \eqref{vn_L} of $L$, we find $C_L>1$ such that
\begin{align*}
	1& =\int_{\mathbb{R}^N} \left|\frac{\phi }{\psi_n}\right|^{r(x)} \diff \bar{\nu}_n=\int_{\mathbb{R}^N} \left|\frac{\phi }{\psi_n}\right|^{r(x)}\left(c_1(x)|v_n|^{r(x)}+c_2(x)a(x)^{\frac{s(x)}{q(x)}}|v_n|^{s(x)} \right)\diff x\\
	&\leq \int_{\mathbb{R}^N} c_1(x)\left|\frac{\phi v_n}{\psi_n}\right|^{r(x)} \diff x +  \int_{\mathbb{R}^N}c_2(x) a(x)^{\frac{s(x)}{q(x)}}|v_n|^{s(x)} \left(\frac{1}{2L} + C_L \left|\frac{\phi }{\psi_n}\right|^{s(x)} \right)  \diff x \\
	&\leq C_L \int_{\mathbb{R}^N}  \mathcal{B}\left(x, \left|\frac{\phi v_n}{\psi_n}\right|\right)\diff x +\frac{1}{2}, \,\, \forall n \in \mathbb{N}.
\end{align*}
Consequently, it holds
\begin{align*}
	\int_{\mathbb{R}^N}  \mathcal{B}\left(x, \left|\frac{(2C_L)^{\frac{1}{r^-}}\phi v_n}{\psi_n}\right|\right)\diff x \geq 1, \,\, \forall n \in \mathbb{N}.
\end{align*}
By Proposition~\ref{prop.nor-mod.D}, the last inequality yields
\begin{align*}
	0<	(2C_L)^{\frac{-1}{r^-}}\|\phi\|_{L^{r(\cdot)}_{\bar{\nu}_n}(\mathbb{R}^N)} \leq \|\phi v_n\|_{\mathcal{B}}, \,\, \forall n \in \mathbb{N}.
\end{align*}
From this and \eqref{rho_n} we obtain
\begin{equation} \label{lim.phivn}
	0< (2C_L)^{\frac{-1}{r^-}} \|\phi\|_{L^{r(\cdot)}_{\bar{\nu}}(\mathbb{R}^N)} \leq \liminf_{n \to \infty}  \|\phi v_n\|_{\mathcal{B}}.
\end{equation}

Let $\beta_n:=\|\phi v_n\|$ for $n\in \mathbb{N}$. Obviously, $\{\beta_n\}_{n \in \N}$ is bounded in $\R$; hence, along a subsequence we have
\begin{equation} \label{lim.beta_n}
	\lim_{n \to \infty} \beta_n =\bar{\beta}.
\end{equation}
By \eqref{S.ineq.1} and \eqref{lim.phivn}, it holds $\bar{\beta}>0$; hence  we may assume that $\beta_n>0$ for all $n \in \mathbb{N}$. In view of \eqref{modular-norm.XV} we have
\begin{align}\label{betan.1}
	\notag
	1& = \int_{\mathbb{R}^N} \mathcal{H}\left(x,\left|\frac{ v_n\nabla \phi +\phi \nabla v_n}{\beta_n}\right|\right)\diff x  +  \int_{\mathbb{R}^N} V(x) \mathcal{H}\left(x,\left|\frac{\phi v_n}{\beta_n}\right|\right)\diff x    \\ 
	&\leq
	2^{q^+}\int_{\mathbb{R}^N} \mathcal{H}\left(x,\left|\frac{\phi \nabla v_n}{\beta_n}\right|\right)\diff x   + 2^{q^+}\int_{\mathbb{R}^N} \mathcal{H}\left(x,\left|\frac{v_n\nabla \phi }{\beta_n}\right|\right)\diff x + \int_{\mathbb{R}^N}V(x) \mathcal{H}\left(x,\left|\frac{\phi v_n}{\beta_n}\right|\right)\diff x.
\end{align}
Note that we have
\begin{equation*}
	|t|^{p(x)}\leq \eps+\left(1+\eps^{-(q-p)^+}\right)|t|^{q(x)},\quad \forall \eps>0,\ \forall t\in\R,\ \forall x\in \RN,
\end{equation*}
and therefore,
\begin{equation}\label{younga}
	\mathcal{H}(x,|t\eta|) \leq \eps |t|^{p(x)} + \left(1+\eps^{-(q-p)^+}\right) |\eta|^{q(x)} \mathcal{H}(x,|t|),\quad \forall \eps>0,\ \forall t,\eta\in\R,\ \forall x\in \RN.
\end{equation}
By means of \eqref{younga}, we find $\widetilde{C}_L>1$ depending only on $L$ and data such that
\begin{align*}
	\int_{\mathbb{R}^N} \mathcal{H}\left(x,\left|\frac{\phi \nabla v_n}{\beta_n}\right|\right)\diff x \leq \frac{1}{2^{q^++2}L}\int_{\mathbb{R}^N}|\nabla v_n|^{p(x)}\, \diff x+ \widetilde{C}_L 	\int_{\mathbb{R}^N} \left|\frac{\phi}{\beta_n}\right|^{q(x)} \mathcal{H}\left(x,|\nabla v_n|\right)\, \diff x
\end{align*}
and 
\begin{align*}
	\int_{\mathbb{R}^N}V(x) \mathcal{H}\left(x,\left|\frac{\phi v_n}{\beta_n}\right|\right)\,\diff x \leq \frac{1}{4L} \int_{\mathbb{R}^N}V(x)| v_n|^{p(x)}\, \diff x +  \widetilde{C}_L \int_{\mathbb{R}^N} \left|\frac{\phi}{\beta_n}\right|^{q(x)}V(x) \mathcal{H}\left(x,|v_n|\right)\, \diff x
\end{align*}
for all $n\in\N$. Utilizing the last two inequalities and recalling \eqref{vn_L}, we derive from \eqref{betan.1} that
\begin{align}\label{betan.2}
	\frac{1}{2} \leq 2^{q^+}\widetilde{C}_L \int_{\mathbb{R}^N} \left|\frac{\phi}{\beta_n}\right|^{q(x)} \diff \bar{\mu}_n   + 2^{q^+}  \int_{\mathbb{R}^N} \mathcal{H}\left(x,\left|\frac{ v_n\nabla \phi }{\beta_n}\right|\right) \diff x.
\end{align} 
Fixing $R>0$ such that $\spp (\phi) \subset B_R$, one has
\begin{align}\label{nab.phi}
	\int_{\mathbb{R}^N} \mathcal{H}\left(x,\left|\frac{ v_n\nabla \phi }{\beta_n}\right|\right) \diff x \leq \max\left\{\frac{1}{\beta_n^{p^-}},\frac{1}{\beta_n^{q^+}}\right\} \left(1+\|\,|\nabla \phi|\,\|_{L^\infty(\mathbb{R}^N)}^{q^+}\right) \int_{B_R} \mathcal{H}(x,|v_n|)\diff x.
\end{align}
Invoking  Theorem~\ref{T.mainE} (ii) and \eqref{T.conv.of.v_n-w}, we obtain 
\begin{align*}
	\lim_{ n \to \infty} \int_{B_R} \mathcal{H}(x,|v_n|)\diff x =0. 
\end{align*}
From this, \eqref{lim.beta_n} and \eqref{nab.phi} we deduce
\begin{equation}\label{PT3.1.lH}
	\lim_{n \to \infty} \int_{\mathbb{R}^N} \mathcal{H}\left(x,\left|\frac{ v_n\nabla \phi }{\beta_n}\right|\right) \diff x =0.
\end{equation}
Passing to the limit  in \eqref{betan.2} as $n \to \infty$ and using \eqref{T.w*-bar.mu}, \eqref{lim.beta_n} and \eqref{PT3.1.lH}, we obtain
\begin{align*}
	\frac{1}{2} \leq 2^{q^+}\widetilde{C}_L \int_{\mathbb{R}^N} \left|\frac{\phi}{\bar{\beta}}\right|^{q(x)} \diff \bar{\mu},
\end{align*}
and thus, 
\begin{equation*}
	\int_{\mathbb{R}^N} \left|\frac{\phi}{(2^{q^++1}\widetilde{C}_L)^{-\frac{1}{q^-}}\bar{\beta}}\right|^{q(x)} \diff \bar{\mu} \geq 1.
\end{equation*}
Hence, by Proposition~\ref{prop.norm-modular} we arrive at
\begin{equation} \label{R.side.2}
	0<\bar{\beta}=\lim_{n\to \infty} \|\phi v_n\| \leq (2^{q^++1}\widetilde{C}_L)^{\frac{1}{q^-}} \|\phi\|_{L_{\bar{\mu}}^{q(\cdot)}(\mathbb{R}^N)}.
\end{equation}
Invoking \eqref{S.ineq.1}, \eqref{lim.phivn}  and   \eqref{R.side.2} we deduce \eqref{CCP1.St1}; hence, \eqref{T.ccp.form.nu} follows in view of Lemma~\ref{L.reserveHolder}.

Next, we show that   $\{x_i\}_{i\in I }\subset \mathscr{C}$.
Suppose contrarily that  there exists $x_i\in \mathbb{R}^N\setminus\mathscr{C}$ for some $i\in I$. Thanks to the closedness of $\mathscr{C}$ we find $\delta>0$ such that $\overline{B_{2\delta}(x_i)}\subset \mathbb{R}^N\setminus\mathscr{C}$. Then by $(\calbf{C})$, we have that $r(x)<p^\ast(x),\ s(x)<q^\ast(x)$ for all  $x \in \overline{B_{\delta}(x_i)}$. Using Theorem~\ref{T.mainE} (ii) and \eqref{un_weak.conv} we obtain 
\begin{equation*}
	\lim\limits_{n \to \infty}\int_{B_{\delta}(x_i)} \mathcal{B}(x,|u_n|) \diff x =\int_{B_{\delta}(x_i)}\mathcal{B}(x,|u|) \diff x.
\end{equation*}
Using this and applying \cite[Proposition 1.203]{Fonseca}, we  have 
\begin{equation*}
	\nu\left(B_{\delta}(x_i)\right) \leq \liminf_{n \to \infty}\int_{B_{\delta}(x_i)}\mathcal{B}(x,|u_n|)\diff x 	= \int_{B_{\delta}(x_i)}\mathcal{B}(x,|u|) \diff x.
\end{equation*}
On the other hand, 	  taking into account \eqref{T.ccp.form.nu}  we infer
\begin{equation*}
	\nu\left(B_{\delta}(x_i)\right) \geq \int_{B_{\delta}(x_i)}\mathcal{B}(x,|u|)\diff x+\nu_i > \int_{B_{\delta}(x_i)}\mathcal{B}(x,|u|)\diff x, 
\end{equation*}
a contradiction. So $\{x_i\}_{i\in I}\subset \mathscr{C}$.

Next, we aim to prove \eqref{T.ccp.nu_mu} when $I\ne\emptyset$. Fix $i \in I$ and take $\delta>0$. Let $\phi_{i,\delta}$ be  as in Lemma \ref{L.un.grad.p}, and for a function $h \in  C_+(\mathbb{R}^N) $,  denote
$$h^+_\delta:= \sup_{x\in B_{\delta}(x_i)} h(x) \quad \text{and} \quad h^-_\delta:= \inf_{x\in B_{\delta}(x_i)} h(x). $$
Since $\phi_{i,\delta}u_n \in X_V$, by \eqref{def_S} we have
\begin{equation} \label{S_Phi_del}
	S \|\phi_{i,\delta} u_n\|_{\mathcal{B}} \leq \|\phi_{i,\delta} u_n\|.
\end{equation}
Applying Proposition~\ref{prop.nor-mod.D}  gives
\begin{align*}
	\|\phi_{i,\delta} u_n\|_{\mathcal{B}} & \geq \min \left\{ \left(\int_{B_\delta(x_i)} \mathcal{B}(x,|\phi_{i,\delta} u_n|)\diff x\right)^{\frac{1}{r^-_\delta}} , \left(\int_{B_\delta(x_i)} \mathcal{B}(x,|\phi_{i,\delta} u_n|)\diff x\right)^{\frac{1}{s^+_\delta}} \right\} \\
	& \geq \min \left\{ \left(\int_{B_{\delta/2}(x_i)} \mathcal{B}(x,| u_n|)\diff x\right)^{\frac{1}{r^-_\delta}} , \left(\int_{B_{\delta/2}(x_i)} \mathcal{B}(x,| u_n|)\diff x\right)^{\frac{1}{s^+_\delta}} \right\}.
\end{align*}
It follows from the last inequality, \eqref{nu_n-to-nu*} and \cite[Proposition 1.203]{Fonseca} that
\begin{equation*}
	\liminf_{n \to \infty} 	\|\phi_{i,\delta} u_n\|_{\mathcal{B}} \geq  \min \left\{\nu(B_{\delta/2}(x_i))^{\frac{1}{r^-_\delta}}, \nu(B_{\delta/2}(x_i))^{\frac{1}{s^+_\delta}} \right \}.
\end{equation*}
Thus, by the continuity of $r$ and $s$, we have
\begin{equation} \label{S.nu_i}
	\liminf_{\delta \to 0^+}\liminf_{n \to \infty} 	\|\phi_{i,\delta} u_n\|_{\mathcal{B}} \geq \min \left\{\nu_i^{\frac{1}{r(x_i)}}, \nu_i^{\frac{1}{s(x_i)}} \right \}.
\end{equation}
On the other hand, by using \eqref{modular-norm.XV} and the fact that $\operatorname{supp} (\phi_{i,\delta})\subset B_\delta(x_i)$ we obtain
\begin{equation}\label{S.mu_i.1}
	\|\phi_{i,\delta} u_n\| \leq \max\left\{(	I_{n,\delta})^{\frac{1}{p_{\delta}^-}}, 	(I_{n,\delta})^{\frac{1}{q_\delta^+}}\right\},
\end{equation}
with
\begin{equation*}
	I_{n,\delta}:=\int_{B_\delta(x_i)} \Big[\mathcal{H}\left(x,|\phi_{i,\delta} \nabla u_n + u_n  \nabla \phi_{i,\delta} |\right) +  V(x) \mathcal{H}(x,|\phi_{i,\delta} u_n|)\Big]\diff x.
\end{equation*}
By applying \eqref{Ineq2a}, for any given $\eps>0$  we find $C_\eps>1$ independent of $n$ and $\delta$ such  that
\begin{align}
	I_{n,\delta}
	\leq  (1+\eps) \int_{\RN} \phi_{i,\delta} \Big[\mathcal{H}(x,|\nabla u_n|) +  V(x)\mathcal{H}(x,|u_n|)  \Big] \diff x + C_\eps \int_{\RN} \mathcal{H}(x,|u_n \nabla \phi_{i,\delta}| ) \diff x.  \label{Jn.delta.1}
\end{align}
This and \eqref{mu_n to mu*} yield
\begin{align*}
	\limsup_{n \to \infty} 	I_{n,\delta}
	&\leq  (1+\eps) \int_{\RN} \phi_{i,\delta} \diff \mu + C_\eps \lim_{n\to\infty}\int_{\RN} \mathcal{H}(x,|u_n \nabla \phi_{i,\delta}| ) \diff x\\
	& \leq  (1+\eps)\mu\left(\overline{B_\delta(x_i)}\right)+ C_\eps \lim_{n\to\infty}\int_{\RN} \mathcal{H}(x,|u_n \nabla \phi_{i,\delta}| ) \diff x.
\end{align*}
Using this and Lemma \ref{L.un.grad.p}, we derive from \eqref{S.mu_i.1} that
\begin{align} \label{S.mu_i.2}
	\limsup_{\delta \to 0^+}\limsup_{n \to \infty} \|\phi_{i,\delta} u_n\| \leq (1+\eps) \max \left\{\mu_i^{\frac{1}{p(x_i)}}, \mu_i^{\frac{1}{q(x_i)}} \right \},
\end{align}
with $\mu_i=\mu(\{x_i\})$. 
Gathering \eqref{S_Phi_del}, \eqref{S.nu_i}, \eqref{S.mu_i.2} and the fact that $\{x_i\}_{i \in I} \subset \mathscr{C}$ we infer
\begin{equation*}
	S\min \left\{\nu_i^{\frac{1}{p^\ast(x_i)}}, \nu_i^{\frac{1}{q^\ast(x_i)}} \right \}  \leq (1+\eps) \max \left\{\mu_i^{\frac{1}{p(x_i)}}, \mu_i^{\frac{1}{q(x_i)}} \right\}.
\end{equation*}
From this we obtain \eqref{T.ccp.nu_mu} since $\eps>0$ was chosen arbitrarily. In particular, $\{x_i\}_{i \in I}$ are atoms of $\mu$.

We complete the proof by showing \eqref{T.ccp.form.mu}.
Clearly, for any $\phi\in C_0(\mathbb{R}^N)$  with $\phi\geq 0$, the functional $u\mapsto \int_{\mathbb{R}^N}\phi(x)\left[\mathcal{H}(x,|\nabla u|) + V(x) \mathcal{H}(x,| u|)\right]\diff x$ is convex and differentiable on $X_V$. Hence, it is weakly lower semicontinuous and therefore
\begin{multline*}
	\int_{\mathbb{R}^N}\phi(x)\Big[\mathcal{H}(x,|\nabla u|) + V(x) \mathcal{H}(x,| u|)\Big]\diff x\\
	 \leq \liminf_{n\to  \infty } \int_{\mathbb{R}^N }\phi(x)\Big[\mathcal{H}(x,|\nabla u_n|) + V(x) \mathcal{H}(x,|u_n|)\Big]\diff x =\int_{\mathbb{R}^N}\phi \diff \mu.
\end{multline*}
Thus, $\mu\geq \mathcal{H}(\cdot,|\nabla u|) +  V \mathcal{H}(\cdot,| u|) $. Since $\mu_*:=\sum_{i\in I} \mu_i \delta_{x_i}$ and $\mu_{**}:=\mathcal{H}(\cdot,|\nabla u|) +  V \mathcal{H}(\cdot,| u|)$ are orthogonal, \eqref{T.ccp.form.mu} follows. The proof is complete.
\end{proof}


\begin{proof}[\rm \textbf{Proof of Theorem~\ref{Theo.CCP.inf}}]
Let $R>0$ and $n\in\N$. Let $\phi_R$ be as in Lemma \ref{L.un.grad.phiR}, and define
$$T_n(x):=\mathcal{H}(x,|\nabla u_n|) +  V(x) \mathcal{H}(x,| u_n|)\quad\text{for}\ x\in\RN.$$
Then, we decompose
\begin{align}
	\int_{ \mathbb{R}^N}T_n(x)\diff x =\int_{ \mathbb{R}^N}\phi_R^{p(x)}T_n(x)\diff x 
	+\int_{ \mathbb{R}^N}(1-\phi_R^{p(x)})T_n(x)\diff x. \label{T.ccp.inf.decompose1}
\end{align} 
By estimating 
\begin{align*}
	\int_{ B^c_{2R}} T_n(x)\diff x  \leq \int_{\mathbb{R}^N}\phi_R^{p(x)} T_n(x)\diff x
	\leq \int_{B^c_R}  T_n(x)\diff x,
\end{align*}
we obtain
\begin{equation}\label{T.ccp.inf.mu_inf}
	\lim_{R\to\infty}\limsup_{n \to \infty}\int_{\mathbb{R}^N}\phi_R^{p(x)} T_n(x)\diff x = \mu_\infty.
\end{equation}
On the other hand, \eqref{mu_n to mu*} and the fact that $1-\phi_R^{p(x)} \in C_c(\mathbb{R}^N)$ give
\begin{equation}\label{T.ccp.infinity.phiR}
	\lim_{n\to\infty}\int_{\mathbb{R}^N}\left(1-\phi_R^{p(x)}\right)T_n(x)\diff x=\int_{\mathbb{R}^N}\left(1-\phi_R^{p(x)}\right)\diff \mu.
\end{equation}
Note that $\lim\limits_{R\to\infty}\int_{\mathbb{R}^N}\phi_R^{p(x)}\diff\mu=0$ in view of the Lebesgue dominated convergence theorem. Thus,  \eqref{T.ccp.infinity.phiR} yields
\begin{equation} \label{T.ccp.mu(RN)}
	\lim_{R\to\infty}\lim_{n\to \infty }\int_{\mathbb{R}^N}\left(1-\phi_R^{p(x)}\right) T_n(x) \diff x=\mu(\mathbb{R}^N).
\end{equation}
Using \eqref{T.ccp.inf.mu_inf} and \eqref{T.ccp.mu(RN)}, we obtain \eqref{T.ccp.infinity.mu} by taking limit superior as $n\to\infty$ and then taking limit as $R\to\infty$ in \eqref{T.ccp.inf.decompose1}.

In the same manner, we decompose
\begin{equation}\label{T.ccp.inf.decompose2}
	\int_{ \mathbb{R}^N}\mathcal{B}(x,|u_n|)\diff x=\int_{ \mathbb{R}^N}\phi_R^{s(x)} \mathcal{B}(x,|u_n|)\diff x  + \int_{ \mathbb{R}^N}\left(1-\phi_R^{s(x)}\right)\mathcal{B}(x,|u_n|)\diff x.
\end{equation}
Similar arguments to those leading to \eqref{T.ccp.inf.mu_inf} and \eqref{T.ccp.mu(RN)} give
\begin{equation}\label{T.ccp.inf.nu_inf}
	\lim_{R\to\infty}\limsup_{n \to \infty}\int_{\mathbb{R}^N}\phi_R^{s(x)} \mathcal{B}(x,|u_n|)\diff x = \nu_\infty
\end{equation}
and 
\begin{equation}\label{T.ccp.inf.nu(RN)}
	\lim_{R\to\infty}\lim_{n\to \infty }\int_{\mathbb{R}^N}\left(1-\phi_R^{s(x)}\right) \mathcal{B}(x,|u_n|)\diff x=\nu(\mathbb{R}^N).
\end{equation}
Making use of \eqref{T.ccp.inf.decompose2}--\eqref{T.ccp.inf.nu(RN)}, we obtain \eqref{T.ccp.infinity.nu}.

Finally, we claim \eqref{T.ccp.inf.nu_mu} when $(\calbf{E}_\infty)$ is additionally assumed. Let $\epsilon\in(0,p_\infty)$ be arbitrary and fixed. Then by $(\calbf{E}_\infty)$, we find $R_0>0$ such that
\begin{equation}\label{T.ccp.inf.h}
	|h(x)-h_\infty|<\epsilon, \quad \forall x \in B^c_{R_0}
\end{equation}
for each  $h \in \{r,s,p,q\}$. Obviously, $\phi_Ru_n \in X_V$, then by \eqref{def_S} we have
\begin{align} \label{T.ccp.inf.S.phiR}
	S\|\phi_R u_n\|_{\mathcal{B}}& \leq \|\phi_R u_n\|.
\end{align}
For $R>R_0$, using  \eqref{T.ccp.inf.h} and Proposition~\ref{prop.nor-mod.D} we have
\begin{align*}
	\|\phi_Ru_n\|_{\mathcal{B}} 
	&\geq \min\left\{\left(\int_{B_R^c}\phi_R^{s(x)}\mathcal{B}(x,|u_n|) \diff x\right)^{\frac{1}{r_\infty-\epsilon}}, \left(\int_{B_R^c}\phi_R^{s(x)} \mathcal{B}(x,|u_n|) \diff x\right)^{\frac{1}{s_\infty+\epsilon}}\right\}.
\end{align*}
From this and  \eqref{T.ccp.inf.nu_inf}, we obtain
\begin{equation} \label{nu.inf1}
	\liminf_{R\to\infty}\limsup_{n\to \infty} \|\phi_R u_n\|_{\mathcal{B}} \geq \min\biggl\{\nu_\infty^{\frac{1}{r_\infty-\epsilon}},\,\nu_\infty^{\frac{1}{s_\infty+\epsilon}} \biggr\}.
\end{equation}
On the other hand, for $R>R_0$, from \eqref{m-n}  it holds that
\begin{align} \label{T.ccp.inf.phiR_un.1}
	\|\phi_R u_n\|
	\leq & \max\bigg\{\left( I_{n,R}\right)^{\frac{1}{p_\infty-\epsilon}},\left( I_{n,R}\right)^{\frac{1}{q_\infty+\epsilon}}\bigg\},
\end{align}
with $I_{n,R}:=\displaystyle \int_{B_R^c} \Big[ \mathcal{H}(x,|\nabla(\phi_Ru_n)|) +  V(x)\mathcal{H}(x,|\phi_Ru_n|) \Big] \diff x$.
Using \eqref{Ineq2a} again, we find $C_\epsilon>1$ independent of $n,R$ such that
\begin{align*}
	\notag
	I_{n,R}&=  \int_{B_R^c} \Big[\mathcal{H}(x,| u_n \nabla\phi_R +  \phi_R \nabla u_n| ) +   V(x)\mathcal{H}(x,|\phi_Ru_n|)\Big] \diff x   \notag \\
	&\leq (1+\epsilon) \int_{B_R^c} \phi^{p(x)}_R \Big[ \mathcal{H}(x,|\nabla u_n|)+ V(x) \mathcal{H}(x,|u_n|)\Big] \diff x + C_\epsilon \int_{B_R^c}  \mathcal{H}(x,|u_n \nabla \phi_R|)\diff x,
\end{align*}
i.e.,\begin{align}\label{JnR}
		I_{n,R}\leq(1+\epsilon) \int_{B_R^c} \phi^{p(x)}_R T_n(x) \diff x + C_\epsilon \int_{B_R^c}  \mathcal{H}(x,|u_n \nabla \phi_R|)\diff x.
\end{align}
Using \eqref{T.ccp.inf.mu_inf} and Lemma~\ref{L.un.grad.phiR}, we derive from \eqref{JnR} that
\begin{equation*}
	\limsup_{R \to \infty}\limsup_{n \to \infty}I_{n,R} \leq (1+\epsilon)\mu_\infty.
\end{equation*}
By this and \eqref{T.ccp.inf.phiR_un.1}, we obtain
\begin{equation} \label{mu.inf1}
	\limsup_{R \rightarrow \infty}\limsup_{n \to \infty} \|\phi_R u_n\| \leq (1+\epsilon)^{\frac{1}{p_\infty-\epsilon}}  \max \left\{\mu_\infty^{\frac{1}{p_\infty-\epsilon}},\mu_\infty^{\frac{1}{q_\infty+\epsilon}} \right \}.
\end{equation}
From \eqref{T.ccp.inf.S.phiR}, \eqref{nu.inf1} and \eqref{mu.inf1} we derive
\begin{equation*}
	S\min\biggl\{\nu_\infty^{\frac{1}{s_\infty+\epsilon}} , \nu_\infty^{\frac{1}{r_\infty-\epsilon}} \biggr\} \leq  (1+\epsilon)^{\frac{1}{p_\infty-\epsilon}} \max \left\{\mu_\infty^{\frac{1}{p_\infty-\epsilon}},\mu_\infty^{\frac{1}{q_\infty+\epsilon}} \right \}.
\end{equation*}
Hence, \eqref{T.ccp.inf.nu_mu} follows since $\epsilon \in (0,p_\infty)$ was taken  arbitrarily. The proof is complete.
\end{proof}
We conclude this section by proving Theorem~\ref{Theo.CCP2}.
\begin{proof}[\rm \textbf{Proof of Theorem~\ref{Theo.CCP2}}] We define $S$ as in \eqref{def_S} with $X_V$ replaced by $W_V^{1,\mathcal{H}}(\mathbb{R}^N)$. Then, by Remark~\ref{Rmk.WH}, it holds $S\in (0,\infty)$. It is straightforward to verify that the proofs of Theorems~\ref{Theo.CCP1} and \ref{Theo.CCP.inf} remain valid when $X_V$ is replaced by $W_V^{1,\mathcal{H}}(\mathbb{R}^N)$.
\end{proof}


\section{The existence and concentration of solutions}\label{existence}

In this section, we will establish the existence and concentration of solutions to problem \eqref{eq1} using variational methods. The concentration-compactness principles, obtained in Section~\ref{ccp}, play a decisive role in our approach.

\subsection{Statements of main results}${}$

\vskip5pt

Let the assumptions $(\calbf{O})$, $(\calbf{C})$ and  $(\calbf{E}_\infty)$ hold. Consider the following problem \begin{eqnarray}\label{eq1'}
-\operatorname{div}\mathcal{A}(x,\nabla u)+ \lambda V(x) A(x,u)= f(x,u)+\theta B(x,u) \quad \text{in}~ \RN,
\end{eqnarray}
where $\mathcal{A}$, $A$ are given by \eqref{A}; $f$, $B$ are defined as
\begin{equation}\label{def.f}
f(x,t):=b_1(x)|t|^{\ell_1(x)-2}t+b_2(x)a(x)^{\frac{\ell_2(x)}{q(x)}}|t|^{\ell_2(x)-2}t,
\end{equation}
\begin{equation}\label{def.B}
B(x,t):=c_1(x)|t|^{r(x)-2}t+c_2(x)a(x)^{\frac{s(x)}{q(x)}}|t|^{s(x)-2}t;
\end{equation}
and $\lambda$, $\theta$ are positive parameters. We assume further  on the nonlinear term $f$ that
\begin{enumerate}
\item[$(\calbf{S})$]  
$0<b_1(\cdot) \in L^{\frac{p^\ast(\cdot)}{p^\ast(\cdot)-\ell_1(\cdot)}}(\mathbb{R}^N)$ and  $0\leq b_2(\cdot) \in L^{\frac{q^\ast(\cdot)}{q^\ast(\cdot)-\ell_2(\cdot)}}(\mathbb{R}^N)$, where $\ell_1, \ell_2 \in C_+(\mathbb{R}^N)$ such that $\ell_1(\cdot) < \ell_2(\cdot)$,  $\ell_1(\cdot) \ll p^\ast(\cdot)$ and $\ell_2(\cdot) \ll q^\ast(\cdot)$.
\end{enumerate}
For each $\lambda>0$, we denote by $\left(X_\lambda,\|\cdot\|_\lambda \right)$ the space obtained by replacing $V$ with $\lambda V$ in $\left(X_V,\|\cdot\|\right)$. By a weak solution of problem  \eqref{eq1'} we mean a function $u\in X_\lambda$ such that
\begin{align*}
\int_{\mathbb{R}^N}\mathcal{A}(x,\nabla u)\cdot \nabla v\, \diff x  & + \int_{\mathbb{R}^N} \lambda V(x)A(x,u)v\, \diff x\\
& =
\int_{\mathbb{R}^N}f(x,u)v \, \diff x + \theta \int_{\mathbb{R}^N} B(x,u)v \, \diff x, \quad
\forall v \in X_\lambda.
\end{align*}
Note that this definition is well defined in views of Theorems \ref{T.mainE} and \ref{Lem.com.sub}. The existence of weak solutions to problem~\eqref{eq1'} with $\lambda$ fixed is as follows.

\begin{theorem} \label{Theo.Superlinear}
Assume that $(\calbf{O})$, $(\calbf{C})$,  $(\calbf{E}_\infty)$  and $(\calbf{S})$ are  satisfied with $q^+<\min\{\ell_1^-,r^-\}$. Then, for $\lambda>0$  given,  there exists  $\theta_0>0$ such that for each $\theta \in (0,\theta_0)$, problem \eqref{eq1'} has a nontrivial weak solution. 
\end{theorem}


As mentioned earlier, we are interested in the concentration of weak solutions at the bottom of the potential  $V$ when $\lambda\to\infty$. For this purpose, we assume further that	
\begin{itemize}
\item [$(\calbf{V}_2)$]  $V\in C(\RN)$, $\Omega_V:=\textup{int} \left(V^{-1}(0) \right)$ is a nonempty bounded smooth domain, and $\overline{\Omega}_V= V^{-1}(0)$.
\end{itemize}
The limit problem associated with problem \eqref{eq1'} is the following:
\begin{equation}\label{eq2} 
\begin{cases}
	-\operatorname{div} \mathcal{A}(x,\nabla u) = f(x,u)+\theta B(x,u) &   \text{ in } \Omega_V,\\
	u = 0 &  \text{ on } \partial \Omega_V.
\end{cases}
\end{equation}
By a weak solution of problem \eqref{eq2} we mean a function $u_* \in W^{1,\mathcal{H}}(\Omega_V)\cap W_0^{1,1}(\Omega_V)$ such that
\begin{align*}
\int_{\Omega_V} \mathcal{A}(x,\nabla u_*)\cdot \nabla v \, \diff x 	= \int_{\Omega_V} f(x,u_*)v \,\diff x + \theta  \int_{\Omega_V} B(x,u_*)v \, \diff x, \quad \forall v \in C_c^{\infty}(\Omega_V).
\end{align*}
Our last main result on the concentration of weak solutions is stated as follows.
\begin{theorem} \label{Theo.Superlinear2}
In addition to the hypotheses of Theorem~\ref{Theo.Superlinear}, assume that $(\calbf{V}_2)$ holds. Then, there exists  $\theta_0>0$ such that for each given $\theta \in (0,\theta_0)$ and any $\lambda>0$, problem \eqref{eq1'} has a nontrivial weak solution $u_\lambda$. Furthermore, if $\lambda_n\to\infty$, then $u_{\lambda_n} \to u_0$ along a subsequence in $X_V$, with $u_0=0$ a.e. in $\Omega_V^c$ and $u_0\big|_{\Omega_V}$ being a nontrivial weak solution of the limit problem \eqref{eq2}.
\end{theorem}


\subsection{Proofs of Theorems \ref{Theo.Superlinear} and \ref{Theo.Superlinear2}}${}$

\vskip5pt

Let the assumptions of Theorem \ref{Theo.Superlinear} hold, and let $\lambda,\theta>0$.  Define  $\widehat{A},F:\RN \times [0,\infty) \to \R$ and $\widehat{F}, \widehat{B}:\RN \times \R \to \R$ as
\begin{eqnarray*}
\widehat{A}(x,t):=\frac{t^{p(x)}}{p(x)}+a(x)\frac{t^{q(x)}}{q(x)},
\end{eqnarray*}
\begin{equation*}\label{def.F}
F(x,t):= b_1(x)t^{\ell_1(x)} + b_2(x)a(x)^{\frac{\ell_2(x)}{q(x)}} t^{\ell_2(x)}, 
\end{equation*}
\begin{equation*}
\widehat{F}(x,t): = \int_0^t f(x,\tau) \diff \tau = b_1(x) \frac{|t|^{\ell_1(x)}}{\ell_1(x)} + b_2(x)a(x)^{\frac{\ell_2(x)}{q(x)}} \frac{|t|^{\ell_2(x)}}{\ell_2(x)}
\end{equation*}
and 
\begin{equation*}
\widehat{B}(x,t): = \int_0^t B(x,\tau) \diff \tau = c_1(x)\frac{|t|^{r(x)}}{r(x)} + c_2(x) a(x)^{\frac{s(x)}{q(x)}}\frac{|t|^{s(x)}}{s(x)}.
\end{equation*}
Note that by Theorem~\ref{Lem.com.sub}, we have
\begin{equation}\label{XV.comp.emd.LF}
X_\lambda \hookrightarrow\hookrightarrow L^{F}(\mathbb{R}^N).
\end{equation}
In order to determine weak solutions to problem~\eqref{eq1'}, we define $J:X_\lambda \to \mathbb{R}$ as
\begin{align*}
J (u)&:=\int_{\mathbb{R}^N}\widehat{A}(x,|\nabla u|)\diff x +  \int_{\mathbb{R}^N}\lambda V(x)\widehat{A}(x,|u|) \diff x - \int_{ \mathbb{R}^N} 	\widehat{F}(x,u) \diff x	-\theta \int_{\mathbb{R}^N}\widehat{B}(x,u)\diff x,\quad u\in X_\lambda.
\end{align*}
By applying a standard argument utilizing Theorem~\ref{T.mainE} and the embedding \eqref{XV.comp.emd.LF}, we can readily demonstrate that $J$ is of class $C^1$, and its Fr\'echet derivative $J':X_\lambda \to( X_\lambda)^\ast$ is given by
\begin{multline}\label{Diff1}
\left\langle J' (u),v\right\rangle=  \int_{\mathbb{R}^N} \mathcal{A}(x,\nabla u)\cdot\nabla v\,\diff x +  \int_{\mathbb{R}^N} \lambda V(x) A(x,u)v \diff x \\ 
-\int_{\mathbb{R}^N}f(x,u)v\,\diff x- \theta \int_{\mathbb{R}^N} B(x,u)v\,\diff x,\quad \forall\, u,v\in X_\lambda.
\end{multline}
It is clear that a critical point of $J$ is a weak solution of problem \eqref{eq1'}. The following result is useful for finding critical points of $J$.

\begin{lemma}\label{Lem.PS1.c}
For any $\theta>0$ and $\lambda > 0$, the functional $J $ satisfies the  $\textup{(PS)}_c$ condition with $c\in \mathbb{R}$ satisfying
\begin{equation}\label{PS1.c.cond}
	c<\round{\frac{1}{q^+}-\frac{1}{r^-}}\min\left\{S^{\tau_1},S^{\tau_2}\right\}\min \left\{\theta^{-\sigma_1},\theta^{-\sigma_2}\right\},
\end{equation}	
where $S$ is given in \eqref{def_S}, $\tau_1:=\left(\frac{ps}{s-p}\right)^-$, $\tau_2:=\left(\frac{qr}{r-q}\right)^+$, $\sigma_1:=\left(\frac{p}{s-p}\right)^-$, and $\sigma_2:=\left(\frac{q}{r-q}\right)^+$.
\end{lemma}

\begin{proof}
Let $\lambda,\theta>0$, and let $c\in\R$ satisfy \eqref{PS1.c.cond}. Let $\{u_n\}_{n \in \mathbb{N}}$ be a $\textup{(PS)}_c$-sequence for $J $ in $X_\lambda$, namely,
\begin{equation}\label{Le.PS1.PS-seq}
	J(u_n)\to c\ \ \text{and}\ \ J'(u_n)\to 0 \quad \text{as} ~ n \to \infty.
\end{equation}
Putting $\alpha:=\min \{\ell_1^-,r^-\}$ and employing estimate \eqref{m-n}, 
we have that for $n$ large,
\begin{align*}
	\notag
	c+1+\|u_n\|_\lambda& \geq J(u_n)-\frac{1}{\alpha}\scal{J'(u_n),u_n}\\ \notag
	&\geq \round{\frac{1}{q^+}-\frac{1}{\alpha}} \left[ \int_{\mathbb{R}^N}\mathcal{H}(x,|\nabla u_n|) \diff x + \int_{\mathbb{R}^N}\lambda V (x)\mathcal{H}(x,|u_n|) \diff x \right]    \\ 
	&\geq  C_1\left(\|u_n\|_\lambda^{p^-}-1\right).
\end{align*}
This implies that $\{u_n\}_{n \in \N}$ is bounded in $X_\lambda$ since $p^->1$.
Then, according to Theorems~\ref{Theo.CCP1} and \ref{Theo.CCP.inf}, it holds, up to a subsequence, that
\begin{gather}
	u_n(x) \to u(x) \quad \text{a.a.} \ \ x \in\mathbb{R}^N,  \label{PS1.}\\
	u_n \rightharpoonup u  \quad \text{in} \  X_\lambda, \label{PS1.weak1}\\
	\mathcal{H}(\cdot,|\nabla u_n|)+ \lambda V \mathcal{H}(\cdot,|u_n|) \overset{\ast }{\rightharpoonup }\mu \geq \mathcal{H}(\cdot,|\nabla u|)+\lambda V \mathcal{H}(\cdot,|u|)+ \sum_{i\in I} \mu_i \delta_{x_i} \ \text{in}\  \mathcal{R}(\mathbb{R}^N),     
	\label{PS1.mu}\\
	\mathcal{B}(\cdot,|u_n|)\overset{\ast }{\rightharpoonup }\nu=\mathcal{B}(\cdot,|u|) + \sum_{i\in I}\nu_i\delta_{x_i} \ \text{in}\ \mathcal{R}(\mathbb{R}^N),\label{PS1.nu}\\
	S \min \left \{\nu_i^{\frac{1}{r(x_i)}},\nu_i^{\frac{1}{s(x_i)}}\right\}  \leq \max \left\{\mu_i^{\frac{1}{p(x_i)}},\mu_i^{\frac{1}{q(x_i)}}\right\}, \quad \forall i\in I,
	\label{PS1.mu.nu}
\end{gather}
and
\begin{gather}
	\limsup_{n \to \infty}\int_{\mathbb{R}^N} \Big[\mathcal{H}(x,|\nabla u_n|)+\lambda V(x) \mathcal{H}(x,|u_n|) \Big] \, \diff x = \mu(\mathbb{R}^N)+\mu_\infty, 
	\label{PS1.mu.inf}\\
	\underset{n\to\infty}{\limsup}\int_{\mathbb{R}^N}\mathcal{B}(x,|u_n|)\diff x = \nu(\mathbb{R}^N)+\nu_\infty,
	\label{PS1.nu.inf}\\
	S  \min\left\{\nu_\infty^{\frac{1}{r_\infty}},\nu_\infty^{\frac{1}{s_\infty}}\right\} \leq \max\left\{\mu_\infty^{\frac{1}{p_\infty}},\mu_\infty^{\frac{1}{q_\infty}}\right\} .\label{PS1.mu.nu.inf}
\end{gather}
We claim that $I=\emptyset$.  Suppose contrarily that there is $i \in I$. For each $\delta>0$, let $\phi_{i,\delta}$ be defined as in Lemma~\ref{L.un.grad.p}. 
For any $n\in\N$, it is clear that $\phi_{i,\delta}u_n \in X_\lambda$; hence, 
\begin{multline}\label{Deri.J}	
	\int_{\mathbb{R}^N} \phi_{i,\delta} \Big[\mathcal{H}(x,|\nabla u_n|) + \lambda V(x) \mathcal{H}(x,| u_n|) \Big] \diff x = 	\langle J'(u_n) ,\phi_{i,\delta}u_n \rangle  + \theta \int_{\mathbb{R}^N}\phi_{i,\delta} \mathcal{B}(x,|u_n|)\diff x \\
	+ \int_{\mathbb{R}^N}\phi_{i,\delta} F(x,|u_n|) \diff x - \int_{\mathbb{R}^N} \mathcal{A}(x,\nabla u_n) \cdot \nabla \phi_{i,\delta} u_n \diff x.
\end{multline}
Let $\epsilon>0$ be arbitrary. Invoking \eqref{young}, we have 
\begin{align} \label{Deri.J.b}
	\int_{\RN} \Big| \mathcal{A}(x,\nabla u_n) \cdot \nabla \phi_{i,\delta} u_n \Big| \diff x & \leq \int_{\RN} |\nabla u_n|^{p(x)-1}|\nabla \phi_{i,\delta}||u_n|\diff x  + \int_{\mathbb{R}^N} a(x)|\nabla u_n|^{q(x)-1}|\nabla \phi_{i,\delta}||u_n|\diff x \notag\\
	&\leq \epsilon \int_{\mathbb{R}^N} \mathcal{H}(x,|\nabla u_n|)\diff x + C_\epsilon \int_{\mathbb{R}^N} \mathcal{H}(x,|\nabla \phi_{i,\delta} u_n|)\diff x\notag\\
	&\leq \epsilon M + C_\epsilon \int_{\mathbb{R}^N} \mathcal{H}(x,|\nabla \phi_{i,\delta} u_n|)\diff x,
\end{align}
where $C_\epsilon>0$ is independent of $n$ and $\delta$, and
\begin{equation} \label{M}
	M:=\sup_{n \in \mathbb{N}} \int_{\mathbb{R}^N}  \left[\mathcal{H}(x,|\nabla u_n|)+\lambda V(x) \mathcal{H}(x,|u_n|)\right] \diff x \in (0,\infty)
\end{equation}
due to the boundedness of  $\{u_n\}_{n \in \N}$ in $X_\lambda$. Utilizing \eqref{Deri.J.b} and the fact that $\spp (\phi_{i,\delta}) \subset B_\delta(x_i)$, we derive from \eqref{Deri.J} that
\begin{align}\label{PS1.est1}  
	\notag
	&\int_{\mathbb{R}^N} \phi_{i,\delta} \Big[\mathcal{H}(x,|\nabla u_n|)+\lambda V(x) \mathcal{H}(x,|u_n|) \Big] \diff x   \leq \langle J'(u_n),\phi_{i,\delta} u_n \rangle + \theta \int_{\mathbb{R}^N} \phi_{i,\delta} \mathcal{B}(x,|u_n|) \diff x  \\
	& \hspace{2cm}  +  \int_{B_\delta(x_i)} \phi_{i,\delta} F(x,|u_n|) \diff x+\epsilon M +C_\epsilon \int_{B_\delta(x_i)} \mathcal{H}(x,| \nabla \phi_{i,\delta} u_n|) \diff x .
\end{align}
Clearly, $\{\phi_{i,\delta}u_n\}_{n \in \mathbb{N}}$ is bounded in $X_\lambda$; hence, \eqref{Le.PS1.PS-seq} implies that
\begin{equation} \label{PS1.lim1}
	\lim_{n \to \infty} \langle J'(u_n),\phi_{i,\delta} u_n \rangle = 0.
\end{equation}
By invoking \eqref{XV.comp.emd.LF}, it follows from \eqref{PS1.weak1} that
\begin{equation}  \label{PS.c.lim2}
	\lim_{n \to \infty} \int_{B_\delta(x_i)} \phi_{i,\delta}  F(x,|u_n|) \diff x = 	\int_{B_\delta(x_i)} \phi_{i,\delta}  F(x,|u|) \diff x.
\end{equation}
Also, as obtained in \eqref{PL3.6-1} we have
\begin{equation} \label{PS1.lim3}
	\lim_{n \to \infty} \int_{B_\delta(x_i)} \mathcal{H}(x,| \nabla \phi_{i,\delta} u_n|) \diff x =  \int_{B_\delta(x_i)} \mathcal{H}(x,| \nabla \phi_{i,\delta} u|) \diff x .
\end{equation}
Passing to the limit as $n \to \infty$  in  \eqref{PS1.est1} and employing \eqref{PS1.mu}--\eqref{PS1.nu},  \eqref{PS1.lim1}--\eqref{PS1.lim3} we arrive at
\begin{equation}\label{PS1.mu.nu.lim}
	\int_{\RN} \phi_{i,\delta}  \diff \mu \leq \theta \int_{\RN}\phi_{i,\delta} \diff \nu+\int_{B_\delta(x_i)} \phi_{i,\delta} F(x,|u|) \diff x+\epsilon M + C_\epsilon\int_{B_\delta(x_i)} \mathcal{H}(x,| \nabla \phi_{i,\delta} u|) \diff x .
\end{equation}
The fact that $F(\cdot,|u|)\in L^1(\RN)$ gives
\begin{equation}\label{PS1.lim4}
	\lim_{\delta \to 0^+} \int_{B_\delta(x_i)} \phi_{i,\delta} F(x,|u|) \diff x = 0.
\end{equation}
On the other hand, applying Lemma \ref{L.un.grad.p} yields
\begin{equation}\label{PS1.lim5}
	\lim_{\delta \to 0^+} \int_{B_\delta(x_i)} \mathcal{H}(x,| \nabla \phi_{i,\delta} u|) \diff x =0.
\end{equation}
By letting $\delta \to 0^+$ in \eqref{PS1.mu.nu.lim} and using \eqref{PS1.lim4}--\eqref{PS1.lim5} we have
\begin{equation}\label{4.m-n}
	\mu_i \leq \theta \nu_i + \epsilon M.
\end{equation}
This leads to
\begin{equation} \label{PL4.3-mn}
	\mu_i \leq \theta \nu_i,
\end{equation}
since $\epsilon>0$ was chosen arbitrarily. From \eqref{PL4.3-mn} and  \eqref{PS1.mu.nu} we obtain
\begin{equation*}  \label{PS1.est2}
	S\min\left\{(\theta^{-1}\mu_i)^{\frac{1}{r(x_i)}},(\theta^{-1}\mu_i)^{\frac{1}{s(x_i)}} \right\} \leq \max\left\{ \mu_i^{\frac{1}{p(x_i)}}, \mu_i^{\frac{1}{q(x_i)}} \right \}.
\end{equation*}
This yields
\begin{equation} \label{PS1.mu_i}
	\mu_i \geq S^{\frac{\eta_i\beta_i}{\beta_i- \eta_i}}\theta^{-\frac{\eta_i}{\beta_i-\eta_i}},
\end{equation}
where $\eta_i \in \{p(x_i),q(x_i)\}$ and $\beta_i \in \{r(x_i),s(x_i)\}$. It is not difficult to see that
\begin{equation*}
	\left(\frac{ps}{s-p}\right)^-\leq \frac{p(x_i)s(x_i)}{s(x_i)-p(x_i)}\leq \frac{\eta_i\beta_i}{\beta_i-\eta_i}\leq \frac{q(x_i)r(x_i)}{r(x_i)-q(x_i)}\leq \left(\frac{qr}{r-q}\right)^+
\end{equation*}
and
\begin{equation*}
	\left(\frac{p}{s-p}\right)^-\leq\frac{p(x_i)}{s(x_i)-p(x_i)}\leq \frac{\eta_i}{\beta_i-\eta_i}\leq \frac{q(x_i)}{r(x_i)-q(x_i)}\leq \left(\frac{q}{r-q}\right)^+.
\end{equation*}
The last two inequalities jointly with \eqref{PS1.mu_i} and \eqref{4.m-n} imply
\begin{equation} \label{PS1.mu_i.2}
	\theta \nu_i \geq	\mu_i \geq \min\left\{S^{\tau_1},S^{\tau_2}\right\}\min \left\{\theta^{-\sigma_1},\theta^{-\sigma_2}\right\},
\end{equation}
where 
\begin{equation}\label{PL4.3-0}
	\tau_1:=\left(\frac{ps}{s-p}\right)^-,\ \tau_2:=\left(\frac{qr}{r-q}\right)^+,\ \sigma_1:=\left(\frac{p}{s-p}\right)^-\  \text{and} \ \sigma_2:=\left(\frac{q}{r-q}\right)^+.
\end{equation}
On the other hand, since $q^+ < \min\{\ell_1^-,r^-\}$, it follows from  \eqref{Le.PS1.PS-seq} and \eqref{PS1.nu.inf} that
\begin{align}\label{PS1.c.nu}
	\notag c &=\lim_{n\to\infty}\left[J(u_n)-\frac{1}{q^+}\langle J'(u_n) ,u_n\rangle \right]\\
	&\geq  \left(\frac{1}{q^+}-\frac{1}{r^-}\right) \theta \, \limsup_{n \to \infty} \int_{\mathbb{R}^N}\mathcal{B}(x,|u_n|) \diff x = \left(\frac{1}{q^+}-\frac{1}{r^-}\right)\theta\left[\nu(\mathbb{R}^N)+\nu_\infty\right]. 
\end{align}
Taking  into account  \eqref{PS1.nu}, \eqref{PS1.mu_i.2} and \eqref{PS1.c.nu} we obtain
\begin{equation}\label{PL4.3-c}
	c \geq   \left(\frac{1}{q^+}-\frac{1}{r^-}\right) \theta \nu_i
	\geq
	\left(\frac{1}{q^+}-\frac{1}{r^-}\right)\min\left\{S^{\tau_1},S^{\tau_2}\right\}\min \left\{\theta^{-\sigma_1},\theta^{-\sigma_2}\right\}.
\end{equation}
This is in contrast to \eqref{PS1.c.cond}; in other words, $I=\emptyset$.

Next, we aim to show that $\nu_\infty=\mu_\infty=0$. To this end, we claim
\begin{equation}\label{PS1.mu.nu.inf.1}
	\mu_\infty \leq \theta\nu_\infty.
\end{equation} 
Indeed, for each $R>0$ let $\phi_R$ be defined as in Lemma~\ref{L.un.grad.phiR}. It follows that
\begin{gather}
	\notag
	\int_{\mathbb{R}^N} \phi_R \Big[\mathcal{H}(x,|\nabla u_n|)+\lambda V(x)\mathcal{H}(x,|u_n|)\Big]\diff x = \Big\langle J'(u_n) ,\phi_R u_n \Big \rangle +  \int_{\mathbb{R}^N}F(x,|u_n|)\phi_R \diff x\\ 
	+ \theta\int_{\mathbb{R}^N}  \mathcal{B}(x,|u_n|) \phi_R \diff x
	- \int_{\mathbb{R}^N}\mathcal{A}(x,\nabla u_n) \cdot \nabla \phi_R u_n  \diff x. \label{PS1.mu.nu.inf.2}
\end{gather}
From the boundedness of $\{\phi_R u_n\}_{n \in \N}$ in $X_\lambda$ and \eqref{Le.PS1.PS-seq}, it follows that
\begin{align}\label{J'un}
	\lim_{ n \to \infty}\Big\langle J'(u_n) ,\phi_R u_n \Big\rangle=0.
\end{align}
Using the same arguments leading to \eqref{T.ccp.inf.mu_inf} and \eqref{T.ccp.inf.nu_inf}, we obtain
\begin{equation}\label{PS1.mu.nu.inf.3}
	\mu_\infty=\lim_{R\to\infty}\limsup_{n\to\infty}\int_{\mathbb{R}^N} \phi_R \Big[\mathcal{H}(x,|\nabla u_n|)+\lambda V(x)\mathcal{H}(x,|u_n|)\Big]\diff x,
\end{equation}
and 
\begin{equation}\label{PS1.nu.inf.est1}
	\nu_\infty=\lim_{R\to\infty}\limsup_{n \to \infty}\int_{\mathbb{R}^N}\phi_R \mathcal{B}(x,|u_n|) \diff x.
\end{equation}
In view of the compact embedding \eqref{XV.comp.emd.LF}, it follows from \eqref{PS1.weak1} that
\begin{equation} \label{PS1.ets.f}
	\lim_{R \rightarrow \infty}\lim_{ n \to \infty}\int_{\mathbb{R}^N} F(x,|u_n|) \phi_R \diff x = \lim_{R \rightarrow \infty} \int_{\mathbb{R}^N}  F(x,|u|) \phi_R \diff x = 0.
\end{equation}
Similar arguments to those leading to \eqref{Deri.J.b} give, for an arbitrary $\eps>0$,
\begin{align} \label{A.nabla.phi}
	\int_{\mathbb{R}^N} \Big| \mathcal{A}(x,\nabla u_n) \cdot \nabla \phi_R u_n \Big| \diff x
	& \leq \eps M + C_\eps \int_{\mathbb{R}^N}\mathcal{H}(x,| \nabla \phi_R u_n|)\diff x,
\end{align}
with $M$ given by \eqref{M} and $C_\eps>0$ independent of $n$ and $R$.
Using Lemma \ref{L.un.grad.phiR} we have
\begin{equation} \label{H.nabla.phiR}
	\lim_{R \to \infty}\limsup_{n \to \infty} 	\int_{\mathbb{R}^N} \mathcal{H}(x,|\nabla \phi_R u_n|) \diff x =0.
\end{equation} 
Taking limit superior in \eqref{PS1.mu.nu.inf.2} as $n \to \infty$ and then taking limit as $R \to \infty$ with taking into account \eqref{J'un}-\eqref{H.nabla.phiR}, we obtain
\begin{equation*}
	\mu_\infty\leq\theta\nu_\infty + \eps M.
\end{equation*}
Hence,  \eqref{PS1.mu.nu.inf.1} holds since $\eps>0$ is small arbitrarily.
Now, suppose on the contrary that $\nu_\infty>0$. From \eqref{PS1.mu.nu.inf} and \eqref{PS1.mu.nu.inf.1},  we have
\begin{equation*}
	S\min\left\{ (\theta^{-1}\mu_\infty)^{\frac{1}{s_\infty}}, (\theta^{-1}\mu_\infty)^{\frac{1}{r_\infty}} \right\} \leq \max\left\{ \mu_\infty^{\frac{1}{p_\infty}}, \mu_\infty^{\frac{1}{q_\infty}} \right\}.
\end{equation*}
This leads to
\begin{equation}\label{PS1.nu.inf.est4}
	\mu_\infty \geq S^{\frac{\eta_\infty\beta_\infty}{\beta_\infty-\eta_\infty}}\theta^{-\frac{\eta_\infty}{\beta_\infty-\eta_\infty}},
\end{equation}
with $\eta_\infty \in \{p_\infty,q_\infty\}$ and $\beta_\infty \in \{r_\infty,s_\infty\}$. Note that the assumptions on exponents yield $p_\infty\leq q_\infty< r_\infty\leq s_\infty$. We have
\begin{equation*}
	\left(\frac{ps}{s-p}\right)^- \leq\frac{p(x) s(x) }{s(x) -p(x)}\leq \frac{q(x) r(x) }{r(x) -q(x) } \leq \left(\frac{qr}{r-q}\right)^+, \quad \forall x \in \mathbb{R}^N.
\end{equation*}
Thus,
\begin{equation}\label{PL4.3-1}
	\left(\frac{ps}{s-p}\right)^- \leq\frac{p_\infty s_\infty }{s_\infty -p_\infty}\leq \frac{q_\infty r_\infty }{r_\infty -q_\infty } \leq \left(\frac{qr}{r-q}\right)^+.
\end{equation}
On the other hand, for $\eta_\infty \in \{p_\infty,q_\infty\}$ and $\beta_\infty \in \{r_\infty,s_\infty\}$ we have
\begin{equation}\label{PL4.3-2}
	\frac{p_\infty s_\infty }{s_\infty -p_\infty }\leq \frac{\eta_\infty \beta_\infty }{\beta_\infty -\eta_\infty }\leq \frac{q_\infty r_\infty }{r_\infty -q_\infty }.
\end{equation}
Combining \eqref{PL4.3-1} with \eqref{PL4.3-2} gives
\begin{equation*}
	\left(\frac{ps}{s-p}\right)^- \leq \frac{\eta_\infty \beta_\infty}{\beta_\infty-\eta_\infty} \leq \left(\frac{qr}{s-p}\right)^+.
\end{equation*}
Similarly, it holds that
\begin{equation*}
	\left(\frac{p}{s-p}\right)^- \leq \frac{\eta_\infty }{\beta_\infty-\eta_\infty} \leq \left(\frac{q}{r-q}\right)^+.
\end{equation*}
The last two estimates, together with \eqref{PS1.mu.nu.inf.1} and \eqref{PS1.nu.inf.est4}, imply that
$$\theta\nu_\infty\geq \mu_\infty \geq \min\left\{S^{\tau_1},S^{\tau_2}\right\}\min \left\{\theta^{-\sigma_1},\theta^{-\sigma_2}\right\},$$	
with $\tau_1$, $\tau_2$, $\sigma_1$ and $\sigma_2$ given in \eqref{PL4.3-0}. From this and \eqref{PS1.c.nu}, we obtain	
\begin{align*}
	c &\geq \left(\frac{1}{q^+}-\frac{1}{r^-}\right)\theta\nu_\infty
	\geq  \left(\frac{1}{q^+}-\frac{1}{r^-}\right)\min\left\{S^{\tau_1},S^{\tau_2}\right\}\min \left\{\theta^{-\sigma_1},\theta^{-\sigma_2}\right\},
\end{align*}
a contradiction. Thus, $\nu_\infty=\mu_\infty=0.$	
By this and the fact that $I=\emptyset$, we deduce from \eqref{PS1.nu} and \eqref{PS1.nu.inf} that
\begin{equation}\label{PL4.3-lim}\underset{n\to\infty}{\lim\sup}\int_{\mathbb{R}^N}\mathcal{B}(x,|u_n|)\diff x=\int_{\mathbb{R}^N}\mathcal{B}(x,|u|)\diff x.
\end{equation}
By \eqref{PS1.} and Fatou's lemma, we obtain
$$
\int_{\mathbb{R}^N}\mathcal{B}(x,|u|)\diff x\leq \liminf_{n\to \infty}\int_{\mathbb{R}^N}\mathcal{B}(x,|u_n|)\diff x.
$$
Combining this with \eqref{PL4.3-lim} gives
$$\lim_{n\to\infty}\int_{\mathbb{R}^N}\mathcal{B}(x,|u_n|)\diff x=\int_{\mathbb{R}^N}\mathcal{B}(x,|u|)\diff x.$$
From this and Lemma~\ref{L.brezis-lieb} we obtain
$$\lim_{n\to\infty}\int_{\mathbb{R}^N}\mathcal{B}(x,|u_n-u|)\diff x=0,$$
and thus, in view of Proposition~\ref{prop.nor-mod.D}, it holds
\begin{equation*}
	u_n \to u \ \  \text{in} \ \ L^\mathcal{B}(\mathbb{R}^N).
\end{equation*}
Finally, analysis similar to that in the last part of \cite[Proof of Lemma 5.1]{HaHo2023}, taking into account Lemma~\ref{Lem.S+}, shows that $u_n \to u$ in $X_\lambda$. The proof is complete.
\end{proof}


The next lemma provides a Mountain Pass geometry of the functional $J$.
\begin{lemma} \label{Lem.J(u).geo}
Let $\lambda >0$ and $\theta>0$. Then, the following assertions hold.
\begin{enumerate}
	\item[(i)] There exist $\delta \in (0,1)$ and $\rho>0$ such that $J(u) \geq \rho $ if $\|u\|_\lambda =\delta$. 
	\item[(ii)] There exists $v\in X_\lambda$ independent of $\theta$  such that $\|v\|_{\lambda} >1$ and $J(v)<0$ for all $\theta>0$.
	\item[(iii)] If $(\calbf{V}_2)$ is additionally assumed, then there exists $v\in X_V$ independently of $\lambda$ and $\theta$ such that $\|v\|_{\lambda} >1$ and $J(v)<0$ for all $\lambda>0$ and $\theta>0$.
\end{enumerate} 
\end{lemma}
\begin{proof}
Let $\theta,\lambda >0$. 
By \eqref{Xl.emb.B} and \eqref{XV.comp.emd.LF}, we find $C_1>1$ such that
\begin{equation}\label{4norms}
	\max\{\|u\|_\mathcal{B},\|u\|_{F}\} \leq C_1 \|u\|_\lambda,~~ \forall u \in X_\lambda.
\end{equation}
For any $u\in X_\lambda$ with $\|u\|_\lambda = \delta\in \left(0,C_1^{-1}\right)$, we apply Proposition~\ref{prop.nor-mod.D}, together with \eqref{m-n} and \eqref{4norms}, to obtain 
\begin{align*}
	J (u)  &\geq \frac{1}{q^+}  \int_{ \mathbb{R}^N} \Big[ \mathcal{H}(x,|\nabla u|) + \lambda V(x) \mathcal{H}(x,| u|) \Big] \diff x -\frac{1}{\ell_1^-} \int_{ \mathbb{R}^N}  F(x,|u|)  \diff x
	-\frac{\theta}{r^-} \int_{ \mathbb{R}^N} \mathcal{B}(x,|u|) \diff x\notag \\
	&\geq \frac{1}{q^+} \|u\|^{q^+}_\lambda -\frac{1}{\ell_1^-}\left(C_1\|u\|_\lambda\right)^{\ell_1^-} - \frac{\theta}{r^-}\left(C_1\|u\|_\lambda\right)^{r^-}= \frac{1}{q^+} \delta^{q^+} -\frac{C_1^{\ell_1^-}}{\ell_1^-}\delta^{\ell_1^-} - \frac{C_1^{r^-}\theta}{r^-}\delta^{r^-}:=\rho.
\end{align*}	
Since $q^+<\min\{\ell_1^-,r^-\}$, we can choose $\delta>0$ small enough such that $\rho>0$, and thus, $(i)$ has been shown. 

In order to get $(ii)$, let us fix $\phi \in X_\lambda \setminus \{0\}$. For any $\tau>1$ we have
\begin{equation*}
	J(\tau\phi) \leq \frac{\tau^{q^+}}{p^-}\int_{\mathbb{R}^N} \Big[ \mathcal{H}(x,|\nabla \phi|) + \lambda V(x) \mathcal{H}(x,|\phi|) \Big]\diff x  -\frac{\tau^{\ell_1^-}}{\ell_2^+} \int_{ \mathbb{R}^N}  F(x,\phi|)  \diff x \to -\infty \ \text{ as } \ \tau \to \infty.
\end{equation*}
Thus, by taking $v=\tau\phi$ with $\tau>0$ large enough, it holds  $\|v\|_\lambda>1$, and $J(v)<0$ for all $\theta>0$. Clearly, $v$ is independent of $\theta$.

Finally, for showing $(iii)$, let us fix $\phi_0 \in C_c^{\infty}(\Omega_V)$ with $\|\phi_0\|=1$; hence, it holds $\|\phi_0\|_\lambda=1$ for all $\lambda>0$. For any $\tau>1$ we have
\begin{equation*}
	J(\tau\phi_0) \leq \frac{\tau^{q^+}}{p^-}\int_{\mathbb{R}^N} \mathcal{H}(x,|\nabla \phi_0|)\diff x -\frac{\tau^{\ell_1^-}}{\ell_2^+} \int_{ \mathbb{R}^N}  F(x,\phi_0|)  \diff x \to -\infty \ \text{ as } \ \tau \to \infty.
\end{equation*}
Thus, by taking $v=\tau\phi_0$ with $\tau>0$ large enough, the conclusion of $(iii)$ follows. The proof is complete.

\end{proof}


Let $\rho$ and $v$ be determined in Lemma~\ref{Lem.J(u).geo} $(i)$-$(ii)$. Define
\begin{equation}\label{c_lamb}
c_{\lambda}:=\inf _{\gamma \in \Gamma_{\lambda}} \max _{0 \leqslant  \tau \leqslant 1} J(\gamma(\tau)) ,
\end{equation}
where
\begin{equation}
\Gamma_{\lambda}:=\left\{\gamma \in C\left([0,1],\, X_\lambda\right): \gamma(0)=0, \gamma(1)=v\right\}.
\end{equation}
From Lemma~\ref{Lem.J(u).geo} $(i)$, the definition of $c_\lambda$, and the fact that $\gamma(\tau)=\tau v\in \Gamma_{\lambda}$, we have
\begin{equation} \label{c_l}
0<\rho\leq c_\lambda\leq\max_{0\leq\tau\leq 1}J(\tau v).
\end{equation}
For all $0\leq \tau \leq 1$, we have
\begin{align*}	
J(\tau v) &\leq \frac{\tau^{p^-}}{p^-} \int_{\mathbb{R}^N} \Big[ \mathcal{H}(x,|\nabla v|) + \lambda V(x) \mathcal{H}(x,|v|) \Big]\diff x  -\frac{\tau^{\ell_2^+}}{\ell_2^+}  \int_{\mathbb{R}^N} F(x,|v|)\diff x \notag \\
&\leq  a_0\tau^{p^-}-b_0\tau^{\ell_2^+}:=g(\tau),
\end{align*}
where 
$$a_0:=\frac{1}{p^-}\displaystyle\int_{\mathbb{R}^N} \Big[ \mathcal{H}(x,|\nabla v|) + \lambda V(x) \mathcal{H}(x,|v|) \Big]\diff x >0\ \ \text{and} \ \ b_0:=\frac{1}{\ell_2^+}\displaystyle\int_{\mathbb{R}^N} F(x,|v|)\diff x>0.$$
Thus,
\begin{align}\label{g}
0<c_\lambda\leq \max _{\tau\in [0,1]}g(\tau)=g\left(d_0\right)=a_0d_0^{p^-}-b_0d_0^{\ell_2^+}.
\end{align}
where $d_0:=\min\left\{1, \left(\frac{a_0p^-}{b_0\ell_2^+}\right)^{\frac{1}{\ell_2^+ - p^-}}\right\}$. 

\vskip5pt
\begin{proof}[\rm \textbf{Proof of Theorem~\ref{Theo.Superlinear}}]
Let $\lambda>0$ be given. Note that with $v$ found in Lemma~\ref{Lem.J(u).geo} $(ii)$, $g(d_0)$ is independent of $\theta$. Thus, we find $\theta_0>0$ such that
\begin{equation} \label{theta_0}
	0<g(d_0)<\round{\frac{1}{q^+}-\frac{1}{r^-}}\min\left\{S^{\tau_1},S^{\tau_2}\right\}\min \left\{\theta^{-\sigma_1}, \theta^{-\sigma_2} \right\},\quad \forall \theta\in (0,\theta_0).
\end{equation}
Let $\theta\in (0,\theta_0)$. From \eqref{g} and \eqref{theta_0}, it holds that
\begin{equation} \label{Bound.level.c}
	0<\rho\leq c_\lambda < \round{\frac{1}{q^+}-\frac{1}{r^-}}\min\{S^{\tau_1},S^{\tau_2}\}\min \{\theta^{-\sigma_1}, \theta^{-\sigma_2} \}.
\end{equation}
By invoking \cite[Lemma 3.1]{Garcia}, we infer from the Mountain Pass geometry of $J$ obtained in Lemma~\ref{Lem.J(u).geo} that there exists a $\textup{(PS)}_{c_\lambda}$-sequence $\{u_n\}_{n\in\N}$ for $J$ in $X_\lambda$. Then, by \eqref{Bound.level.c} and Lemma~\ref{Lem.PS1.c}, we have that $u_n\to u_\lambda$ along a subsequence in $X_\lambda$; hence, $J'(u_\lambda)=0$ and $J(u_\lambda)=c_\lambda>0$. Thus, $u_\lambda$ is a nontrivial weak solution to problem \eqref{eq1'}. The proof is complete.
\end{proof}

\vskip5pt
Next, we aim to prove Theorem~\ref{Theo.Superlinear2} when $(\calbf{V}_2)$ is additionally assumed. Define
$$X_0:=\{u\in X_V:\ u=0\ \ \text{a.e. in } \Omega_V^c\}.$$
Note that for any $u\in X_0$ it holds that $u\big|_{\Omega_V}\in  W^{1,\mathcal{H}}(\Omega_V)\cap W_0^{1,1}(\Omega_V)$. Indeed, for $u\in X_0$, it holds $u\in W^{1,\mathcal{H}}(\mathbb{R}^N)$ due to Theorem~\ref{T.Xv-P}. This implies that $u\big|_{\Omega_V}\in W^{1,\mathcal{H}}(\Omega_V)$ and $u\in W^{1,p^-}(\RN)$. Hence, by \cite[Proposition 9.18]{Bre.Book} we infer $u\big|_{\Omega_V}\in W_0^{1,1}(\Omega_V)$.

Let us consider $\widetilde{J}: X_0\to\R$ defined as
\begin{align*}
	\widetilde{J}(u):=\int_{ \Omega_V} \widehat{A}(x,\nabla u)\diff x 
	- \int_{ \Omega_V} \widehat{F}(x,u)\diff x - \theta \int_{ \Omega_V} \widehat{B}(x,u)\diff x,\quad u \in X_0. 
\end{align*}
Obviously, $\widetilde{J}=J\big|_{X_0}$ and its critical points are weak solutions to problem \eqref{eq2}. 

With $\rho$ and $v$ being determined in Lemma~\ref{Lem.J(u).geo} $(i)$ and $(iii)$, we define $c_\lambda$ and $g(d_0)$ as in \eqref{c_lamb} and \eqref{g}, respectively. Furthermore, define
\begin{equation}\label{c_*}
	c_*:=\inf _{\gamma \in \overline{\Gamma}(\Omega_V)} \max _{0 \leqslant \tau \leqslant 1}\widetilde{J}(\gamma(\tau)), 
\end{equation}
where
\begin{equation}
	\overline{\Gamma}(\Omega_V):=\left\{\gamma \in C\left([0,1], X_0\right): \gamma(0)=0, \gamma(1)=v\right\}.
\end{equation}
Note that $X_0 \subset X_\lambda$ for all $\lambda>0$, thus we obtain
\begin{equation} \label{c_l'}
	0<\rho\leq c_\lambda \leq c_*,\quad \forall \lambda >0,\ \forall\theta>0.
\end{equation}
Since $V\equiv 0$ on $\Omega_V$, arguing as that obtained \eqref{g} we have
\begin{align}\label{g'}
	c_*\leq g\left(d_0\right).
\end{align}
It is worth emphasizing that $g(d_0)$ is now independent of both $\lambda$ and $\theta$.
\begin{proof}[\rm \textbf{Proof of Theorem~\ref{Theo.Superlinear2}}]
Let $\theta_0>0$ be such that 
\begin{equation} \label{theta_0'}
	g(d_0)< \round{\frac{1}{q^+}-\frac{1}{r^-}}\min\{S^{\tau_1},S^{\tau_2}\}\min \{\theta^{-\sigma_1}, \theta^{-\sigma_2} \}, \quad \forall \theta \in (0,\theta_0).
\end{equation}
Let $\theta\in (0,\theta_0)$ be given. From \eqref{c_l'}-\eqref{theta_0'}, it holds that
\begin{equation} \label{Bound.level.c'}
	0<\rho\leq c_\lambda \leq c_*< D_\theta:= \round{\frac{1}{q^+}-\frac{1}{r^-}}\min\{S^{\tau_1},S^{\tau_2}\}\min \{\theta^{-\sigma_1}, \theta^{-\sigma_2} \}, \quad \forall \lambda>0.
\end{equation}
Then, by Lemmas~\ref{Lem.PS1.c} again, for each $\lambda>0$ problem \eqref{eq1'} admits a nontrivial weak solution $u_\lambda$ , which is a critical point of $J$ corresponding to the critical value $c_\lambda$.

Next, let $\lambda_n\to\infty$, and thus we may assume $\lambda_n>1$ for all $n \in \N$. For each $n\in \mathbb{N}$, by denoting $u_n:=u_{\lambda_n}$ we have
\begin{equation}\label{u_n-cn}
J'(u_n)=0~ \text{and} ~ J(u_n)=c_{\lambda_n}.
\end{equation}
Since $c_{\lambda_n}\leq c_*<D_\theta$, up to a subsequence, we have
\begin{equation}\label{PT4.2-c0}
	c_{\lambda_n}\to c_0\leq c_*<D_\theta.
\end{equation}
Thus, $\{u_n\}_{n\in \mathbb{N}}$ is a $\textup{(PS)}_{c_0}$-sequence for $J$ in $X_V$. We claim that $\{u_n\}_{n\in \mathbb{N}}$ is bounded in $X_V$. Indeed, by putting $\alpha=\min\{\ell_1^-,r^-\}$ again, it follows from \eqref{m-n} and \eqref{Bound.level.c'} that
\begin{align*}
D_{\theta} > c_{\lambda_n} = J(u_n)-\frac{1}{\alpha}\scal{J'(u_n),u_n}\geq \left(\frac{1}{q^+}-\frac{1}{\alpha}\right)\left(\|u_n\|^{p^-}_{\lambda_n} -1 \right),\quad \forall n \in \N.
\end{align*}
Thus, there exists $C_\theta>0$ independent of $n$ such that
\begin{equation}\label{ulambda bounded}
||u_n\|\leq \|u_n\|_{\lambda_n} \leq C_\theta, \quad \forall n \in \N.
\end{equation}
Then, by Theorems \ref{Theo.CCP1} and \ref{Theo.CCP.inf} again, we have
\begin{gather}
u_n(x) \to u_0(x) \quad \text{a.a.} \ x \in\mathbb{R}^N,  \label{CC.un.to.u0.aa}\\
u_n \rightharpoonup u_0  \quad \text{in} \  X_V, \label{CC.un.to.u0.w}\\
\mathcal{H}(\cdot,|\nabla u_n|)+ V \mathcal{H}(\cdot,|u_n|) \overset{\ast }{\rightharpoonup }\mu \geq \mathcal{H}(\cdot,|\nabla u_0|)+V \mathcal{H}(\cdot,|u_0|)+ \sum_{i\in I} \mu_i \delta_{x_i} \ \text{in}\  \mathcal{R}(\mathbb{R}^N),     
\label{CC.conv.mu}\\
\mathcal{B}(\cdot,|u_n|)\overset{\ast }{\rightharpoonup }\nu=\mathcal{B}(\cdot,|u_0|) + \sum_{i\in I}\nu_i\delta_{x_i} \ \text{in}\ \mathcal{R}(\mathbb{R}^N),\label{CC.conv.nu}\\
S \min \left \{\nu_i^{\frac{1}{r(x_i)}},\nu_i^{\frac{1}{s(x_i)}}\right\}  \leq \max \left\{\mu_i^{\frac{1}{p(x_i)}},\mu_i^{\frac{1}{q(x_i)}}\right\}, \quad \forall i\in I,
\label{CC.mu.nu}
\end{gather}
and 
\begin{gather}
\limsup_{n \to \infty}\int_{\mathbb{R}^N} \Big[\mathcal{H}(x,|\nabla u_n|)+V(x) \mathcal{H}(x,|u_n|) \Big] \, \diff x = \mu(\mathbb{R}^N)+\mu_\infty, 
\label{CC.mu.inf}\\
\underset{n\to\infty}{\limsup}\int_{\mathbb{R}^N}\mathcal{B}(x,|u_n|)\diff x = \nu(\mathbb{R}^N)+\nu_\infty,
\label{CC.nu.inf}\\
S  \min\left\{\nu_\infty^{\frac{1}{r_\infty}},\nu_\infty^{\frac{1}{s_\infty}}\right\} \leq \max\left\{\mu_\infty^{\frac{1}{p_\infty}},\mu_\infty^{\frac{1}{q_\infty}}\right\} .\label{CC.mu.nu.inf}
\end{gather}
We claim that
\begin{align}\label{u0=0.a.e}
u_0 \in X_0.
\end{align}
Indeed, by Fatou's lemma with taking into account \eqref{m-n}, \eqref{ulambda bounded} and \eqref{CC.un.to.u0.aa}, we obtain
\begin{align*}
0 \leq \int_{\RN} V(x)|u_0|^{p(x)} \diff x & \leq \liminf_{n\to\infty}\int_{\RN} V(x)|u_n|^{p(x)} \diff x = \liminf_{n\to\infty}\frac{1}{\lambda_n}\int_{\RN} \lambda_nV(x)|u_n|^{p(x)} \diff x\\
&  \leq \liminf_{n\to\infty}\frac{1+\|u_n\|^{q^+}_{\lambda_n}}{\lambda_n} =0.
\end{align*}
Thus, $\int_{\RN} V(x)|u_0|^{p(x)} \diff x=0$; hence, $u_0 =0$  a.e. in $\Omega_V^c.$ So, \eqref{u0=0.a.e} has been proved.

Next, we aim to show that
\begin{equation}\label{un.to.u0.XV}
u_n \to u_0 \ \  \text{along a subsequence in} \ \  X_V.
\end{equation}
To this end, we will first prove that
\begin{align}\label{un.to.u0.strong.B1}
u_n \to u_0 \ \  \text{in} \ \  L^{\mathcal{B}}(\RN).
\end{align}
As in the proof of Lemma~\ref{Lem.PS1.c}, this will be done if we can show that $I=\emptyset$ and $\mu_\infty=\nu_\infty=0$.  Suppose, on the contrary, that there exists $i\in I$. For each $\delta>0$, let $\phi_{i,\delta}$ be as in Lemma~\ref{L.un.grad.p}. Note that
\begin{multline*}
\int_{\mathbb{R}^N} \phi_{i,\delta} \Big[\mathcal{H}(x,|\nabla u_n|) + V(x) \mathcal{H}(x,| u_n|) \Big]  \diff x
 \leq \int_{\mathbb{R}^N} \phi_{i,\delta} \Big[\mathcal{H}(x,|\nabla u_n|) + \lambda_nV(x) \mathcal{H}(x,| u_n|) \Big] \diff x\\
= 	\langle J'(u_n) ,\phi_{i,\delta}u_n \rangle - \int_{\mathbb{R}^N} \mathcal{A}(x,\nabla u_n) \cdot \nabla \phi_{i,\delta} u_n \diff x 
+ \int_{\mathbb{R}^N}\phi_{i,\delta} F(x,|u_n|) \diff x + \theta \int_{\mathbb{R}^N}\phi_{i,\delta} \mathcal{B}(x,|u_n|)\diff x.
\end{multline*}
Similar arguments to those leading to \eqref{PL4.3-mn} and \eqref{PL4.3-c} give $\mu_i \leq \theta \nu_i$, and then,
$c_0\geq D_\theta,$ which contradicts \eqref{PT4.2-c0}. So, we have shown that $I=\emptyset$. 

In order to show $\mu_\infty=\nu_\infty=0$, let $\phi_R$ be as in Lemma~\ref{L.un.grad.phiR}. Since 
\begin{align*}
&\int_{\mathbb{R}^N} \phi_R(x)\Big[\mathcal{H}(x,|\nabla u_n|) + V(x) \mathcal{H}(x,| u_n|) \Big]  \diff x \leq \int_{\mathbb{R}^N} \phi_{R}(x) \Big[\mathcal{H}(x,|\nabla u_n|) + \lambda_nV(x) \mathcal{H}(x,| u_n|) \Big] \diff x\\&=\Big\langle J'(u_n) ,\phi_R u_n \Big \rangle +  \int_{\mathbb{R}^N}F(x,|u_n|)\phi_R \diff x+ \theta\int_{\mathbb{R}^N}  \mathcal{B}(x,|u_n|) \phi_R \diff x
- \int_{\mathbb{R}^N}\mathcal{A}(x,\nabla u_n) \cdot \nabla \phi_R u_n  \diff x,
\end{align*}
we argue similarly to those in Lemma \ref{Lem.PS1.c} to obtain $
\mu_\infty=\nu_\infty=0.$ Hence, \eqref{un.to.u0.strong.B1} has been  proved.

Now, we make use of \eqref{un.to.u0.strong.B1} to prove \eqref{un.to.u0.XV}. From Proposition \ref{prop.Holder} and the boundedness of $\{u_n\}$ in $L^{\mathcal{B}}(\RN)$, we have
\begin{align*}
\int_{\RN}& \Big |B(x,u_n)(u_n-u_0) \Big|\, \diff x \\
& \leq \int_{\RN} \Big[c_1(x)|u_n|^{{r(x)-1}} + c_2(x)a(x)^{\frac{s(x)}{q(x)}}|u_n|^{{s(x)-1}}\Big]|u_n-u_0| \diff x\\
& \leq 2\left\|c_1^{\frac{1}{r'}}|u_n|^{r-1}\right\|_{L^{r'(\cdot)}(\RN)}\left\|c_1^{\frac{1}{r}}|u_n-u_0|\right\|_{L^{r(\cdot)}(\RN)} \\
&\ \ \ + 2\left\|\left(c_2a^{\frac{s}{q}}\right)^{\frac{1}{s'}}|u_n|^{s-1}\right\|_{L^{s'(\cdot)}(\RN)}\left\|\left(c_2a^{\frac{s}{q}}\right)^{\frac{1}{s}}|u_n-u_0|\right\|_{L^{s(\cdot)}(\RN)}\\
&\leq C_2\|u_n-u_0\|_{\mathcal{B}},\quad \forall n\in\N.
\end{align*}
Combining this with \eqref{un.to.u0.strong.B1} yields 
\begin{align}\label{B.un}
\lim_{ n \to \infty}\int_{\RN} B(x,u_n)(u_n-u_0) \, \diff x = 0.
\end{align}
We derive from \eqref{XV.comp.emd.LF} and \eqref{CC.un.to.u0.w} that
\begin{align}\label{un-LF}
u_n \to u_0 ~ \text{in} ~ L^{F}(\RN).
\end{align}
Using this and arguing as that leading to \eqref{B.un}, we obtain 
\begin{align}\label{F.un}
\lim_{ n \to \infty}\int_{\RN} f(x,u_n)(u_n-u_0) \, \diff x = 0.
\end{align}
On account of \eqref{u_n-cn}, \eqref{B.un}, and \eqref{F.un}, we infer
\begin{align}\label{An.0}
\lim_{ n \to \infty} \left[\int_{\RN} \mathcal{A}(x,\nabla u_n)\cdot \nabla (u_n-u_0) \diff x + \int_{\RN} \lambda_nV(x)A(x,u_n)(u_n-u_0) \diff x \right]= 0.
\end{align} 
Note that for any $h \in C_+(\RN)$, one has
\begin{equation*}\label{Montone.A}
\left(|\xi|^{h(x)-2}\xi-|\eta|^{h(x)-2}\eta\right)\cdot (\xi-\eta) \geq 0, \quad\forall \xi,\eta \in \mathbb{R}^N,\\ \xi \neq \eta.
\end{equation*}
Using this, we obtain
\begin{align*} 
\notag
0 & \leq \int_{\RN} \Big[\mathcal{A}(x,\nabla u_n)-\mathcal{A}(x,\nabla u_0)\Big]\cdot \nabla (u_n-u_0)\diff x \notag\\
& \hspace{2cm} + \int_{\RN} V(x)\Big[ A(x, u_n)-A(x, u_0)\Big] (u_n-u_0)\diff x \notag \\
& \leq \int_{\RN} \Big[\mathcal{A}(x,\nabla u_n)-\mathcal{A}(x,\nabla u_0)\Big]\cdot \nabla (u_n-u_0)\diff x \notag\\
& \hspace{2cm} + \int_{\RN} \lambda_nV(x)\Big[A(x, u_n)-A(x, u_0)\Big] (u_n-u_0)\diff x.
\end{align*}
Thus,
\begin{align} \label{CC.un.to u0}
	0&	\leq \int_{\RN} \mathcal{A}(x,\nabla u_n)\cdot \nabla (u_n-u_0) \diff x + \int_{\RN} \lambda_nV(x)A(x,u_n)(u_n-u_0) \diff x \notag \\
&\quad - \int_{\RN}  \mathcal{A}(x,\nabla u_0)\cdot \nabla (u_n-u_0) \diff x - \int_{\RN}\lambda_nV(x) A(x,u_0)(u_n-u_0)\diff x.
\end{align}
On account of \eqref{CC.un.to.u0.w}, we infer
\begin{align}\label{An.1}
\lim_{ n \to \infty} \int_{\RN}  \mathcal{A}(x,\nabla u_0)\cdot \nabla (u_n-u_0) \diff x =0.
\end{align}
Since $u_0 = 0$ a.e. in $\Omega_V^c$ and $V=0$ on $\Omega_V$, we have
\begin{align}\label{An.2}
\int_{\RN}\lambda_nV(x) A(x,u_0)(u_n-u_0)\diff x & =\int_{\Omega_V} \lambda_nV(x) A(x,u_0)(u_n-u_0)\diff x=0,\quad \forall n\in\N.
\end{align}
Gathering \eqref{An.0}, \eqref{An.1} and \eqref{An.2}, we derive from \eqref{CC.un.to u0} that
\begin{align*}
\lim_{ n \to \infty}  \bigg\{\int_{\RN} &\Big[\mathcal{A}(x,\nabla u_n)-\mathcal{A}(x,\nabla u_0)\Big]\cdot \nabla (u_n-u_0)\diff x \notag\\
& \hspace{2cm} + \int_{\RN} V(x)\Big[ A(x, u_n)-A(x, u_0)\Big] (u_n-u_0)\diff x\bigg\}=0.
\end{align*}
Utilizing \eqref{An.1} and \eqref{An.2} again, the last limit yields
\begin{align*}
\lim_{ n \to \infty} \int_{\RN} \Big[  \mathcal{A}(x,\nabla u_n)\cdot \nabla (u_n-u_0)\diff x + V(x)A(x, u_n)(u_n-u_0) \Big] \diff x  =0.
\end{align*}
Thus, \eqref{un.to.u0.XV} follows in view of Lemma \ref{Lem.S+}.

Next, we will show that $u_0\big|_{\Omega_V}$ is indeed a weak solution to the limit problem \eqref{eq2}, namely, 
\begin{align}\label{Sol.limit.1}
\int_{\Omega_V} \mathcal{A}(x,\nabla u_0) \cdot \nabla v \diff x =  \int_{\Omega_V} f(x,u_0)v \diff x  
+ \theta \int_{ \Omega_V}B(x,u_0)v\diff x,\quad \forall v\in C_c^\infty(\Omega_V).
\end{align}
To this end, let $v \in C_c^\infty(\Omega_V)$ be arbitrary. We claim that 
\begin{align} \label{cc.1}
\lim_{ n \to \infty}\int_{\Omega_V} \mathcal{A}\left(x,\nabla u_n\right)\cdot \nabla v\,\diff x =  \int_{\Omega_V} \mathcal{A}\left(x,\nabla u_0\right)\cdot \nabla v\,\diff x.
\end{align}
Indeed, by \eqref{un.to.u0.XV} we have
\begin{equation*}\label{un.to.u0.H}
\lim_{n\to\infty}\int_{\RN} \mathcal{H}\left(x,\left|\nabla u_n - \nabla u_0\right|\right) \diff x =0.
\end{equation*}
Therefore, in view of \cite[Theorem 4.9]{Bre.Book} it holds that
\begin{align}\label{grad.un.to.grad.u0}
\nabla u_n \to \nabla u_0 \quad \text{a.e. in}~ \RN
\end{align}
and 
\begin{equation}\label{grad.un.p}
|\nabla u_n|^{p(x)}+a(x)|\nabla u_n|^{q(x)}\leq h(x) \quad\text{for a.a.} \ x\in\R^N\ \text{and all } n\in\N ,
\end{equation}
with some $h\in L^1(\R^N)$. Utilizing these facts, we apply the Lebesgue dominated convergence theorem to obtain
\begin{align}\label{Jn.1}
\lim_{ n \to \infty} \int_{ \Omega_V} \left||\nabla u_n|^{p(x)-2}\nabla u_n - |\nabla u_0|^{p(x)-2}\nabla u_0\right|^{p'(x)}\diff x = 0
\end{align}
and 
\begin{align}\label{Jn.2}
\lim_{ n \to \infty} \int_{ \Omega_V} a(x)\left||\nabla u_n|^{q(x)-2}\nabla u_n - |\nabla u_0|^{q(x)-2}\nabla u_0\right|^{q'(x)}\diff x = 0.
\end{align}
Furthermore, by Proposition~\ref{prop.Holder} we have
\begin{align*}
\int_{\Omega_V} &a(x)\left||\nabla u_n|^{q(x)-2}\nabla u_n - |\nabla u_0|^{q(x)-2}\nabla u_0\right||\nabla v| \diff x\notag \\
& \leq 2 \left\|a^{\frac{1}{q'}}\left||\nabla u_n|^{q-2}\nabla u_n - |\nabla u_0|^{q-2}\nabla u_0\right| \right\|_{L^{q'(\cdot)}(\Omega_V)} \left\|a^{\frac{1}{q}}|\nabla v|\right\|_{L^{q(\cdot)}(\Omega_V)}, \quad \forall n \in \N.
\end{align*}
By Proposition~\ref{prop.nor-mod.D} $(v)$, it follows from the last estimate and \eqref{Jn.2} that
\begin{align*}
\lim_{ n \to \infty}	\int_{\Omega_V} a(x) \left[|\nabla u_n|^{q(x)-2}\nabla u_n - |\nabla u_0|^{q(x)-2}\nabla u_0\right] \cdot \nabla v\, \diff x = 0.
\end{align*}
In the same manner, we obtain
\begin{align*}
\lim_{ n \to \infty}	\int_{\Omega_V} \left[|\nabla u_n|^{p(x)-2}\nabla u_n - |\nabla u_0|^{p(x)-2}\nabla u_0\right] \cdot \nabla v\, \diff x = 0.
\end{align*}
Combining the last two limits yields \eqref{cc.1}. 
Similarly, we obtain
\begin{align}
\lim_{ n \to \infty}\int_{ \Omega_V} B(x,u_n)v \diff x  = \int_{\Omega_V} B(x,u_0)v \diff x \label{I1}
\end{align}
and
\begin{align}
\lim_{ n \to \infty}	\int_{  \Omega_V} f(x,u_n)v \diff x = \int_{  \Omega_V} f(x,u_0)v \diff x. \label{I2} 
\end{align}
Using \eqref{u_n-cn} and the fact that $\operatorname{supp} (v)\subset \Omega_V$, we have
\begin{equation} \label{un-weak-sol-eq1}
\int_{\Omega_V} \mathcal{A}(x,\nabla u_n)  \cdot \nabla v \diff x 
=  \int_{\Omega_V} f(x,u_n)v \diff x+ 
\theta \int_{ \Omega_V}B(x,u_n)v \diff x.
\end{equation}
Passing to the limit as $n \to \infty$  in \eqref{un-weak-sol-eq1}, using \eqref{cc.1}, \eqref{I1} and \eqref{I2}, we derive \eqref{Sol.limit.1}.

Finally, we will show $u_0\ne 0$. To this end, we first note that by \eqref{Xl.emb.B} and \eqref{XV.comp.emd.LF} again we can take $C_3>1$ (independent of $\lambda_n$) such that
\begin{equation}\label{4norms'}
\max\{\|u\|_\mathcal{B},\|u\|_{F}\} \leq C_3 \|u\|\leq C_3 \|u\|_{\lambda_n},\quad \forall u \in X_V,\ \forall n\in\N.
\end{equation}
For each $n \in \mathbb{N}$, it follows from \eqref{u_n-cn} that
\begin{equation}\label{concen.1}
\int_{ \mathbb{R}^N} \mathcal{H}(x,|\nabla u_n|) \diff x +  \int_{ \mathbb{R}^N} \lambda_nV(x) \mathcal{H}(x,|u_n|)\diff x = \int_{ \mathbb{R}^N} F(x,|u_n|)\diff x + \theta \int_{ \mathbb{R}^N} \mathcal{B}(x,|u_n|)\diff x. 
\end{equation}
If $\|u_n\|_{\lambda_n}\leq 1$, then by \eqref{m-n} and \eqref{concen.1} it holds that
\begin{equation*}
\|u_n\|^{q^+}_{\lambda_n} \leq (1+\theta)C_3^{s^+} \|u_n\|^{\ell_1^-}_{\lambda_n}.
\end{equation*}
From this and the fact that $u_n\ne 0$ we obtain that 
\begin{equation}\label{PT4.2-1}
	\|u_n\|_{\lambda_n} \geq \round{(1+\theta)C_3^{s^+}}^{\frac{1}{q^+-\ell_1^-}}.
\end{equation}
If  $\|u_n\|_{\lambda_n}>1$, then \eqref{PT4.2-1} holds automatically.

From the above analysis, we find $C_4\in (0,1)$ independent of $\lambda_n$ such that 
$$\|u_n\|_{\lambda_n}\geq C_4,\quad \forall n\in\N.$$
This and \eqref{concen.1} imply
\begin{align}\label{cc.2}
\int_{ \mathbb{R}^N} F(x,|u_n|)\diff x + \theta\int_{ \mathbb{R}^N} \mathcal{B}(x,|u_n|)\diff x \geq \min \left\{\|u_n\|^{p^-}_{\lambda_n},\|u_n\|^{q^+}_{\lambda_n}\right\} 
\geq C_4^{q^+}. 
\end{align} 
Passing to the limit as $n\to\infty$ in the last estimate, using \eqref{un.to.u0.strong.B1} and \eqref{un-LF} we arrive at
\begin{equation*}
\int_{ \mathbb{R}^N} F(x,|u_0|) \diff x + \theta\int_{ \mathbb{R}^N} \mathcal{B}(x,|u_0|) \diff x
\geq C_4^{q^+}>0.
\end{equation*}
From this we obtain $u_0 \not \equiv 0$. The proof is complete.
\end{proof}

\section*{Acknowledgment}
Ky Ho was supported by the University of Economics Ho Chi Minh City (UEH), Vietnam.


\end{document}